\documentclass[11pt,a4paper]{amsart}
\usepackage{color}
\usepackage{pstricks-add}\usepackage{tikz}\usepackage{pst-func}
\usepackage[all,cmtip]{xy}
\usepackage[USenglish,british,american,australian]{babel}
\usepackage[scale=0.75]{geometry}

\usepackage{amsfonts, amsmath, latexsym, amsthm, amssymb}
\usepackage{latexsym}
\usepackage[square,sort,comma,numbers]{natbib}
\usepackage{pgf,tikz}
\usepackage{mathrsfs}
\usetikzlibrary{arrows}

\newtheorem{MainTeo}{Main Theorem}
\newtheorem{Teo}{Theorem}
\newtheorem{Lema}[Teo]{Lemma}
\newtheorem{Prop}[Teo]{Proposition}
\newtheorem{Cor}[Teo]{Corollary}

\newtheorem{Remark}{Remark}

\newcommand{\calm}{{\mathcal M}}

\newcommand{\calo}{{\mathcal O}}

\newcommand{\calb}{{\mathcal B}}
\newcommand{\calr}{{\mathcal R}}
\newcommand{\cali}{{\mathcal I}}
\newcommand{\h}{\mathrm{H}}
\newcommand{\into}{\hookrightarrow}
\newcommand{\onto}{\twoheadrightarrow}
\DeclareMathOperator{\s}{\mathcal{S}}

\DeclareMathOperator{\Hom}{Hom}
\DeclareMathOperator{\mult}{\mathrm{mult}}

\DeclareMathOperator{\coker}{coker}

\newcommand{\sing}{\operatorname{Sing}}
\newcommand{\PP}{\mathbb{P}^3}

\newcommand{\lhom}{{\mathcal H}{\it om}}
\newcommand{\lext}{{\mathcal E}{\it xt}}
\newcommand{\calt}{{\mathcal T}}

\DeclareMathOperator{\ext}{Ext}
\DeclareMathOperator{\rk}{{rk}}
\DeclareMathOperator{\op3}{\mathcal{O}_{\mathbb{P}^3}}
\DeclareMathOperator{\supp}{\mathrm{Supp}}

\usepackage[
pdfauthor={Charles}
pdftitle={Report I}
pdfdisplaydoctitle=true,
pdfstartview={FitH},
bookmarks=false,
]{hyperref}

\title[Moduli space of rank 2 sheaves of odd degree]{Irreducible components of the moduli space of rank 2 sheaves of odd determinant on the projective space}

\author[C. Almeida]{Charles Almeida}
\address{IME-USP \\
Departamento de Matem\'atica \\ Rua do Mat\~ao, 1010 \\
05508-090 S\~ao Paulo-SP, Brazil}
\email{charlesalmd@ime.usp.br}

\author[M. Jardim]{Marcos Jardim}
\address{IMECC - UNICAMP \\
Departamento de Matem\'atica \\ Rua S\'ergio Buarque de 
Holanda, 651 \\
13083-970 Campinas-SP, Brazil}
\email{jardim@ime.unicamp.br}

\author[A. S. Tikhomirov]{Alexander S. Tikhomirov}
\address{Department of Mathematics\\
National Research University 
Higher School of Economics\\
6 Usacheva Street\\ 
119041 Moscow, Russian Federation}
\email{astikhomirov@mail.ru}
\begin{document}

\begin{abstract}
We describe new irreducible components of the moduli space of 
rank $2$ semistable torsion free sheaves on the 
three-dimensional projective space whose generic point 
corresponds to non-locally free sheaves whose singular locus
is either 0-dimensional or consists of a line plus disjoint 
points. 
In particular, we prove that the moduli spaces of semistable 
sheaves with Chern classes $(c_1,c_2,c_3)=(-1,2n,0)$ and 
$(c_1,c_2,c_3)=(0,n,0)$ always contain at least one rational 
irreducible component. As an application, 
we prove that the number of such components grows as the 
second Chern class grows, and compute the exact number of 
irreducible components of the moduli spaces of rank 2 
semistable torsion free sheaves with Chern classes 
$(c_1,c_2,c_3)=(-1,2,m)$ for 
all possible values for $m$; all components turn out to be 
rational. Furthermore, we also prove that these moduli spaces 
are connected, showing that some of sheaves here considered 
are smoothable.
\end{abstract}

\maketitle
\tableofcontents 

\section{Introduction}

Following the proof of existence of a projective moduli 
scheme parametrizing S-equivalence classes of semistable 
sheaves on a projective variety by Maruyama \cite{Maruyama}, 
the study of the geometry of such moduli spaces has been a 
central topic of research within algebraic geometry. Although 
a lot is known for curves and surfaces, general results for 
three dimensional varieties are still lacking. In fact, 
moduli spaces of sheaves on 3-folds turn out to be quite 
complicated spaces (as it is illustrated by Vakil's Murphy's 
law \cite{V}), particularly with several irreducible 
components of various dimensions.

The goal of this paper is to advance on the study of the 
moduli space of semistable rank 2 sheaves on $\PP$ with fixed 
Chern classes $(c_1,c_2,c_3)$, which we will denote by 
$\calm(c_1,c_2,c_3)$. Additionally, we will also consider the 
open subset consisting of stable reflexive sheaves, denoted 
by $\calr(c_1,c_2,c_3)$; when $c_3=0$, this is actually the 
moduli space of stable locally free sheaves, and this will be 
denoted by $\calb(c_1,c_2)$. Questions on the geometry of 
such spaces, such as connectedness, or the number of 
irreducible components, seem to be less explored if compared 
to the study of the geometry of the Hilbert schemes of curves 
in the projective $3$-space for instance; some results for 
Hilbert schemes of curves can be found in \cite{H1966,KO2015,
K2017,N1997,NS2001}.

A rich literature on these moduli spaces was produced, 
especially in the 1980's and 1990's, studying 
$\calr(c_1,c_2,c_3)$ and $\calb(c_1,c_2)$ for specific values 
of the Chern classes. For instance, the geometry of 
$\calb(0,c_2)$ and $\calb(-1,c_2)$ is completely understood 
for $c_2$ up to $5$, see \cite{Barth1,H1978,C1983,ES1981,
AJTT2017} for $c_1=0$, and \cite{Hart2,Banica}  for $c_1=-1$. 
In addition, Ein characterized a infinite series of 
irreducible components of $\calb(c_1,c_2)$ and proved that 
the number of irreducible components of $\mathcal{B}(c_1,c_2
)$ goes to infinity as the $c_2$ goes to infinity \cite{Ein}. 

Regarding reflexive sheaves, $\calr(c_1,c_2,c_3)$ is known 
for $c_2\le3$ and all possible values for $c_3$, see 
\cite{Chang} and the references therein. Some extremal values 
are also known, namely, $\calr(-1, c_2, c_2^2)$ was studied 
by Hartshorne in \cite{Harshorne-Reflexive}, Chang described  
$\calr(0, c_2, c_2^2 -2c_2 +4)$ in \cite{Chang2}, while 
Mir\'o-Roig studied  $\calr(-1;c_2;c_2^2-2c_2 +4)$ in 
\cite{Maria-Miro2}, and the moduli spaces $\calr(-1, c_2, 
c_2^2 - 2rc_2 +2r(r+1))$  for $1 \leq r \leq (-1 + \sqrt{4c_2 
-7})\slash 2$, and $c_2$ greater than $5$, and $\calr(-1, 
c_2, c_2^2 -2(r-1)c_2 )$ for $c_2$ greater than $8$ in 
\cite{Maria-Miro}.

Even less is known for torsion free sheaves. Okonek and 
Spindler proved in \cite{OS1985} that 
$\calm(0,c_2,c_2^2-c_2+2)$ and $\calm(-1,c_2,c_2^2)$ are 
irreducible for $c_2 \geq 6$. For small values of $c_2$, 
Mir\'o-Roig and Trautmann proved in \cite{Miro-Trautmann} 
that $\calm(0,2,4)$ is irreducible, while Le Potier showed in 
\cite[Chapter 7]{LeP1993} that $\calm(0,2,0)$ has exactly $3$ 
irreducible rational components; more recently, it was shown
in \cite{JMT} that $\calm(0,2,0)$ is connected. Trautmann has 
also argued that $\calm(0,2,2)$ has exactly $2$ irreducible 
components \cite{Tr}. More recently, Schmidt proved in 
\cite{S2018}, that $\calm(0,c_2,c_2^2-c_2+2)$ and 
$\calm(-1,c_2,c_2^2)$ are irreducible for any $c_2 \geq 0$, 
using methods different from the ones employed by Okonek and Spindler and by Mir\'o-Roig and Trautmann. 

$\calm(0,c_2,0)$ for $c_2\ge2$ was studied in \cite{JMT}, 
where new infinite series of irreducible components are 
described. The starting point is the identification of three 
different types of torsion free sheaves; more precisely, let 
$E$ be a torsion free sheaf on $\PP$, and set $Q_E := E^{\vee 
\vee} \slash E$, which we assume to be nontrivial; we have 
the following fundamental sequence
\begin{equation}\label{fundamental}
0 \to E \to E^{\vee \vee} \to Q_E \to 0
\end{equation}
and say that $E$ has
\begin{itemize}
\item \textit{0-dimensional singularities} if $\dim Q_E =0$;
\item \textit{1-dimensional singularities} if $Q_E$ has pure 
dimension 1;
\item \textit{mixed singularities} if $\dim Q_E =1$, but 
$Q_E$ is not pure.
\end{itemize}

With this definition in mind, a systematic way of producing 
examples of irreducible components of $\calm(0,c_2,0)$ whose 
generic point corresponds to a torsion free sheaf with 
$0$-dimensional and $1$-dimensional singularities is given in 
\cite{JMT}. Furthermore, the third author and Ivanov 
\cite{IT} constructed irreducible components of 
$\calm(0,3,0)$ whose generic point corresponds to a torsion 
free sheaf with mixed singularities. Additionally, in a 
recent paper \cite{I2019}, Ivanov proved that $\calm(0,3,0)$ 
has at least 11 irreducible components. 

Our first goal in this paper is to generalize the results 
presented in \cite{IT,JMT}, and show how to produce 
irreducible components of $\calm(c_1,c_2,c_3)$, for values of 
$c_1$, $c_2$ and $c_3$ also including cases with $c_1=-1$ and 
$c_3 \neq 0$, for sheaves with $0$-dimensional, 
$1$-dimensional, and mixed singularities. More precisely, we 
prove the following two statements.

\begin{MainTeo}\label{main1}
For each $e \in \{-1,0\}$, let $n$ and $m$ be positive 
integers such that $en \equiv m (\mathrm{mod~2})$. Let 
$\mathcal{R}^{*}$ be a nonsingular, irreducible component of 
$\mathcal{R}(e,n,m)$ of expected dimension $8n-3+2e$.
\begin{itemize}
\item[(i)] For each $l\ge1$, there exists an irreducible component 
$$ \mathrm{T}(e,n,m,l)\subset \mathcal{M}(e,n,m-2l) $$
of dimension $8n-3+2e+4l$ whose generic sheaf $[E]$ satisfies 
$[E^{\vee\vee}]\in\mathcal{R}^{*}$ and 
$\mathrm{length}(Q_E)=l$.
\item[(ii)] For each $r\ge2$ and $s\ge1$ such that $2r+2s\le 
m+e+2$, or $r=1$ and $s=0$ when $-e=n=m=1$, there exists an 
irreducible component 
$$ \mathrm{X}(e,n,m,r,s)\subset\calm(e,n+1,m+2+c_1-2r-2s)$$
of dimension of dimension $8n+4s+2r+2+e$, whose generic sheaf 
$[E]$ satisfies $[E^{\vee\vee}]\in\mathcal{R}^{*}$ and $Q_E$ 
is supported on a line plus $s$ points.
\end{itemize}
\end{MainTeo}

The case $e=0$ of the first part of the previous theorem is 
just \cite[Theorem 7]{JMT}; we prove here the case $e=-1$. 
The second part is a generalization of \cite[Theorem 3]{IT}, 
which covers the cases $e=0$, $n=2$, $m=2,4$. 

Our second goal in this paper concerns the problem of 
rationality of irreducible components of the moduli spaces 
$\calm(e,n,m)$.
The study of this problem for the moduli components of 
locally free sheaves, which are contained in $\calm(-1,2n,0)$ 
and $\calm(0,n,0),\ n\ge1$, dates back to late 70-ies and 
early 1980-ies. The rationality of these
moduli components was proved for $n\le 3$ in case $e=0$
\cite{Barth1,H1978,ES1981,Hart2}. The first infinite series of
rational moduli components were constructed and studied in 
\cite{BH,ES1981,V1,V2}. Recently, A. Kytmanov, A. Tikhomirov 
and S. Tikhomirov \cite{KTT} showed that there is a large 
infinite series of rational moduli components of locally free 
sheaves from $\calm(-1,2n,0)$ and $\calm(0,n,0)$ which 
includes the above mentioned series. These are the so-called 
Ein components which were first found and studied by 
A. P. Rao \cite{Rao} and L. Ein \cite{Ein}. However, it is 
still an open question whether these components exist for every 
$n$ sufficiently (there are gaps for some small values of $n$,  
see \cite{KTT} for details). One of the central results of our 
paper states that, for any $n\ge1$ there exist
rational irreducible components of $\calm(-1,2n,0)$ and of
$\calm(0,n,0)$.  The precise statement
is given by the following theorem. 

\begin{MainTeo}\label{main2}
~~\newline 
\begin{itemize}
\item[(i)] For any $n\ge1$, the scheme $\calm(-1,2n,0)$ contains 
at least one rational, generically reduced, irreducible 
component with generic sheaf having 0-dimen-sional 
singularities. For any $n\ge3$, $\calm(-1,2n,0)$ 
contains at least one rational, generically reduced, 
irreducible component with generic sheaf having purely 
1-dimensional singularities, respectively, at least 
$2(n^2-n-1)$ rational, generically reduced, irreducible 
components with generic sheaves having singularities of 
mixed dimension.

\item[(ii)] For any $n\ge2$, the scheme $\calm(0,n,0)$ contains at least one rational, generically reduced, irreducible component 
with generic sheaf having 0-dimensional singularities. For 
any $n\ge3$, the scheme $\calm(0,n,0)$ contains at least one
rational, generically reduced, irreducible component with 
generic sheaf having purely 1-dimensional singularities. For 
any $n\ge4$, the scheme $\calm(0,n,0)$ contains at least 
$\frac{n(n-3)}{2}$ rational, generically reduced, irreducible
components with generic sheaves having singularities of 
mixed dimension.
\end{itemize}
\end{MainTeo}

In addition, we also show that $\calm(e,n,m)$ has rational irreducible components for $e=-1,0$ and $n,m$ varying in a wide range (see Theorem \ref{Tmain2} for the precise statement).

The proof of this theorem is based on the above 
mentioned results of Chang \cite{Chang2}, Mir\'o-Roig 
\cite{Maria-Miro}, Okonek--Spindler \cite{OS1985} and
Schmidt \cite{S2018} on reflexive sheaves, and uses 
elementary transformations of reflexive
sheaves along finite sets of points.

We give two applications of our constructions. First, we 
prove that the number of irreducible components of 
$\calm(-1,n,0)$ whose generic point corresponds to a sheaf 
with mixed singularities grows as $n$ grows, see Theorem 
\ref{einmixed} below. Second, we provide a full description 
of the irreducible components of $\calm(-1,2,m)$.

\begin{MainTeo}\label{main3}
The moduli spaces $\calm(-1,2,m)$ are connected and
\begin{itemize}
\item[(i)] $\calm(-1,2,4)$ is irreducible and rational of 
dimension 11;
\item[(ii)] $\calm(-1,2,2)$ is connected and has exactly $2$ 
irreducible rational components of dimensions 11, and 15;
\item[(iii)] $\calm(-1,2,0)$ is connected and has exactly $4$ 
irreducible rational components of dimensions 11, 11, 15, and 
19. 
\end{itemize}
\end{MainTeo}

Note that the rationality of all of these components of 
follows directly from Main Theorem \ref{main2}. 

\begin{figure}[ht] \label{figure} 
\includegraphics[scale=1]{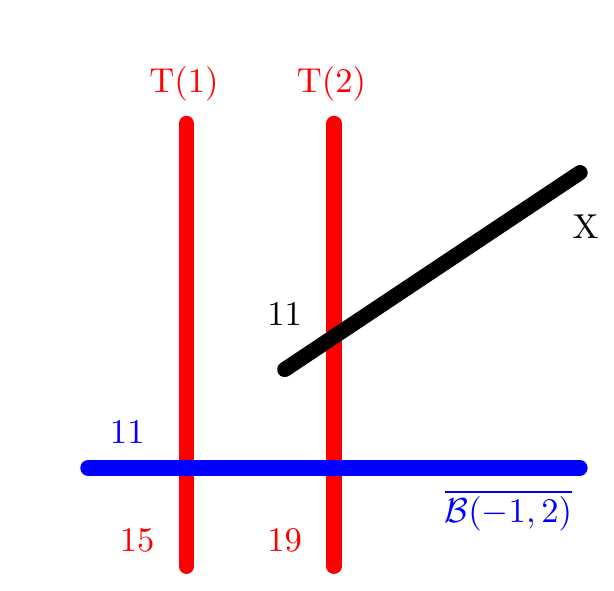}
\caption{This is a representation of the geography of the 
moduli space $\calm(-1,2,0)$. Each segment represents one of 
the irreducible components of $\calm(-1,2,0)$ and it is 
accompanied by its dimension. The blue component is the 
closure of $\calb(-1,2)$. The red components are of the type 
described in the first part of Main Theorem 1, namely 
$\mathrm{T}(l)=\mathrm{T}(-1,2,2l,l)$. The black component is 
of the type described in the second part of Main Theorem 1, 
namely $\mathrm{X}=\mathrm{X}(-1,1,1,1,0)$. The intersection 
of lines indicate when the corresponding components 
intersect.}
\end{figure}

We emphasize that proving that these moduli spaces are 
connected is quite relevant, since it is not known whether 
moduli spaces of rank 2 sheaves are in general connected, as 
it is the case for Hilbert schemes. In addition, we also 
provide very concrete descriptions of the generic points in 
each irreducible component; for a more detailed statement, 
see Theorems \ref{M(-1,2,2)} and \ref{M(-1,2,0)} for the 
cases $c_3=2$ and $c_3=0$, respectively.  A representation of 
the geography of $\calm(-1,2,0)$ is presented in Figure 
\ref{figure}, showing how the various irreducible components 
intersect one another.

Another important aspect of the proof of the connectedness 
part of Main Theorem \ref{main3} is that we are implicitly 
showing that some of the sheaves presented in Main Theorem 
\ref{main1} are smoothable. To be more precise, a semistable 
non locally free sheaf with $c_3=0$ is said to be 
\emph{smoothable} if it can be deformed into a stable locally 
free sheaf, that is, if it lies in the closure of an 
irreducible component of $\calb(c_1,c_2)$ within 
$\calm(c_1,c_2,0)$. In the observations following the proof 
of Theorem \ref{connected2} we provide certain sufficient 
conditions for smoothability of sheaves in $\calm(-1,2,0)$.

\bigskip

The paper is organized as follows. In Section \ref{First 
Computations} we build up some basic techniques, and 
preliminary results. We compute the dimensions of the Ext 
groups of torsion free sheaves in terms of their Chern 
classes, and use it in Section \ref{New Irreducible 
Components} in order to produce the examples of irreducible 
components of the moduli space of torsion free sheaves, and 
to prove Main Theorem \ref{main1}. These results are then 
explored in end of Section \ref{New Irreducible Components} 
to prove that the number of irreducible components of 
$\calm(c_1,c_2,0)$ whose generic point correspond to a sheaf 
with mixed singularities goes to infinity as $c_2$ goes to 
infinity, thus providing our first application. Main Theorem
\ref{main2} is proved in Section \ref{Rational Components}.

The remainder of the paper is occupied with the proof of Main
Theorem \ref{main3}.  
The irreducibility of $\calm(-1,2,4)$ is established in 
Section \ref{irreducible of M(-1,2,4)}. After further 
technical results in Sections \ref{irreducible of M(2)} and 
\ref{descr of X} regarding the families of sheaves introduced 
in Main Theorem \ref{main1}, we dedicate Sections 
\ref{irreducible of M(-1,2,2)} and \ref{irreducible of 
M(-1,2,0)} to describing all irreducible components of 
$\calm(-1,2,m)$ for $c_3=2$ and $c_3=0$, respectively. The 
connectedness of $\calm(-1,2,m)$ is finally established in 
Section \ref{Connectedness of M(2)}.

\bigskip

\noindent{\bf Acknowledgements.}
This work started with discussions among the authors during a visit to SISSA in May 2016; we thank Ugo Bruzzo and SISSA for its support and hospitality. CA was supported by the FAPESP grants number 2014/08306-4 and 2016/14376-0; part of this work was made when he was visiting the University of Barcelona, and he is grateful for its warm hospitality, also thanking Rosa Maria Mir\'o Roig for the several useful discussions on this topic. MJ is partially supported by the CNPq grant number 302889/2018-3 and the FAPESP Thematic Project 2018/21391-1; part of this work was done during a visit to the University of Edinburgh, and later completed during a visit to the Simons Center for Geometry and Physics; MJ is grateful for the hospitality of both institutions. AT was supported by funding within the framework of the State Maintenance Program for the Leading Universities of the Russian Federation "5-100". AT also thanks the Max Planck Institute for Mathematics in Bonn for hospitality, where this work was partially done during the winter of 2017. This work was also partially funded by CAPES - Finance Code 001.

\section{First computations}\label{First Computations}

In order to study the moduli spaces of torsion free sheaves 
on $\PP$ we will need an explicit method to compute $\dim 
\ext^1(E,E)$, which gives us the dimension of the tangent 
space of the isomorphism class of a stable torsion free sheaf 
$E$ as a point the moduli space. Our main goal in this 
section is to prove the following theorem. 

\begin{Teo}\label{dimext1}
Let $E$ be a stable rank 2 torsion free sheaf on $\PP$ with
$e:=c_1(E)\in\{-1,0\}$. Then
$$\dim\ext^1(E,E)-\dim\ext^2(E,E)=8c_2(E)-3-2c_1(E)^2=8c_2(E)-
3-2e.$$
\end{Teo}

Note that this result generalizes \cite[Lemma 5d)]{JMT} and 
\cite[Lemma 10]{JMT}, which establish the formula above for 
stable rank 2 torsion free sheaves with 0- and 1-dimensional 
singularities, respectively, in the case $c_1(E)=0$. The 
proofs for sheaves with 0- and 1-dimensional singularities 
with arbitrary $c_1$ are quite similar to the one in 
\cite{JMT}; therefore, we only include here the proof for 
sheaves with mixed singularities. 
 
\noindent Theorem \ref{dimext1} together with the deformation theory 
yields
\begin{Cor}\label{dim calm}
Any irreducible component of the moduli space 
$\calm(e,c_2,c_3)$ has dimension at least $8c_2-3+2e$.
\end{Cor}

\begin{Lema}\label{exts general}
If $E$ is a torsion free sheaf on $\PP$, then:
\begin{itemize}
\item[(i)] $\ext^1(E,E) = \h^1(\lhom(E,E))\oplus\ker 
d^{01}_2$;
\item[(ii)] $\ext^2(E,E) = \ker d^{02}_3 \oplus \ker d^{11}_2 
\oplus \coker d^{01}_2$;
\item[(iii)] $\ext^3(E,E) = \coker d^{02}_3$.
\end{itemize}
Here, $d^{pq}_j$ are the differentials in the j-th page of 
the spectral sequence for local to global ext's 
$E^{pq}_2:=\h^p(\lext^q(E,E))$. In particular, we have
$$ \sum_{j=0}^{3}(-1)^j\dim \ext^j(E,E) = \chi(\lhom(E,E)) - 
\chi(\lext^1(E,E)) + h^0(\lext^2(E,E)). $$
\end{Lema}
\begin{proof}
The first part is a standard calculation with the spectral 
sequence $E^{pq}_2:=\h^p(\lext^q(E,E))$, which converges in 
its forth page, because the spectral maps vanish. Note that 
$\h^p(\lext^q(E,E))=0$ for $p\ge2$ 
and $q\ge1$, since $\dim\lext^q(E,E)\le1$ for $q\ge1$. 
Furthermore, applying the functor $\lhom(\cdot,E)$ to the 
fundamental sequence \eqref{fundamental}, we get an 
epimorphism $\lext^3(E^{\vee\vee},E)\onto\lext^3(E,E)$ and 
the isomorphism $\lext^2(E,E)\simeq\lext^3(Q_E,E)$; however, 
the sheaf on the left vanishes because $E^{\vee\vee}$ is 
reflexive, so $\lext^3(E,E)=0$ as well. Finally, we also 
check that $\dim\lext^2(E,E)=0$; indeed, 
$E$ admits a resolution of the form
\begin{equation}\label{res-e}
0 \to L_2 \to L_1 \to L_0 \to E \to 0,
\end{equation}
where $L_k$ are locally free sheaves; we then get an 
epimorphism 
$$ \lext^3(Q_E,\op3)\otimes L_0 \onto \lext^3(Q_E,E), $$
which implies that $\dim\lext^3(Q_E,E)=0$ since 
$\dim\lext^3(Q_E,\op3)=0$.\\
The second claim is an immediate consequence of the first, 
since $\dim\lext^2(E,E)=0$.
\end{proof}

Assuming that $E$ is $\mu$-semistable provides a useful 
simplification of the previous general result.

\begin{Lema}\label{Extisomixed}
If $E$ be a $\mu$-semistable torsion free sheaf on $\PP$, 
then:
\begin{itemize}
\item[(i)] $\ext^1(E,E) = \h^1(\lhom(E,E))\oplus\ker
d^{01}_2$;
\item[(ii)]  $\ext^2(E,E) = \h^0(\lext^2(E,E)) \oplus 
\h^1(\lext^1(E,E)) \oplus \coker d^{01}_2$;
\item[(iii)]  $\ext^3(E,E)=0$.
\end{itemize}
Here, $d^{01}_2$ is the spectral sequence differential
$d^{01}_2: \h^0(\lext^1(E,E)) \to \h^2(\lhom(E,E))$.
\end{Lema}
\begin{proof}
The last item follows from  Serre duality, we have
$$ \ext^3(E,E) \simeq \Hom(E,E(-4))^* = 0, $$
with the vanishing given by $\mu$-semistability. 
In addition, we argue that $\mu$-semistability also implies 
that $\h^3(\lhom(E,E))=0$. Indeed, applying the functors 
$\lhom(\cdot,E)$ and $\lhom(E^{\vee\vee},\cdot)$ to the 
fundamental sequences \eqref{fundamental} we obtain, 
respectively,
$$ 0 \to \lhom(E^{\vee\vee},E) \to \lhom(E,E) \to 
\lext^1(Q_E,E) \to \cdots 
$$
and 
$$ 
0 \to \lhom(E^{\vee\vee},E) \to \lhom(E^{\vee\vee},
E^{\vee\vee})\to\lhom(E^{\vee\vee},Q_E)\to\cdots 
$$
In both sequences, the rightmost sheaf has dimension at most 
1, hence so does the cokernel of the leftmost monomorphism, 
and it follows that 
$$ \h^3(\lhom(E,E)) \simeq \h^3(\lhom(E^{\vee\vee},E)) \simeq 
\h^3(\lhom(E^{\vee\vee},E^{\vee\vee})). $$
However
$$ 
\h^3(\lhom(E^{\vee\vee},E^{\vee\vee}))=\ext^3(E^{\vee\vee},
E^{\vee\vee}) \simeq \Hom(E^{\vee\vee},E^{\vee\vee}(-4))^*= 
0; 
$$
the first equality follows from the spectral sequence for 
local to global ext's for $E^{\vee\vee}$, the isomorphism in 
the middle is given by Serre duality, and the vanishing is a 
consequence of the $\mu$-semistability of $E^{\vee\vee}$.

It follows that $d^{pq}_2=0$ except for $d^{01}_2$, while 
$d^{pq}_3=0$ for every $p$ and $q$. This means that 
$E^{pq}_2$ converges in its third page, providing the desired 
result.
\end{proof}

The following technical lemma will be helpful in our next 
argument.

\begin{Lema}\label{Technic}
Let $F$ be a torsion free sheaf. If $E$ is a subsheaf of $F$ 
for which the quotient sheaf $Z:=F/E$ is 0-dimensional, then 
\begin{equation}\label{sum}
\sum_{j=0}^3 (-1)^j\chi(\lext^j(Z,E)) + \sum_{j=0}^3 
(-1)^j\chi(\lext^j(F,Z)) = 0.
\end{equation}
\end{Lema}
\begin{proof}
Break a locally free resolution of $E$ as in \eqref{res-e} 
into two short exact sequences
$$ 0 \to L_2 \to L_1 \to K \to 0 ~~{\rm and}~~ 0 \to K \to 
L_0 \to E \to 0 . $$
Applying functor $\lhom(Z,-)$ and passing to Euler 
characteristic on the first sequence, we have:
\begin{align}\label{chi1}
\chi(\lext^2(Z,K)) - \chi(\lext^3(Z,K)) = 
\chi(\lext^3(Z,L_2)) - \chi(\lext^3(Z,L_1)) = 
\end{align}
\begin{align*}
=(\rk(L_2)-\rk(L_1))\chi(Z),
\end{align*}
since 
$\chi(\lext^3(Z,L_k))=\chi(\lext^3(Z,\mathcal{O}_{\PP})\otimes
L_k) = \rk(L_k)\cdot\chi(Z)$. Now, applying the functor 
$\lhom(Z,-)$ to the second exact sequence we obtain the 
isomorphism $\lext^1(Z,E)\simeq \lext^2(Z,K)$ and passing to 
the Euler characteristic we have
\begin{equation*}
\chi(\lext^2(Z,E)) - \chi(\lext^3(Z,E)) = \chi(\lext^3(Z,K)) 
- \chi(\lext^3(Z,L_0)).
\end{equation*}
Subtracting $\chi(\lext^1(Z,E))$ from the left hand side and 
$\chi(\lext^2(Z,K))$ from the right hand side, and then 
substituting for (\ref{chi1}) we have:
\begin{equation}\label{chiEE}
\sum_{j=0}^3 (-1)^j\chi(\lext^j(Z,E)) = 
(\rk(L_1)-\rk(L_2)-\rk(L_0))\cdot\chi(Z) = -\rk(E)\chi(Z) .
\end{equation}
Since $\dim\lext^j(Z,E)=0$, we have
$$ \chi(\lext^j(Z,E)) = h^0(\lext^j(Z,E)) = $$ 
$$ = \dim\ext^j(Z,E) 
\stackrel{\rm SD}{=} \dim\ext^{3-j}(E,Z) = 
\chi(\lext^{3-j}(E,Z)) , $$
where the supercript SD indicates the use of Serre duality. 
The formula \eqref{chiEE} applied to the sheaf $F$ then yields
$$ \sum_{j=0}^3 (-1)^j\chi(\lext^j(F,Z)) = \rk(F)\chi(Z) .$$
The fact that $\rk(F)=\rk(E)$ provides the desired identity.
\end{proof}

\begin{Lema}\label{Lemamixed}
Let $E$ be a rank 2 torsion free sheaf with mixed 
singularities. Then:
$$ \displaystyle \sum_{j=0}^{3}(-1)^j\dim \ext^j(E,E) = - 
8c_2(E) + 4 + 2c_1(E)^2. $$
\end{Lema}

\begin{proof}
Let $Z_E \into Q_E$ the maximal 0-dimensional subsheaf of 
$Q_E$, and set
$T_E := Q_E/Z_E$ to be the pure 1-dimensional quotient; we 
assume that both $Z_E$ and $T_E$ are nontrivial. Let $E'$ be 
the kernel of the composed epimorphism $E^{\vee\vee}\onto Q_E 
\onto T_E$; note that it also fits into the following short 
exact sequence
\begin{equation}\label{puremixed}
0 \to E \to E^{\prime} \to Z_E\to 0.
\end{equation}
Note that $c_1(E')=c_1(E)$ and $c_2(E')=c_2(E)$. In addition, 
$(E')^{\vee\vee}\simeq E^{\vee\vee}$, and $Q_{E'}\simeq T_E$, 
thus $E'$ is a torsion free sheaf with 1-dimensional 
singularities. It follows that $E'$ has homological dimension 
1 (that is $\lext^p(E',G)=0$ for $p\ge2$ and every coherent 
sheaf $G$), so the proof of \cite[Proposition 
3.4]{Harshorne-Reflexive} also applies for $E'$, and we 
conclude that
$$ \displaystyle \sum_{j=0}^{3}(-1)^j\dim \ext^j(E',E') = - 
8c_2(E) + 4 + 2c_1(E)^2. $$
Therefore, it is enough to prove that 
$$ \displaystyle \sum_{j=0}^{3}(-1)^j\dim \ext^j(E,E) = 
\sum_{j=0}^{3}(-1)^j\dim \ext^j(E',E'), $$
which, by Lemma \ref{exts general} is equivalent to show that
$$   \displaystyle \chi(\lhom(E,E)) - \chi(\lext^1(E,E)) + 
h^0(\lext^2(E,E)) = \chi(\lhom(E^{\prime},E^{\prime})) - 
\chi(\lext^1(E^{\prime},E^{\prime})). $$

\noindent To see this, note that applying the functor $\lhom(E^{\prime},-)$ to 
the sequence (\ref{puremixed}) we obtain:
\begin{align}
& \chi(\lhom(E^{\prime},E)) - 
\chi(\lhom(E^{\prime},E^{\prime})) +  
\chi(\lhom(E^{\prime},Z_E)) -  \nonumber \\
& \chi(\lext^1(E^{\prime},E)) + 
\chi(\lext^1(E^{\prime},E^{\prime}))-\chi(\lext^1(E^{\prime},Z
_E))=0. \nonumber
\end{align}
Next, applying the functor $\lhom(-,E)$ to the sequence 
(\ref{puremixed}) we have
\begin{align}
 &   \chi(\lhom(E^{\prime},E)) - \chi(\lhom(E,E)) + 
\chi(\lext^1(Z_E,E)) -   \nonumber \\
 & \chi(\lext^1(E^{\prime},E)) + 
 \chi(\lext^1(E,E))-\chi(\lext^2(Z_E,E))= 0. \nonumber
\end{align}
Taking the difference between these last two equations we 
obtain
$$ \chi(\lhom(E^{\prime},E^{\prime})) -  
\chi(\lext^1(E^{\prime},E^{\prime}))  =  \chi(\lhom(E,E)) -  
\chi(\lext^1(E,E)) +  $$
$$ - \chi(\lext^1(Z_E,E)) + \chi(\lext^2(Z_E,E)) + 
\chi(\lhom(E^{\prime},Z_E)) - \chi(\lext^1(E^{\prime},Z_E)) = 
$$
$$ 
\chi(\lhom(E,E))-\chi(\lext^1(E,E))+\chi(\lext^3(Z_E,E)), 
$$
with the second equality following from applying the formula 
established in Lemma \ref{Technic} to the sheaves $E$ and 
$E'$. Applying the functor $\lhom(-,E)$ to the sequences
$$ 
0 \to E' \to E^{\vee\vee} \to T_E \to 0 ~~{\rm and}~~
0 \to Z_E \to Q_E \to T_E \to 0 
$$
we conclude that $\lext^3(T_E,E)=0$ and 
$\lext^3(Q_E,E)\simeq\lext^3(Z_E,E)$. We already noticed in 
the proof of Lemma \ref{exts general} that 
$\lext^3(Q_E,E)\simeq\lext^2(E,E)$, thus 
$$ \chi(\lext^3(Z_E,E))=\chi(\lext^2(E,E))=h^0(\lext^2(E,E)), $$ 
as desired.
\end{proof}

Gathering the above results we are in position to prove the 
Theorem \ref{dimext1}.
\begin{proof}[Proof of Theorem \ref{dimext1}]
By Lemma \ref{Lemamixed}, it is enough to show that $\dim \Hom(E,E) = 1$ and $\ext^3(E,E) = 0$, but these follow 
easily from the stability of $E$.
\end{proof}

The following proposition will be a technical tool that will 
help us to compute explicitly the dimension of $\ext^1(E,E)$ 
for certain torsion free sheaves.

\begin{Prop}\label{important}
Let $F$ be a stable rank 2 reflexive sheaf on $\PP$, with 
$\dim \ext^2(F,F) = 0$. Let $Z$ be an artinian sheaf, and $T$ 
be a sheaf of pure dimension 1 such that $H^1(\lhom(F,T))=0$; 
set $Q :=Z \oplus T$ and assume also that 
$\sing(F)\cap\supp(Q)=\emptyset$. If $\varphi: F \to Q$ is an 
epimorphism, then, for $E := \ker\varphi$,
\begin{itemize}
\item[(i)] $E$ is a stable rank 2 torsion free sheaf;
\item[(ii)] $c_1(E) = c_1(F)$ and  $c_2(E) = c_2(F) + 
\mult(T)$, where $\mult(T)$ denotes the multiplicity of the 
sheaf $T$;
\item[(iii)] $\ext^2(E,E) = \h^0(\lext^3(Z,E)) \oplus 
\ext^3(T,E)$.
\end{itemize}
\end{Prop}
\begin{proof}
The items (i) and (ii) are straightforward calculations; we 
will prove (iii). 
First we will show that the spectral sequence map  
$$d_2^{01}: \h^0(\lext^1(E,E)) \to \h^2(\lhom(E,E))$$
\noindent is an epimorphism.
Consider the exact sequence:
\begin{equation}\label{mainmixed}
 0\to E \to F \to Q \to 0.
\end{equation}
Applying the functor $\lhom(F,-)$ to \eqref{mainmixed}, once 
$\coker\{\lhom(F,E)\to \lhom(E,E)\}$ is supported in 
dimension 1, we have 
\begin{equation}\label{toprow}
\h^2(\lhom(F,E)) \to \h^2(\lhom(F,F)) \to 0.
\end{equation}
Next apply $\Hom(F,-)$ in the sequence (\ref{mainmixed}), by 
hypothesis, $\ext^2(E,E) = 0$, then we have
$$\ext^1(F,Q) \to \ext^2(F,E) \to 0. $$
To see that $\ext^1(F,Q)$ vanishes, note that $\ext^p(F,Q) = 
0$ for $p=2,3$ because $F$ is reflexive. $\lext^1(F,Q) = 0$ 
because $\sing(F) \cap \supp(Q) = \emptyset$. 
In addition, 
$\h^p(\lhom(F,Q)) = 0$ for $p=2,3$ because $\dim Q = 1$. From 
the spectral sequence, $\ext^1(F,Q) = \h^1(\lhom(F,Q))$ which 
vanishes by hypothesis. Therefore $d_2^{01}: 
\h^0(\lext^1(F,E)) \to \h^2(\lhom(F,E))$ is surjective. Then 
we have
\begin{equation} \label{d012-diagram}
\xymatrix{ \h^0(\lext^1(F,E)) \ar[d]\ar[r]^{d^{01}_2} & 
\h^2(\lhom(F,E)) \ar[d] \\
\h^0(\lext^1(E,E)) \ar[r]^{d^{01}_2} & \h^2(\lhom(E,E)),
}\end{equation}
\noindent where the vertical arrow in the left is the natural 
map coming from the exact sequence (\ref{mainmixed}), and 
horizontal maps came from the spectral sequence. Since the 
top row map, and the right vertical map are surjective, we
have that the bottom map is surjective as we wanted.
Now, applying $\lhom(-,E)$ to the sequence (\ref{mainmixed}) 
we have
$$\lext^2(E,E) \simeq \lext^3(Q,E) \simeq \lext^3(Z,E) \oplus 
\lext^3(T,E).$$
Furthermore, there is an exact sequence
\begin{equation*}
\xymatrix{ \lext^1(F,E) \ar[r] & \lext^1(E,E) \ar[r]^{f}& 
\lext^2(Q,E) \ar[r]& 0,  }
\end{equation*}
where $\dim \ker f = 0$, since $\dim\lext^1(F,E)=0$. Thus
\begin{equation}\label{almost} 
\ext^2(E,E) = \h^0(\lext^3(Z,E)) \oplus \h^0(\lext^3(T,E)) 
\oplus \h^1(\lext^2(T,E)).
\end{equation}
\noindent Since $F$ is reflexive, from \cite[Proposition 
1.1.6]{HL}, we have  $\lext^p(T,F) = 0$ for $p=0,1$ and 
$\mathrm{codim}~\lext^p(T,F)\geq p$ for $p = 2,3$. 
Clearly, $\dim\lext^p(T,E) \le 1$ for $p>0$, while 
$\lhom(T,E) = 0$; using these facts, we obtain from the 
spectral sequence for $\ext^{\cdot}(T,E)$ that
\begin{equation}\label{ext3}
 \ext^3(T,E) = \h^0(\lext^3(T,E)) \oplus \h^1(\lext^2(T,E)).
\end{equation}
Putting together the equations (\ref{almost}) and 
(\ref{ext3})  we obtain item (iii).
\end{proof}

\begin{Remark}
Item (iii) of Proposition \ref{important} also holds when 
$T=0$ without assuming that $\ext^2(F,F)=0$, see the proof of 
Main Theorem \ref{main2}, starting in page \pageref{pf mthm 
2} below.
\end{Remark}

An important ingredient of the Proposition \ref{important} is 
a family of stable reflexive sheaves, that fills out an 
irreducible component of the moduli space, with the expected 
dimension. A priori, it is not clear why such family should 
exist. In \cite{JMT} the authors proved that, indeed, such 
families exists for infinitely many values of the second 
Chern class, provided that the first Chern class is even. 
Below we state a theorem that  shows that this happens also 
for sheaves with odd first Chern class. For simplicity of notation, we define
$$ G_{(a,b,c)}:=a\cdot\mathcal{O}_{\PP}(-3) \oplus b\cdot\mathcal{O}_{\PP}(-2) \oplus c\cdot\mathcal{O}_{\PP}(-1). $$

\begin{Teo}\label{thm newclass} 
For each triple $(a,b,c)$ of positive integers such that 
$3a+2b+c$ is odd, the family of  rank 2 reflexive 
sheaves $F$ obtained as the cokernel of the maps $\alpha$ 
below 
\begin{equation}\label{newclass}
0 \to a\cdot\mathcal{O}_{\PP}(-3) \oplus 
b\cdot\mathcal{O}_{\PP}(-2) \oplus 
c\cdot\mathcal{O}_{\PP}(-1) \stackrel{\alpha}{\longrightarrow}
(a+b+c+2)\cdot\mathcal{O}_{\PP} \to F(k) \to 0,
\end{equation}
where $k:=(3a+2b+c+1)/2$, fills out a nonsingular irreducible component 
$\mathcal{S}(a,b,c)$ of $\mathcal{R}(-1;n;m)$ of expected dimension $8n-5$, with $n$ and $m$ are given by the expressions:
\begin{equation}\label{m(a,b,c)}
\begin{split}
& n =  \frac{1}{4}(3a+2b+c+1)^2 + 3a + b, \\
& m(a,b,c) = 27{{a+2}\choose{3}} + 8{{b+2}\choose{3}} + 
{{c+2}\choose{3}} + 3(3a+2b+5)ab +\\
& + \frac{3}{2}(3a+c+4)ac + (2b+3c+3)bc + 6abc.
\end{split}
\end{equation}
More precisely, let $\widetilde{\s}(a,b,c)\subset 
\Hom\left(G_{(a,b,c)},(a+b+c+2)\cdot\mathcal{O}_{\PP}\right)$ 
be the open subset consisting of monomorphisms with 
0-dimensional degeneracy loci; then
$$ \s(a,b,c) = \widetilde{\s}(a,b,c) / 
((\mathrm{Aut}(G_{(a,b,c)})\times GL(a+b+c+2))/\mathbb{C}^*) 
. $$
\end{Teo} 
\begin{proof}
Let $a,b,c \in \mathbb{Z}$, such that $3a+2b+c$ is odd and non zero, and consider morphisms of the form
$$ \alpha: G_(a,b,c) \to (a+b+c+2)\cdot\mathcal{O}_{\PP}. $$
If the degeneracy locus
$$ \Delta(\alpha) := \{ x\in\PP ~|~ \alpha(x) ~~ \text{is not injective} \} $$
is 0-dimensional, then the cokernel of $\alpha$ is a rank 2 reflexive sheaf $F$ on $\PP$, which we twist by 
$k:=(3a+2b+c+1)/2$, so that $c_1(F)=-1$, and the exact sequence in display (\ref{newclass}) is satisfied.
The dimension of this family of rank 2 reflexive sheaves is given by
$$ \dim\Hom\left(G_{(a,b,c)}, (a+b+c+2)\cdot\mathcal{O}_{\PP}\right) - \dim\mathrm{Aut}\left(G_{(a,b,c)}\right) - (a+b+c+2)^2 + 1 = $$
$$8k^2 + 24a + 8b - 5 = 8c_2(F) -5. $$

Note that for such sheaf $F$ satisfies $h^0(F(-1)) = 0$, thus $F$ is always stable. It only remains for us to check that $\dim \ext^2(F,F) = 0$. This follows from applying the functor $\Hom(\cdot,F(k))$ to the exact sequence in display (\ref{newclass}), and observing that $\h^1(F(t))=0$ for every $t\in\mathbb{Z}$ and that $\h^2(F(k))=0$. Therefore  the family of sheaves given by (\ref{newclass}) provides a component of the moduli space of stable rank 2 reflexive sheaves on $\PP$.
\end{proof}

The case that deserves special attention is the case $a = b 
=0$ and $c = 1$, that give us $c_1(F) = -1$, $c_2(F) = c_3(F) 
= 1$. In \cite[Lemma 9.4]{Harshorne-Reflexive} is shown that 
every reflexive sheaf in $\mathcal{R}(-1,1,1)$ admits  a 
resolution of the form:
\begin{equation}\label{resR(-1,1,1)}
0  \to \mathcal{O}_{\PP}(-2) 
\stackrel{\alpha}{\longrightarrow}
3\cdot\mathcal{O}_{\PP}(-1) \to F \to 0.
\end{equation}

From this sequence we can easily deduce the splitting behaviour of a 
sheaf $F$ in $\calr(-1,1,1)$. Indeed, each one of the $3$ rows 
of the map $\alpha$ can be viewed as the equation of a 
hyperplane in $\mathbb{P}^3$, since $\alpha$ is injective, 
the hyperplane must intersect in exactly one point $p$, that 
coincides with the singularity of the sheaf $F$. Thus, if $l 
\subset \PP$ is a line, if $p \notin l$, then the restriction 
of $F$ on $l$, $F|_{l}$, is isomorphic to $\mathcal{O}_l(-1) 
\oplus \mathcal{O}_l$. On the other hand, if $p \in l$, from 
sequence (\ref{resR(-1,1,1)}), we have that $F|_{l} \simeq 
\mathcal{O}_{p} \oplus 2 \mathcal{O}_l(-1)$. Summarizing, we 
have:

\begin{equation}\label{splitingF}
F|_l=\begin{cases}
\mathcal{O}_l(-1) \oplus \mathcal{O}_l,& \text{if } p \notin 
l,\\
\mathcal{O}_{p} \oplus 2\cdot\mathcal{O}_l(-1),& \text{if } p 
\in l.
\end{cases}
\end{equation}

\begin{Remark}\label{R(-1,1,1)} From \cite[Example 4.2.1]
{Harshorne-Reflexive} it follows that $\mathcal{R}(-1,1,1)$ 
is irreducible non-singular and rational of dimension $3$. 
Moreover, there is an isomorphism 
$\mathcal{R}(-1,1,1)\xrightarrow{\simeq}\PP,\ [F]\mapsto
\sing(F)$, and every sheaf $[F]\in\mathcal{R}(-1,1,1)$
fits in an exact triple $0\to\op3(-2)\to3\cdot\op3(-1)\to 
F\to0$. This yields that $\ext^2(F,F)=0$. Also Theorem 
\ref{thm newclass} implies that $\s(0,0,1)= 
\mathcal{R}(-1,1,1)$. Besides, under the isomorphism 
$\mathcal{R}(-1,1,1)\simeq\PP$, the above exact triple 
globalizes to the exact triple over $\PP\times\PP$:
\begin{equation}\label{univ F}
0\to\op3(-2)\boxtimes\op3\xrightarrow{\alpha}\op3(-1)
\boxtimes T_{\PP}(-1)\to\mathbf{F}\to0,
\end{equation}
where $\mathbf{F}$ is the universal family of reflexive 
sheaves over $\calr(-1,1,1)$, the morphism $\alpha$ is the
composition $\op3(-2)\boxtimes\op3\xrightarrow{i(-1)
\boxtimes\mathrm{id}}4\cdot\op3(-1)\boxtimes\op3
\xrightarrow{\mathrm{id}\boxtimes\epsilon}\op3(-1)\boxtimes
T_{\PP}(-1)$, and $i,\epsilon$ are the morphisms in the 
Euler exact triple $0\to\op3(-1)\xrightarrow{i}4\cdot\op3
\xrightarrow{\epsilon}T_{\PP}(-1)\to0$.
\end{Remark}

\section{Sheaves with 0-dimensional and mixed 
singularities}\label{New Irreducible Components}

In \cite{JMT} the authors produced examples of irreducible 
components with $0$-dimen-sional singularities and pure 
$1$-dimensional singularities, in the moduli space of rank 2 
stable torsion free sheaves with first Chern class equals to 
0, and in \cite{IT} the authors proved the existence of 
irreducible components in $\calm(0,3,0)$ whose generic point 
is a sheaf with mixed singularities, with first Chern class 
equals to 0. The first natural question that arises is that 
if similar constructions can be made for sheaves with odd 
first Chern class, and if it is possible similar irreducible 
components for non zero third Chern class. 
    
We will explicit construct examples  of irreducible 
components of the moduli space of torsion free sheaves with 
mixed singularities. We refer the reader to \cite{IT} for 
some examples in $\calm(0,3,0)$. 

For the rest of this work, 
let $e \in \{-1,0\}$, and $n$, $m$ be two integers such 
that $en \equiv m (\mathrm{mod}~2)$. Let 
\begin{equation}\label{def R*}
\calr^*(e,n,m):=\{[F]\in\calr(e,n,m)\ |\ \ext^2(F,F)=0 \}.
\end{equation}
By semicontinuity, $\calr^{*}(e,n,m)$ is an open smooth 
subset of $\calr(e,n,m)$ such that, in view of Theorem 
\ref{dimext1} and Corollary \ref{dim calm}, 
\begin{equation}\label{dim R*}
\dim_{[F]}\calr^{*}(e,n,m)=\dim\ext^1(F,F)=8n-3+2e,\ \ \ \ \ 
\ \ \ [F]\in\calr^{*}(e,n,m).
\end{equation} 
(Here and below by the dimension $\dim_xX$ of a given scheme 
$X$ locally of finite type at a point $x\in X$ we mean the
maximum of dimensions of irreducible components of $X$ 
passing through the point $x$.)

Let $(\PP)^s_0$ be the open dense subset of $(\PP)^s$ 
consisting of disjoint unions of $s$ distinct points in 
$\PP$. For any closed point $[F]\in\calr^{*}(e,n,m)$, 
define the sets
\begin{equation}\label{PiF}
{\Pi}_{[F]}:=\{S\in(\PP)^s_0~|~S\cap\sing(F)=\emptyset\},
\end{equation}
\begin{equation}\label{X_F}
\mathcal{X}_{[F]}:=\{(l,S)\in\mathrm{G}(2,4)\times
(\PP)^s_0~|~l\cap S=\emptyset,\ (l\cup S)\cap\sing(F)
=\emptyset, ~\mathrm{and}~F|_l=\mathcal{O}_l(e)\oplus
\mathcal{O}_l\}.
\end{equation}
Note that, since any reflexive sheaf $F$ from $\calr^{*}(e,n
,m)$ has 0-dimensional singularities, the set
\begin{equation}\label{R*times P3s}
\left(\calr^{*}(e,n,m)\times(\PP)_0^s\right)_0:= \{( [F],S)\in
\calr^{*}(e,n,m)\times(\PP)_0^s ~|~S\in{\Pi}_{[F]}\} 
\end{equation}
is a dense open subset in $\calr^{*}(e,n,m)\times(\PP)^s$, 
hence by \eqref{dim R*} it is smooth equidimensional of 
dimension
\begin{equation}\label{dim product 1}
\dim\left(\calr^{*}(e,n,m)\times(\PP)_0^s\right)_0=8n+3s+2e-3.
\end{equation}
Respectively, by the Grauert--M\"ulich Theorem, the set
\begin{equation}\label{R*times G times GP3s}
\left(\calr^{*}(e,n,m)\times G(2,4)\times(\PP)_0^s\right)_0:= 
\{( [F],(l,S))\in\calr^{*}(e,n,m)\times G(2,4)\times(\PP)_0^s 
~|~(l,S)\in\mathcal{X}_{[F]}\}
\end{equation}
is a dense open subset in $\calr^{*}(e,n,m)\times G(2,4)
\times(\PP)_0^s$, hence by \eqref{dim R*} it is smooth 
equidimensional of dimension
\begin{equation}\label{dim product}
\dim\left(\calr^{*}(e,n,m)\times G(2,4)\times(\PP)_0^s
\right)_0=8n+3s+2e+1.
\end{equation}
For a pair $([F],S)\in\left(\calr^{*}(e,n,m)\times(\PP)_0^s
\right)_0$, consider the 
$2s$-dimensional vector space $\Hom(F,\calo_S)$ and 
its open dense subset $\Hom(F,\calo_S)_e$ of epimorphisms 
$F\twoheadrightarrow \calo_S$. By construction, the group 
$\mathrm{Aut}(\calo_S)$ acts on $\Hom(F,\calo_S)_e$, and it 
follows that the quotient space $\Hom(F,\calo_S)_e/\mathrm
{Aut}(\calo_S)$ is a smooth irreducible scheme isomorphic to 
a product of projective spaces, where $S=(q_1,...,q_s)$:
\begin{equation}\label{descr fibre 1}
\begin{split}
&\Hom(F,\calo_S)_e/\mathrm{Aut}(\calo_S)\simeq\mathbf
\prod_{i=1}^{s}\mathbb{P}^1_{q_i},\\  
& \dim\Hom(F,\calo_S)_e/\mathrm{Aut}(\calo_S)=s,
\end{split}
\end{equation}
and where $\mathbb{P}^1_{q_i}=\Hom(F,\calo_{q_i})_e/\mathrm
{Aut}(\calo_{q_i})$, $i=1,...,s$.

Now, for any element $\phi\in\Hom(F,\calo_S)_e$ the torsion 
free sheaf $E_{\phi}:=\ker(\phi:F\twoheadrightarrow\calo_S)$ 
is stable, and defines a closed point in $\mathcal{M}(e,n,
m-2s)$. Furthermore, $E_{\phi}\simeq E_{\phi'}$ if, and only 
if, there is a $g\in\mathrm{Aut}(\calo_S)$ such that 
$\phi=g\circ\phi'$. Denote by $[\phi]$ the equivalence class 
of $\phi$ modulo $\mathrm{Aut}(\calo_S)$ and consider the set
\begin{equation}\label{tilde tau}
\tilde{\calt}(e,n,m,s):=\{x=([F],S,[\phi_x])~|~([F],S)
\in\left(\calr^{*}(e,n,m)\times(\PP)_0^s\right)_0, 
\phi_x\in\Hom(F,\calo_S)_e\}.
\end{equation}
By definition, $\tilde{\calt}(e,n,m,s)$ is fibered over
$\left(\calr^{*}(e,n,m)\times(\PP)_0^s\right)_0$
with fiber \linebreak $\Hom(F,\calo_S)_e/\mathrm{Aut}(\calo_S)$
over a given point $([F],S)$. Thus by \eqref{dim R*} and
\eqref{descr fibre} we conclude that $\tilde{\calt}(e,n,m,s)$ is
naturally endowed with a structure of smooth equidimensional
scheme of dimension
$$
\dim\tilde{\calt}(e,n,m,s)=8n+4s+2e-3,
$$
and the number of irreducible components of $\tilde{\calt}
(e,n,m,s)$ is equals to the number of those of 
$\calr^{*}(e,n,m)$.
Furthermore, for any point $t=([F],S),[\phi_x])\in\tilde
{\calt}(e,n,m,s)$, the sheaf $E(t):=\ker\{ \phi: F
\twoheadrightarrow\calo_S\}$ is a stable sheaf from
$\mathcal{M}(e,n,m-2s)$. Hence there is a well-defined 
modular morphism 
$$
\Phi:\ \tilde{\calt}(e,n,m,s)\hookrightarrow 
\mathcal{M}(e,n,m-2s),\ \ \ \ \ \ t\mapsto[E(t)].
$$ 
$\Phi$ is clearly an embedding, since the data $x=([F],(l,S),
[\phi_x])$ are recovered uniquely from $[E(t)]: F\simeq 
E(t)^{\vee\vee},\ S=\supp(E(t)^{\vee\vee}/E(t))$ 
and $\phi$ is the canonical quotient morphism
$E(t)^{\vee\vee}\twoheadrightarrow\calo_S\simeq 
E(t)^{\vee\vee}/E(t)$. We thus set
\begin{equation}
\calt(e,n,m,s):=\Phi(\tilde{\calt}(e,n,m,s))
\simeq\tilde{\calt}(e,n,m,s).
\end{equation} 
Let $\mathrm{T}(e,n,m,s)$ be the closure $\overline
{\calt(e,n,m,s)}$ of $\calt(e,n,m,s)$ in 
$\calm(e,n,m-2s)$. Formula \eqref{dim R*} yields:
\begin{equation}\label{dimfamilymixed1}
\dim\mathrm{T}(e,n,m,s)=\dim\calt(e,n,m,s)=8n+4s+2e-3.
\end{equation}

Respectively, let $r\geq e$. For a triple $([F],(l,S))\in
\left(\calr^{*}(e,n,m)\times G(2,4)\times(\PP)_0^s\right)_0$, 
set 
\begin{equation}\label{QlSr}
Q_{(l,S),r}:=\calo_S\oplus i_{*}\calo_l(r), 
\end{equation}
where $i:
l\hookrightarrow\PP$ is a closed immersion. Consider the 
$(2r+2s+2-e)$-dimensional vector space $\Hom(F,Q_{(l,S),r})$ 
and its open dense subset $\Hom(F,Q_{(l,S),r})_e$ of 
epimorphisms $F\twoheadrightarrow Q_{(l,S),r} $. By 
construction, the group $\mathrm{Aut}(Q_{(l,S),r})$ acts on \linebreak
$\Hom(F,Q_{(l,S),r})_e$, and it follows that the quotient space 
$$ \Hom(F,Q_{(l,S),r})_e/\mathrm{Aut}(Q_{(l,S),r}) $$
is a smooth irreducible scheme isomorphic to a product of 
projective spaces:
\begin{equation}\label{descr fibre}
\begin{split}
&\Hom(F,Q_{(l,S),r})_e/\mathrm{Aut}(Q_{(l,S),r})\simeq\mathbf
{P}^{2r+1-e}_l\times\prod_{i=1}^{s}\mathbb{P}^1_{q_i},\\  
& \dim\Hom(F,Q_{(l,S),r})_e/\mathrm{Aut}(Q_{(l,S),r})=
2r+s+1-e,
\end{split}
\end{equation}
and where 
\begin{equation}\label{Pl}
\mathbf{P}^{2r+1-e}_l=\Hom(F,i_{*}\calo_l
(r))_e/\mathrm{Aut}(i_{*}\calo_l(r))
\end{equation} 
and  $\mathbb{P}^1_{q_i}$ are the same as in \eqref{descr 
fibre 1}.

For any element $\phi\in\Hom(F,Q_{(l,S),r})_e$ the torsion 
free sheaf $E_{\phi}:=\ker\phi$ is stable, and defines a 
closed point in $\mathcal{M}(e,n+1,m-2r-2s-2-e)$. 
Furthermore, $E_{\phi}\simeq E_{\phi^{\prime}}$ if, and only 
if, there is a $g\in\mathrm{Aut}(Q_{(l,S),r})$ such that 
$\phi=g\circ\phi^{\prime}$. Denote by $[\phi]$ the 
equivalence class of $\phi$ modulo $\mathrm{Aut}
(Q_{(l,S),r})$ and consider the set
\begin{equation}\label{familyX}
\begin{split}
& \tilde{\mathcal{X}}(e,n,m,r,s):=\{x= 
([F],(l,S),[\phi_x])~|~([F],(l,S))\in\left(\calr^{*}(e,n,m)
\times G(2,4)\times(\PP)_0^s\right)_0,\\
&[\phi_x]\in\Hom(F,Q_{(l,S),r})_e/\mathrm{Aut}(Q_{(l,S),r})\}.
\end{split}
\end{equation}
By definition, $\tilde{\mathcal{X}}(e,n,m,r,s)$ is fibered 
over
$\left(\calr^{*}(e,n,m)\times G(2,4)\times(\PP)_0^s\right)_0$
with fiber\\ $\Hom(F,Q_{(l,S),r})_e/\mathrm{Aut}(Q_{(l,S),r})$
over a given point $([F],(l,S))$. Thus by \eqref{dim R*} and
\eqref{descr fibre}  $\tilde{\mathcal{X}}(e,n,m,r,s)$ is
naturally endowed with a structure of smooth equidimensional
scheme of dimension
\begin{equation}\label{dim tilde X}
\dim\tilde{\mathcal{X}}(e,n,m,r,s)=8n+4s+2r+e+2,
\end{equation}
and the number of irreducible components of $\tilde
{\mathcal{X}}(e,n,m,r,s)$ equals the number of those of 
$\calr^{*}(e,n,m)$.
Furthermore, for any point 
$$ x=([F],(l,S),[\phi_x])\in\tilde
{\mathcal{X}}(e,n,m,r,s), $$ the sheaf $E(x):=\ker\{ \phi: F
\twoheadrightarrow Q_{(l,s),r}\}$ is a stable sheaf from
$\mathcal{M}(e,n+1,m+2+e-2r-2s)$. Hence there is a 
well-defined modular morphism 
$$
\Psi:\ \tilde{\mathcal{X}}(e,n,m,r,s)\hookrightarrow 
\mathcal{M}(e,n+1,m+2+e-2r-2s),\ \ \ \ \ \ x\mapsto[E(x)].
$$ 
$\Psi$ is clearly an embedding, since the data $x=([F],(l,S),
[\phi_x])$ are recovered uniquely from $[E(x)]: F\simeq 
E(x)^{\vee\vee},\ l\sqcup S=\supp(E(x)^{\vee\vee}
/E(x))$ and $\phi$ is the canonical quotient morphism
$E(x)^{\vee\vee}\twoheadrightarrow Q_{(l,S),r}\simeq 
E(x)^{\vee\vee}/E(x)$. We thus set
\begin{equation}\label{def calX}
\mathcal{X}(e,n,m,r,s):=\Psi(\tilde{\mathcal{X}}(e,n,m,r,s))
\simeq\tilde{\mathcal{X}}(e,n,m,r,s).
\end{equation} 
Let $\mathrm{X}(e,n,m,r,s)$ be the closure $\overline{\mathcal
{X}(e,n,m,r,s)}$ of $\mathcal{X}(e,n,m,r,s)$ in $\calm(e,n+1,
m+2+e-2r-2s)$. Formula \eqref{dim tilde X} yields:
\begin{equation}\label{dimfamilymixed}
\dim\mathrm{X}(e,n,m,r,s)=\dim\mathcal{X}(e,n,m,r,s)= 
8n+4s+2r+2+e.
\end{equation}
\begin{Remark}\label{X(-1,1,1,-1,0)}
By Remark \ref{R(-1,1,1)}, $\calr^*(-1,1,1)=
\calr(-1,1,1)$ is smooth irreducible of the expected 
dimension 
3. Thus \eqref{dimfamilymixed} yields that $\mathrm{X}(-1,1,1,
-1,0)$ is an irreducible scheme of dimension 7.
\end{Remark} 

We now prove the following general result about the schemes $\mathrm{T}(e,n,m,s)$.

\begin{Teo}\label{0dcomp}
Given $s>0$, we have:
\begin{itemize}
\item[(i)] For any nonsingular irreducible component 
$\calr^*$ of of $\calr(e,n,m)$ there corresponds an 
irreducible component of $\mathrm{T}(e,n,m,s)$ of dimension 
$8n-3+2e+4s$ which is also an irreducible component of 
$\calm(e,n,m-2s)$. In particular, if $\calr(e,n,m)$ is 
irreducible, then $\mathrm{T}(e,n,m,s)$ is also irreducible.
\item[(ii)] The generic sheaf $[E]$ of any irreducible 
component of $\mathrm{T}(e,n,m,s)$ satisfies the conditions 
that $[E^{\vee\vee}]\in\calr^*(e,n,m)$ and 
$Q_E=E^{\vee\vee}/E$ is an artinian sheaf of length $s$.
\end{itemize}
\end{Teo} 
\begin{proof}
For the statement (i) of Theorem, it is enough to prove that, 
for each $[E(t)]\in\calt(e,n,m,s)$, $\ext^2(E(t),E(t))=4s$. 
Indeed, in this case Theorem \ref{dimext1} yields that $\dim 
\ext^1(E(t),E(t))=8n-3+2e+4s=$, and this dimension coincides  
with the dimension of $\mathrm{T}(e,n,m,s)$ by 
\eqref{dimfamilymixed1}, and therefore by the deformation 
theory any irreducible component of $\mathrm{T}(e,n,m,s)$ is 
an irreducible component of $\calm(e,n,m-2s)$.  

From Proposition \ref{important} we have $$ \dim\ext^2(E(t),
E(t))=h^0(\lext^3(Q,E(t))), $$
where $Q=E(t)^{\vee\vee}/E(t)$. 
To compute this group, note that, since, by the definition of
$\calt(e,n,m,s)$, $Q=\calo_S$, where $S=\{q_1,\dots,q_s\}\in
(\PP)_0^s$ and $S\cap\sing(E(t)^{\vee\vee})=\emptyset$, we 
have
\begin{equation}\label{sum of exts}
\ext^2(E(t),E(t)) \simeq \h^0(\lext^3(Q, E(t)))\simeq\bigoplus_{q_i\in S} 
\ext^3_{\calo_{\PP,q_i}}(\calo_{q_i},E(t)_{q_i}).
\end{equation}
Take a point $q=q_j$ for some $1\le j\le s$, and an open 
subset
$U$ in $\PP$ not containing other points of $\sing(E(t))$.
Consider the exact sequence $0\to E(t)\to E(t)^{\vee\vee}\to 
Q\to0$ and restrict it onto $U$. We then obtain the following 
exact sequence of sheaves on $U$:
$$ 
0\to\calo_U\oplus\cali_{q,U}\to2\cdot\calo_{U}\to\calo_q\to0, 
$$
where $\cali_{q,U}$ denotes the ideal sheaf of the point $p\in 
U$ and $\calo_q$ denotes the structure sheaf of the point 
$q$ as a subscheme of $U$. In particular, $E(t)|_U\simeq
\calo_U\oplus \cali_{q,U}$, so that
\begin{equation}\label{dir sum}
\ext^3_{\calo_{\PP,q}}(\calo_q,E(t)_q)\simeq\h^0(
\lext^3_{\calo_{U}}(\calo_q, \cali_{q,U}))\oplus 
\h^0(\lext^3_{\calo_{U}}(\calo_q,\calo_U)). 
\end{equation}
Applying the functor $\lhom(-,\cali_{q,U})$ to the sequence 
$0 \to \cali_q\to\calo_{U}\to\calo_q\to 0$, we obtain: 
$\lext^3_{\calo_{U}}(\calo_q, \cali_{q,U}) \simeq \lext^2_
{\calo_{U}}(\cali_{q,U}, \cali_{q,U})$. The last sheaf is an 
artinian sheaf of length $3$ by the proof of \cite
[Proposition 6]{JMT}. Thus, since $\lext^3_{\calo_{U}}
(\calo_q,\calo_U))\simeq \calo_q $, it follows from
\eqref{dir sum} that each point in S contributes with 4 to 
the dimension of $\ext^2(E(t), E(t))$, hence, by \eqref{sum of exts}, $\dim \ext^2
(E(t),E(t))=4s$. The other claims in the statement of Theorem are clear from the definition of $\mathrm{T}(e,n,m,s)$.
\end{proof}

We next proceed to a general result about the schemes 
$\mathrm{X}
(e,n,m,r,s)$.

\begin{Teo}\label{NewComponentsmixed} 
 Let  $e, n, m, r, s$ be integers such that 
$e\in\{-1,0\}$, 
$n,\ m>0$, $r\ge e$ and $s\ge0$. Then the scheme  $\mathrm{X}
(e,n,m,r,s)$ is equidimensional of dimension $8n+4s+2r+2+e$,
and the number of irreducible components of $\mathrm{X}
(e,n,m,r,s)$ is the same as that of $\calr^{*}(e,n,m)$. 
Furthermore, $\mathrm{X}(e,n,m,r,s)$ contains a dense open
subset $\mathcal{X}(e,n,m,r,s)$ such that, for $[E]\in
\mathcal{X}(e,n,m,r,s)$, the following statements hold.
\begin{itemize}
\item[(i)]  If $r\ge1$, then $\dim\ext^1(E,E)=8n+4s+2r+2+e=
\dim\mathrm{X}(e,n,m,r,s)$. Hence, if $\calr^{*}(e,n,m)$ is 
irreducible, then $\mathrm{X}(e,n,m,r,s)$ is an irreducible 
$(8n+4s+2r+2+e)$-dimensional component of $\calm(e,n+1,
m+2+e-2r-2s)$.
\item[(ii)] If $0\ge r\ge e$, then $\dim\ext^1(E,E)=8n+4s+5+
2e>\dim\mathrm{X}(e,n,m,r,s)$.
\end{itemize}
\end{Teo}
\begin{proof} The first claim  follows from 
\eqref{dimfamilymixed} and the above considerations. For the
statements (i) and (ii), consider any sheaf $[E]\in\mathcal
{X}(e,n,m,r,s)$. By definition $[E]$ defines a line $l$ and a 
set of $s$ points $S$ considered as a reduced scheme as: 
$l\sqcup S=\supp(E^{\vee\vee}/E)$. Note that, by 
Proposition \ref{important}.(iii) in which we set $T=i_*
\calo_l(r),\ Z=\calo_S$, where $i:Z\hookrightarrow\PP$ is the 
embedding, one has
\begin{equation}\label{dim ext2}
\dim\ext^2(E,E)=h^0(\lext^3(\mathcal{O}_S,E))+\dim 
\ext^3(i_{*}\mathcal{O}_l(r),E).
\end{equation}
First, one has 
\begin{equation}\label{h0=}
h^0(\lext^3(\mathcal{O}_S,E))=4s.
\end{equation}
The proof of this equality is given in \cite[Proof of 
Prop. 6]{JMT} in the case $e=0$. However, since $\lext^2(E,E)
$ is 0-dimensional, the computation of $h^0(\lext^2(E,E))$ 
is purely local, and gives the same result for $e=-1$.
Next, $\ext^3(i_{*}\mathcal{O}_l(r),E)\simeq 
\Hom(E,i_{*}\mathcal{O}_l(r-4))^{\vee}$ by Serre duality, and
$$ \Hom(E,i_{*}\mathcal{O}_l(r-4))\simeq\h^0(\lhom(E,i_{*}
\mathcal{O}_l(r-4))). $$
To compute $h^0(\lhom(E,i_{*}
\mathcal{O}_l(r-4)))$, 
apply the functor $i^{*}(-\otimes i_{*}\mathcal{O}_l)$ to the triple
$$ 0\to E\to E^{\vee\vee}\to\calo_S\oplus 
i_*\calo_l(r)\to0. $$
Using the fact that $E^{\vee\vee}|_{l}
\simeq\calo_l\oplus\calo_l(e)$ we obtain the exact sequence
$$ 0\to i^{*}Tor_1(i_{*}\mathcal{O}_l(r),i_{*}\mathcal{O}_l)\to 
E|_l\xrightarrow{f}\mathcal{O}_l(e)\oplus\mathcal{O}_l 
\xrightarrow{g}\mathcal{O}_l(r)\to 0. $$
Whence $\ker g\simeq \calo_l(-r+e)$, and since 
$i^{*}Tor_1(i_{*}\mathcal{O}_l(r),i_{*}\mathcal{O}_l)\simeq 
N_{l/\PP}^{\vee}\otimes\mathcal{O}_l(r) \simeq 
2\cdot \mathcal{O}_l(r-1)$, we obtain an exact sequence 
$0\to2\cdot \mathcal{O}_l(r-1)\to E|_l\to\calo_l(-r+e)\to0$.
This triple easily implies that $h^0(\lhom(E,i_{*}
\mathcal{O}_l(r-4)))$ equals $2r-3-e$ for $r\ge1$ and, 
respectively, equals 0 for $0\ge r\ge e$. This together with
\eqref{dim ext2}, \eqref{h0=} and Theorem \ref{dimext1} with
$c_2(E)=n+1$ yields the statements (i) and (ii) of Theorem.
\end{proof}

We conclude this section with our first application of 
Theorem \ref{NewComponentsmixed}. 


The case of moduli spaces $\calm(-1,n,0)$ is interesting from
the point of view that it contains, among others, all those 
irreducible components that have locally free sheaves (i. e.,
vector bundles) as their generic points. Ein had shown in
\cite{Ein} that the number of these components for given 
number $n$ is unbounded as $n$ grows infinitely. Therefore, 
it is important to understand whether the components of 
$\calm(-1,
n,0)$ with non-locally free sheaves as their generic points 
satisfy the similar property. In this section we give an 
affirmative answer to this question in Theorem \ref{einmixed}
below. Namely, the components $\mathrm{X}(-1,n,m,r,s)$ 
described in Theorem \ref{NewComponentsmixed} will serve for
this purpose, with the numbers $n,m,r,s$ chosen 
appropriately.

\begin{Teo}\label{einmixed}
Let $\eta_n$ and $\xi_n$ denote the number of irreducible 
components of $\mathcal{M}(-1,n,0)$ whose generic point 
corresponds to a non-locally free sheaf with mixed 
singularities and with 1-dimensional singularities, 
respectively. Then 
$$\lim\sup_{n\to\infty}\eta_n = 
\lim\sup_{n\to\infty}\xi_n = \infty. $$
\end{Teo}
\begin{proof}
For any odd integer $q\ge1$ set $n_q=9q^2-9q+1$, and for any 
integer $i$ such that $0\le i\le q-1,$ set $a_{q,i}=i,\ 
b_{q,i}=3q-3i-3,\ c_{q,i}=3i+1$. Then, according to Theorem 8 
of \cite{JMT}, for an odd integer $m_{q,i}=m(a_{q,i},b_{q,i},
c_{q,i})$ given by (\ref{m(a,b,c)}) the scheme 
$\calr^*(-1;n_q,
m_{q,i})$ defined in \eqref{def R*} is nonempty. Thus, by 
Theorem \ref{NewComponentsmixed}, to any integers $s$ and $r$ 
such that $0\leq s\leq n_q-1$, $r=\frac{1}{2}(m_{q,i}+1-2s)$, 
there corresponds an equidimensional union $\mathrm{X}(-1,
n_q-1,m_{q,i},r, s)$ of irreducible components of $\calm
(-1,n_q,0)$ of dimension $8n_q+2s+m_{q,i}+3$, where the 
number of irreducible components of $\mathrm{X}(-1,n_q-1,m_
{q,i},r,s)$ is the same as that of $\calr^*(-1;n_q,m_{q,i})$. 
Therefore, since $0\le i\le q-1$, for each odd $q$ we obtain 
at least $q$ different  irreducible components of $\mathrm{X}
(-1,n_q-1,m_{q,i},r,s),\ i=0,...,q-1,$ which are irreducible 
components of $\calm(-1,n_q,0)$, generic points of which are 
sheaves with mixed singularities. Taking $s=0$ we obtain the 
claim about sheaves with 1-dimensional singularities.
\end{proof}

A similar result also holds for sheaves with 0-dimensional 
singularities. The proof is very similar to \cite[Theorem 
9]{JMT}, using the series of components of $\calr(-1,n,m)$ 
produced in Theorem \ref{thm newclass}. More precisely:

\begin{Teo}\label{ein0}
Let $\zeta_n$ denote the number of irreducible components
of $\mathcal{M}(-1,n,0)$ whose generic point corresponds to a non-locally free sheaf with
0-dimensional singularities. Then 
$$\lim\sup_{n\to\infty}\zeta_n  = \infty. $$
\end{Teo}

\section{Infinite collections of rational moduli components}
\label{Rational Components}

In this section we will construct an infinite collection of
rational moduli components of the spaces $\calm(-1,n,m)$ and 
$\calm(0,n,m)$ with generic sheaves $E$ satisfying
$\dim\sing(E)=0$. This collection will include all previously 
known rational moduli components of $\calm(-1,n,m)$ and 
$\calm(0,n,m)$ whose generic sheaf has the property above. As 
a consequence, we can conclude that the moduli schemes 
$\mathcal{M}(-1,n,m)$ and $\mathcal{M}(0,n,m)$ contain at 
least one rational irreducible components for all $n\ge1$ and 
all admissible $m$.

The desired collection will be constructed via elementary
transformations at sets of points from certain moduli 
components of stable reflexive rank 2 sheaves on 
$\mathbb{P}^3$. For this, we invoke results of Chang 
\cite{Chang,Chang2}, Mir\'o-Roig \cite{Maria-Miro}, 
Okonek--Spindler \cite{OS1985} and Schmidt \cite{S2018}
on the moduli spaces of reflexive sheaves $\mathcal{R}(e,n,m)$
, $e\in\{0,1\}$. These are the results 
concerning the moduli spaces of reflexive sheaves with
Chern classes belonging to the set of triples of integer 
numbers \label{sigmas}
\begin{equation}\label{def sigma}
\Sigma:=\Sigma_{-1}\cup\Sigma_{0},
\end{equation}
where $\Sigma_{-1}$ and $\Sigma_0$ being respectively given by
\begin{equation}\label{def sigma -1}
\{(-1,n,n^2)\ |\ n\ge1\} ~\bigcup~ 
\{(-1,n,n^2-2rn+2r(r+1))\ |\ n\ge5,\ 1\le 
r\le(\sqrt{4n-7}-1)/2\}
\end{equation}
and
\begin{equation}\label{def sigma 0}
\{(0,n,n^2-n+2)\ |\ n\ge3\} ~\bigcup~ \{(0,n,n^2-n)\ |\ 
n\ge4\}~\bigcup~ \{(0,n,n^2-3n+8)\ |\ n\ge5\}.
\end{equation}
According to \cite{Chang2,Maria-Miro,OS1985,S2018},
for each triple $(e,n,m)\in\Sigma$, the moduli space of 
stable rank 2 reflexive sheaves $\mathcal{R}(e,n,m)$ 
satisfies the following properties:\\
(I) Each $R=\mathcal{R}(e,n,m)$ is an irreducible, 
nonsingular and rational scheme, and it is a dense open 
subset of an irreducible component of $\mathcal{M}(e,n,m)$;\\ 
(II) $R$ is a fine moduli space, i. e., there exists a 
universal family of reflexive sheaves 
$\boldsymbol{\mathcal{F}}$ on $R\times\PP$. (In the 
case $e=-1$ this a well-known property the moduli spaces of 
rank 2 sheaves on $\mathbb{P}^3$ with odd determinant - see 
for instance \cite[Thm. 4.6.5]{HL}. In the case $e=0$ this 
follows from the explicit constructions of reflexive sheaves 
from $R$ as extensions of standard sheaves. These 
constructions are provided in 
\cite{Chang2,Maria-Miro,OS1985,S2018}.) \\
(III) The dimension of each $R$ is given by:
\begin{equation}\label{dim R -1}
\begin{split}
& \dim\mathcal{R}(-1,n,n^2)=n^2+3n+1\ \ \ {\rm if}\ \ \ 
n\ge2,\ \ \ {\rm resp.},\ \ \  3\ \ \  {\rm if}\ \ \ n=1,\\
& \dim\mathcal{R}(-1,n,n^2-2rn+2r(r+1))=
n^2+(3-2r)n+2r^2+5,\ \ \ if\ \ \ r\ge2,\ \ n\ge5,\\
& \dim\mathcal{R}(-1,n,n^2-2rn+2r(r+1))=n^2+n+6,\ \ \  {\rm 
if}\ \ \ r=1,\ \ n\ge5,\\
\end{split}
\end{equation}
\begin{equation}\label{dim R 0}
\begin{split}
& \dim\mathcal{R}(0,n,n^2-n+2)=n^2+2n+5,\ \ \ n\ge4,\ \ 
\ {\rm resp.},\ \ \  21\ \ \  {\rm if}\ \ \ n=3, \\
& \dim\mathcal{R}(0,n,n^2-n)=n^2+2n+5,\ \ \ n\ge4,\\
& \dim\mathcal{R}(0,n,n^2-3n+8)=n^2+11\ \ \  {\rm if}\ \ \ 
n\ge6,\ \ \ {\rm resp.},\ \ \  37\ \ \  {\rm if}\ \ \ n=5.
\end{split}
\end{equation}
(IV) For $e=-1$ and arbitrary integer $n\ge1$, the maximal
possible $m$ such that $\calm(-1,n,m)\ne\emptyset$ equals 
$n^2$; note that $(-1,n,n^2)\in\Sigma_{-1}$ for $n\ge1$.
For $e=0$ and arbitrary integer $n\ge1$, the maximal possible 
$m$ such that $\calm(-1,n,m)\ne\emptyset$ equals $n^2-n+2$; 
note that $(-1,n,n^2-n+2)\in\Sigma_0$.\\
(V) The dimensions of the components $\mathcal{R}(e,n,m)$
satisfy the relations:
\begin{equation}\label{dim=dim Ext}
\dim\mathcal{R}(e,n,m)=\dim \ext^1(F,F),\ \ \ \ \ \ 
[F]\in\mathcal{R}(e,n,m),\ \ \ \ \ \ (e,n,m)\in
\Sigma_{-1}\cup\Sigma_0.
\end{equation}

Now take an arbitrary scheme $R=\mathcal{R}(e,n,m)$ for 
$(e,n,m)\in\Sigma$ and, similarly to \eqref{R*times P3s}, set
\begin{equation}\label{R times P3s}
\left(R\times(\PP)_0^s\right)_0:= \{([F],S)\in 
R\times(\PP)_0^s ~|~S\in{\Pi}_{[F]}\}, 
\end{equation}
\begin{equation}\label{R times G times P3s}
\left(R\times G(2,4)\times(\PP)_0^s\right)_0:= 
\{( [F],(l,S))\in R\times G(2,4)\times(\PP)_0^s 
~|~(l,S)\in\mathcal{X}_{[F]}\},
\end{equation}
where ${\Pi}_{[F]}$ is defined in \eqref{PiF} for any 
reflexive sheaf $[F]\in R$. In particular, for $s=1$ we have
\begin{equation}\label{minus Sing}
(R\times\PP)_0=\{([F],x)\in R\times\PP\ |\ x\not\in
\sing(F)\}.
\end{equation}
By property (III) above there is a universal sheaf 
$\boldsymbol{\mathcal{F}}$ on $R\times\PP$, and the definition
\eqref{minus Sing} yields that
$$
\mathbf{F}:=\boldsymbol{\mathcal{F}}|_{(R\times\PP)_0}
$$
is a locally free rank 2 sheaf. Hence, there exists an open
subset $U$ of $(R\times\PP)_0$ such that 
\begin{equation}
\mathbf{P}(\mathbf{F}|_U)\simeq U\times\mathbb{P}^1,
\end{equation}
and we have dense open inclusions
\begin{equation}\label{two embeddings}
\mathbf{P}(\mathbf{F})\overset{\mathrm{open}}{\hookleftarrow}
U\times\mathbb{P}^1\overset{\mathrm{open}}{\hookrightarrow}
R\times\PP\times\mathbb{P}^1.
\end{equation}
Now introduce a piece of notation. Let $s$ be a positive 
integer and let $f:X\to Y$ be an arbitrary morphism of 
schemes. The symmetric group $G=\mathcal{S}_s$ acts on 
$W=\prod_{1}^{s}X$ by permutations of factors, and the 
$s$-fold fibered product $X\times_Y\cdots\times_YX$ naturally 
embeds in $W$ as a $G$-invariant subscheme. We will denote by 
$\mathrm{Sym}^s(X/Y)$ the quotient scheme $(X\times_Y\cdots
\times_YX)/G$ and call this quotient the fibered symmetric 
product of $X$ over $Y$.

Fix an integer $s\ge1$. The composition of projections
$$
f:\ \mathbf{P}(\mathbf{F})\xrightarrow{\pi}(R\times\PP)_0
\hookrightarrow R\times\PP\xrightarrow{pr_1}R,
$$
where $\pi$ is the structure morphism, defines the
fibered symmetric product $\mathrm{Sym}^s(\mathbf{P}
(\mathbf{F})/R)$ together with the projection
$f_s:\mathrm{Sym}^s(\mathbf{P}(\mathbf{F})/R)\to R$ which 
factorizes as
$$
f_s:\mathrm{Sym}^s(\mathbf{P}(\mathbf{F})/R)\xrightarrow
{\pi_s}R\times\mathrm{Sym}^s(\PP)\xrightarrow{pr_1}R.
$$
The open embedding $(\PP)_0^s\hookrightarrow
\mathrm{Sym}^s(\PP)$ together with the above projection 
$\pi_s$ defines an open dense embedding of the fibered product
\begin{equation}\label{Y R}
Y_R:=\mathrm{Sym}^s(\mathbf{P}(\mathbf{F})/R)\times
_{\mathrm{Sym}^s(\PP)}(\PP)_0^s\overset{\mathrm{open}}
{\hookrightarrow}\mathrm{Sym}^s(\mathbf{P}(\mathbf{F})/R).
\end{equation}
By the definition of $Y_R$, its set-theoretic description
is the same as that of $\tilde{\calt}(e,n,m,s)$, with 
$\calr^*(e,n,m)$ substituted by $R$:
\begin{equation}\label{Y_R as a set}
Y_R:=\{y=([F],S,[\varphi_y])~|~([F],S)\in\left(R\times(\PP)
_0^s\right)_0, \varphi_y\in\Hom(F,\calo_S)_e\},
\end{equation}
where $\left(R\times(\PP)_0^s\right)_0$ is defined in 
\eqref{R times P3s}.

Now define a new set $X_R$ by the formula similar to 
\eqref{familyX}, with $\calr^*(e,n,m)$ substituted by $R$: 
\begin{equation}\label{X_R as a set}
\begin{split}
& X_R:=\{x=([F],(l,S),[\varphi_x])~|~([F],(l,S))\in
\left(R\times G(2,4)\times(\PP)_0^s\right)_0,\\
&[\varphi_x]\in\Hom(F,Q_x)_e/\mathrm{Aut}(Q_x),\ Q_x:=
Q_{(l,S),r}\},
\end{split}
\end{equation}
where $\left(R\times G(2,4)\times(\PP)_0^s\right)_0$ and 
$Q_{(l,S),r}$ are defined in \eqref{R times G times P3s} and
\eqref{QlSr}, respectively. There is a well-defined projection
\begin{equation}\label{rho=}
\rho:\ X_R\to Y_R\times G(2,4),\ \ \ ([F],(l,S),[\varphi_x])
\mapsto\big(([F],S,[\varphi_x|_{\calo_S}]),l\big) ~~,
\end{equation}
such that 
$$
\mathcal{V}=\rho(X_R)
$$
is a dense open subset of $Y_R\times G(2,4)$, and
\begin{equation}\label{rho-1}
\rho^{-1}(y,l)=\mathbf{P}^{2r+1-e}_l,\ \ \ \ (y,l)\in
\mathcal{V},
\end{equation}
where $\mathbf{P}^{2r+1-e}_l$ is described in \eqref{Pl} and
\eqref{QlSr}. Let $\Gamma\subset G(2,4)\times\PP$ be the 
graph of incidence with projections $G(2,4)\leftarrow\Gamma
\to\PP$. Consider the natural projections
$$
\mathcal{V}\xleftarrow{v}\mathcal{V}\times_G\Gamma\times_{\PP}
(R\times\PP)_0\xrightarrow{g}(R\times\PP)_0\xrightarrow{h}\PP
$$
and a locally free sheaf $\mathbf{E}$ defined as
$$
\mathbf{E}=(v_*(g^*\mathbf{F}\otimes h^*\op3(r))),\ \ \ \ \ 
\rk\mathbf{E}=2r+2-e.
$$
Then 
\begin{equation}\label{rho again}
\rho:\ X_R=\mathbf{P}(\mathbf{E})\to\mathcal{V}
\hookrightarrow Y_R\times G(2,4),
\end{equation}
is a locally trivial $\mathbf{P}^{2r+1-e}$-fibration with 
fibre  $\mathbf{P}^{2r+1-e}_l$ described in \eqref{rho-1}.

From the definition \eqref{R times P3s} of $\left(R\times(\PP)
_0^s\right)_0$ it follows that $\pi_s(Y_R)\subset\left
(R\times(\PP)_0^s\right)_0$. We thus consider the composition
\begin{equation}\label{prn pi s}
f_Y:Y_R\xrightarrow{\pi_s}(R\times(\PP)_0^s)_0\hookrightarrow
R\times(\PP)_0^s\xrightarrow{pr_1}R.
\end{equation}
Note that the open embeddings in \eqref{two embeddings} 
commute with the natural projections 
$f:\mathbf{P}(\mathbf{F})\to R$,\ $f':U\times\mathbb{P}^1\to 
R$ and $f'':R\times\PP\times\mathbb{P}^1\to R$ and therefore 
define the induced dense open embeddings
\begin{equation}\label{two induced emb}
\mathrm{Sym}^s(\mathbf{P}(\mathbf{F})/R)\overset{\mathrm{open}
}{\hookleftarrow}\mathrm{Sym}^s(U\times\mathbb{P}^1/R)
\overset{\mathrm{open}}{\hookrightarrow}\mathrm{Sym}^s
(R\times\PP\times\mathbb{P}^1/R)
\end{equation}
which commute with the induced projections
$f_s:\mathrm{Sym}^s(\mathbf{P}(\mathbf{F})/R)\to R$,\ 
$f'_s:\mathrm{Sym}^s(U\times\mathbb{P}^1/R)\to R$ and
$f''_s:\mathrm{Sym}^s(R\times\PP\times\mathbb{P}^1/R)\to R$.
The diagram \eqref{two induced emb} yields a birational 
isomorphism 
$\mathrm{Sym}^s(\mathbf{P}(\mathbf{F})/R)
\overset{\mathrm{bir}}{\dashleftarrow\dashrightarrow}
\mathrm{Sym}^s(R\times\PP\times\mathbb{P}^1/R)$, hence the 
dense open embedding \eqref{Y R} leads to a birational 
isomorphism
\begin{equation}\label{bir isom}
Y_R\overset{\mathrm{bir}}{\dashleftarrow\dashrightarrow}
\mathrm{Sym}^s(R\times\PP\times\mathbb{P}^1/R).
\end{equation}
On the other hand, from the definition of
$\mathrm{Sym}^s(R\times\PP\times\mathbb{P}^1/R)$ follows an
isomorphism
\begin{equation}\label{isom 1}
\mathrm{Sym}^s(R\times\PP\times\mathbb{P}^1/R)\simeq
R\times\mathrm{Sym}^s(\PP\times\mathbb{P}^1/R).
\end{equation}
Since any symmetric product of a rational variety is also 
rational (see for instance \cite[Ch. 4, Thm. 2.8]{GKZ}), it 
follows from \eqref{isom 1} and the property (I) that 
$\mathrm{Sym}^s(R\times\PP\times \mathbb{P}^1/R)$ is a 
rational irreducible scheme of dimension $4s+\dim R$. Hence 
by \eqref{bir isom}
\begin{equation}\label{rational scheme}
Y_R\ \ \textrm{is a rational irreducible scheme of 
dimension}\ \ 4s+\dim R.
\end{equation}
This together with the the description \eqref{rho again}
yields:
\begin{equation}\label{rational scheme2}
X_R\ \ \textrm{is a rational irreducible scheme of 
dimension}\ \ 4s+2r+5-e+\dim R.
\end{equation}

\begin{Teo}\label{Tmain2} 
For any $(e,n,m)\in\Sigma_{-1}\cup\Sigma_0$ and $R$ irreducible component of $\calr(e,n,m)$, we have: 
\begin{itemize}
\item[(i)] for any integer $s$ such that $0\le 2s\le m$ there exists a 
rational, generically reduced, irreducible component 
$\overline{Y}_R$ of the moduli space $\calm(e,n,m-2s)$ having 
the dimension $4s+\dim\calr(e,n,m)$, where $\dim\calr(e,n,m)$ 
is given by one of the corresponding formulas \eqref{dim R -1}-\eqref{dim R 0}. A generic sheaf from $\overline{Y}_R$ has 0-dimensional singularities;

\item[(ii)] for any integers $r,s$ such that $s\ge0$, $r\ge3$ and $2r+2s\le m+2+e$ there exists a rational, generically reduced, irreducible and  component $\overline{X}_R$ of the moduli space 
$\calm(e,n+1,m+2+e-2r-2s)$ having the dimension $4s+2r+5-e+
\dim\calr(e,n,m)$, with $\dim\calr(e,n,m)$ given by the 
formulas mentioned above. A generic sheaf from 
$\overline{X}_R$ has singularities of pure dimension 1 for
$s=0$, and, respectively, of mixed dimension for $s\ge1$.
\end{itemize}
\end{Teo}

\begin{proof}


\noindent{\bf Item (i).}  We are going to show that $Y_R$ is an open dense subset 
of an irreducible component $\overline{Y}_R$ of 
$\calm(e,n,m-2s)$, where $R=\calr(e,n,m)$. 
To this aim, we first construct a family of sheaves on 
$\PP$ with base $Y_R$ which are obtained from reflexive 
sheaves $[F]\in R$ via elementary transformations along sets 
of $s$ points. Let $H={\rm Hilb}^s(\PP)$ be the Hilbert 
scheme of 0-dimensional subschemes of length $s$ in $\PP$,
together with the universal family of 0-dimensional schemes 
$Z_H\hookrightarrow H\times\PP$. We have an open embedding 
$(\PP)^s_0\hookrightarrow H$ and the induced family
$Z=Z_H\times_H(\PP)^s_0\hookrightarrow(\PP)^s_0\times\PP$.
Given a point $\{S\}\in(\PP)^s_0$, we will denote also by $S$ 
the corresponding 0-dimensional subscheme $Z\times_{H(\PP)^s
_0}\{S\}$ in $\PP$. (This will not cause an ambiguity since 
$S$ is a reduced scheme by the definition of $(\PP)^s_0$.) 
According to \cite[Ch. II, Prop. 7.12]{H}, for a given sheaf 
$[F]\in R$ and the above point $\{S\}$ such that $S\cap\sing
(F)=\emptyset$, choosing a class $[\varphi]$ modulo 
$\mathrm{Aut}(\mathcal{O}_S)$ of an epimorphism
$$
\varphi:F\twoheadrightarrow\mathcal{O}_S
$$
is equivalent to choosing a section $[\varphi]$ of the 
structure morphism $\pi:\mathbf{P}(F|_S)\to S$,
$$
[\varphi]:S\hookrightarrow\mathbf{P}(F|_S).
$$
By the construction of $Y_R$ (see \eqref{Y R}), the section 
$[\varphi]$ is just a point of $Y_R$ lying in the fiber 
$\pi_s^{-1}([F],\xi)$ of the projection $\pi_s:Y_R\to
(R\times(\PP)_0^s)_0$ defined in \eqref{prn pi s}. Using this
description of points of $Y_R$, define a family 
\begin{equation}\label{family for YR}
\{[E_y]\in\calm(e,n,m-2s)\ |\ E_y=\ker(\varphi_y:F
\twoheadrightarrow\mathcal{O}_S),\ y=([F],S,[\varphi_y])\in 
Y_R\}.
\end{equation}
Here, by definition, for each $y=([F],S,[\varphi_y])\in Y_R$, 
the sheaf $E=E_y$ satisfies the exact triple
\begin{equation}\label{triple for E}
0\to E\to F\xrightarrow{\varphi_y}\mathcal{O}_S\to0,
\ \ \ \ \ F=E^{\vee\vee}.
\end{equation} 
In particular, this triple, together with the stability of 
$F$, yields by usual argument the stability of $E$, i. e., the
definition of the family \eqref{family for YR} is consistent. 
The family $\{E_y\}_{y\in Y_R}$ globalizes in a standard way 
to a sheaf $\mathbf{E}$ on $Y_R\times\PP$ such that, for any 
$y\in Y_R$, $\mathbf{E}|_{\{y\}\times\PP}\simeq E_y$. We thus 
have a natural modular morphism
\begin{equation}\label{modular morphism for YR}
\Psi:Y_R\to\calm(e,n,m-2s),\ y\mapsto\left[\mathbf{E}|_{\{y\}\times\PP}\right].
\end{equation}
The morphism $\Psi$ is clearly an embedding, since a point 
$\{y\}$ is recovered from $E=\ker(\varphi_y)$ as the (class 
of the) quotient map $F=E^{\vee\vee}\twoheadrightarrow
E^{\vee\vee}/E$. We therefore identify $Y_R$ with its image
in $\calm(e,n,m-2s)$. Let $\overline{Y}_R$ be the closure
of $Y_R$ in $\calm(e,n,m-2s)$.

We have to show that $\overline{Y}_R$ is an irreducible 
rational component of $\calm(e,n,m-2s)$, where $R=\calr(e,n,
m)$. Here the rationality and the dimension of $\overline{Y}
_R$ are given in display \eqref{rational scheme}. Since 
$\overline{Y}_R$ is irreducible, to prove that 
$\overline{Y}_R$ is an irreducible and generically reduced
component of $\calm(e,n,m-2s)$, it is enough to show that, 
for an arbitrary point $y\in Y_R$ the sheaf 
$E=E_y$ satisfies the equality
\begin{equation}\label{dim ext=dim Y}
\dim\ext^1(E,E)=\dim Y_R=4s+\dim\calr(e,n,m).
\end{equation}
(Note that the equality \eqref{dim ext=dim Y} is beyond the scope of Theorem \ref{0dcomp}, since we cannot assume that 
$\ext^2(F,F)=0$ here). 

Indeed, let $E$ satisfy the triple \eqref{triple for E}. Then,
since $S$ is 0-dimensional and $F$ is reflexive, it follows 
that $\dim\sing(E)=\dim\sing(F)=0$, and therefore $\dim\lext^1
(E,E)=\dim\lext^1(F,E)=0.$ Thus, 
\begin{equation}\label{vanish H1}
\h^1(\lext^1(E,E))=0,\ \ \ \ \h^1(\lext^1(F,E))=0,
\end{equation}
\begin{equation}\label{vanish H2}
\h^2(\lext^1(E,E))=0.
\end{equation}
The first equality in \eqref{vanish H1} and Lemma 
\ref{Extisomixed}.(ii) yield the equality
$\ext^2(E,E)=\h^0(\lext^2(E,E))\oplus\coker d^{01}_2$
where $d^{01}_2$ is the differential $d^{01}_2:\h^0(\lext^1
(E,E))\to\h^2(\lhom(E,E))$ in the spectral sequence of 
local-to-global Ext’s. Moreover, this spectral sequence 
together with \eqref{vanish H2} yields an exact sequence
\begin{equation}\label{exact 1}
\h^0(\lext^1(E,E))\xrightarrow{d_2^{01}}\h^2(\lhom(E,E))\to
\ext^2(E,E)\to\h^0(\lext^2(E,E))\to0.
\end{equation}
Note that, since the reflexive sheaf $F$ has homological 
dimension 1, it follows that 
\begin{equation}\label{Ext2,3=0}
\lext^2(F,E)=0=\lext^3(F,E).
\end{equation}
Therefore, applying to the triple \eqref{triple for E} the
functor $\lext^2(-,E)$ we obtain
\begin{equation}
\lext^2(E,E)=\lext^3(\mathcal{O}_S,E),
\end{equation}
Here, as in \eqref{h0=}, one has $h^0(\lext^3(\mathcal{O}_S,E)
)=4s$, so that
\begin{equation}\label{h0(ext2)}
h^0(\lext^2(E,E))=4s.
\end{equation}
Next, the equality $\lext^2(F,E)=0$ (see \eqref{Ext2,3=0})
together with the second equality in \eqref{vanish H1}
yields an exact sequence similar to \eqref{exact 1}:
\begin{equation}\label{exact 2}
\h^0(\lext^1(F,E))\xrightarrow{d_2^{01}}\h^2(\lhom(F,E))\to
\ext^2(F,E)\to0.
\end{equation}
The exact sequences \eqref{exact 1} and \eqref{exact 2} fit
in a commutative diagram extending \eqref{d012-diagram}
\begin{equation} \label{d012-diagram extended}
\xymatrix{ \h^0(\lext^1(F,E)) \ar[d]\ar[r]^{d^{01}_2} & 
\h^2(\lhom(F,E)) \ar[d]^-{\simeq}\ar[r] & \ext^2(F,E) \ar[r] 
\ar[d] &  0 \ar[d] & \\
\h^0(\lext^1(E,E)) \ar[r]^{d^{01}_2} & \h^2(\lhom(E,E))\ar[r]
& \ext^2(E,E) \ar[r] & \h^0(\lext^2(E,E)) \ar[r] & 0.}
\end{equation}
Here the second vertical map is an isomorphism. Indeed, 
applying to the exact triple \eqref{triple for E} the functor
$\lhom(-,E)$ we obtain an exact sequence $0\to\lhom(F,E)\to
\lhom(E,E)\to\mathcal{A}\to0$, where $\dim\mathcal{A}\le0$ 
since $\mathcal{A}$ is a subsheaf of the sheaf $\lext^1(
\calo_S,E)$ of dimension $\le0$. Now passing to 
cohomology of the last exact triple we obtain the desired
isomorphism. The diagram \eqref{d012-diagram extended} 
together with \eqref{h0(ext2)} yields the relation
\begin{equation}\label{dim=dim+4s}
\dim\ext^2(E,E)=\dim\ext^2(F,E)+4s.
\end{equation}
Now applying to \eqref{triple for E} the functor 
$\Hom(F,-)$ we obtain the exact sequence
\begin{equation}\label{Ext2(F,E)}
\ext^1(F,\calo_S)\to\ext^2(F,E)\to\ext^2(F,F)\to\ext^2
(F,\calo_S).
\end{equation}
Since $\supp(\calo_S)\cap\sing(F)=\emptyset$ and $F$ is 
reflexive, it is easy to show that 
\begin{equation}\label{extj(F,OS)=0}
\ext^j(F,\calo_S)=0,\ \ \ \ \ j>0,
\end{equation} 
- see, e. g., \cite[Proof of Prop. 6]{JMT} for the case
$e=0$; in case $e=-1$ this argument goes on without changing.
Thus, from \eqref{Ext2(F,E)} it follows that 
$\dim\ext^2(F,E)=\dim\ext^2(F,F)$, and the relation 
\eqref{dim=dim+4s} yields 
\begin{equation}\label{dim+4s}
\dim\ext^2(E,E)=\dim\ext^2(F,F)+4s.
\end{equation}
Next, since in \eqref{triple for E} $\dim\calo_S=0$, it 
follows that $c_i(E)=c_i(F),\ i=1,2$. Thus, since both $E$ 
and $F$ are stable, Theorem \ref{dimext1} implies that
$\dim\ext^1(E,E)-\dim\ext^2(E,E)=\dim\ext^1(F,F)-
\dim\ext^2(F,F)$. Hence, by \eqref{dim+4s}
$\dim\ext^1(E,E)=\dim\ext^1(F,F)+4s$. This together with
\eqref{dim=dim Ext} implies \eqref{dim ext=dim Y}.

\vspace{2mm}
\noindent{\bf Item (ii).}
We have to show that $X_R$, where $R=\calr(e,n,m)$, is 
an open dense subset of an irreducible component $\overline{X}
_R$ of $\calm(e,n+1,m+2+e-2r-2s)$. 
Using the pointwise description \eqref{X_R as a set} of 
$X_R$, we consider similarly to \eqref{family for YR} a family
\begin{equation}\label{family for XR}
\{[E_x]\in\calm(e,n+1,m+2+e-2r-2s)\ |\ E_x=\ker(\varphi:F
\twoheadrightarrow Q_x),\ x=([F],(l,S),[\varphi_x])\in X_R\}.
\end{equation}
The rest of the argument below is parallel to that in the 
proof of statement (i) above. The difference is due to the 
fact that the triple \eqref{triple for E} is modified as
\begin{equation}\label{triple for E new}
0\to E\to F\xrightarrow{\varphi_x}Q_x\to0,\ \ \ \ \ 
Q_x=\mathcal{O}_S\oplus i_*\calo_l(r),\ \ \ \ \ F=E^{\vee
\vee},\ \ \ \ \ E=E_x.
\end{equation} 
As for the sheaf $E$ in \eqref{triple for E}, a standard 
argument for the sheaf $E$ in the last triple in view of the
stability of $F$ yields the stability of $E$, i. e., the
definition of the family \eqref{family for XR} is consistent.
This family $\{E_x\}_{x\in X_R}$ globalizes in a standard way 
to a sheaf $\mathbf{E}$ on $X_R\times\PP$ such that, for any 
$x\in X_R$, $\mathbf{E}|_{\{x\}\times\PP}\simeq E_x$. We thus 
have a natural modular morphism similar to \eqref{modular 
morphism for YR}: 
\begin{equation}\label{modular morphism for XR}
\Psi:X_R\to\calm(e,n+1,m+2+e-2r-2s),\ x\mapsto\left[\mathbf{E}|_{\{x\}\times\PP}\right].
\end{equation} 
The morphism $\Psi$ is clearly an embedding, since a point 
$\{x\}$ is recovered from $E=\ker(\varphi_x)$ as the (class 
of the) quotient map $F=E^{\vee\vee}\twoheadrightarrow
E^{\vee\vee}/E$. We therefore identify $X_R$ with its image
in $\calm(e,n+1,m+2+e-2r-2s)$. Let $\overline{X}_R$ be the 
closure of $X_R$ in $\calm(e,n,m-2s)$.

We have to prove that, for $R=\calr(e,n,m)$, the scheme 
$\overline{Y}_R$ is an irreducible rational component of 
$\calm(e,n+1,m+2+e-2r-2s)$. Here the rationality and the 
dimension of $\overline{X}_R$ are given in display 
\eqref{rational scheme2}. Since $\overline{X}_R$ is 
irreducible, to prove that $\overline{X}_R$ is an irreducible 
and generically reduced component of $\calm(e,n+1,m+2+e-2r
-2s)$, it is enough to show that, for an arbitrary point 
$x\in X_R$ the sheaf $E=E_x$ satisfies the equality
\begin{equation}\label{dim ext=dim X}
\dim\ext^1(E,E)=\dim\overline{X}_R=4s+2r+5-e+\dim\calr(e,n,m).
\end{equation}
(Remark that the equality \eqref{dim ext=dim X} is beyond the 
scope of Theorem \ref{0dcomp}, since we cannot assume that 
$\ext^2(F,F)=0$ here). 

Indeed, let $E$ satisfy the triple \eqref{triple for E new}. 
This triple and the definition \eqref{X_R as a set} of $X_R$  
yield that 
\begin{equation}\label{emptyset0}
\mathrm{Supp}(Q_x)=S\sqcup l,\ \ \ \ \ \sing(E)=\mathrm{Supp}
(Q_x)\sqcup\sing(F),\ \ \ \ \ \mathrm{i.\ e.}\ \ \ \ \ 
\mathrm{Supp}(Q_x)\cap\sing(F)=\emptyset.
\end{equation}
Hence, since $F$ is reflexive, 
\begin{equation}\label{lext1(F,E)}
\dim\lext^1(F,E)=0,\ \ \ \ \mathrm{Supp}(\lext^1(F,E))=
\sing(F),
\end{equation}
\begin{equation}\label{Exti(F,E)restr=0}
\lext^i(F,E)|_{\PP\smallsetminus\sing(F)}=0,\ \ \ i\ge1.
\end{equation}
Since $F$ is locally free along $l$ (see \eqref{emptyset0}) 
the isomorphisms $F|_l\simeq\calo_l(e)\oplus\calo_l$ and 
$\lext^2(i_*\calo_l,\op3)\simeq\calo_l(2)$ imply that
\begin{equation}\label{lext2(Ol(r),F)}
\lext^2(i_*\calo_l(r),F)\simeq\calo_l(2-r+e)\oplus\calo_
l(2-r),
\end{equation}
\begin{equation}\label{lext1,3=0}
\lext^1(i_*\calo_l(r),F)=\lext^3(i_*\calo_l(r),F)=0.
\end{equation}
By the same reason, $\lext^j(F,i_*\calo_l(r))=0,\ j>0,$ so 
that, since $r\ge3$, we have 
\begin{equation}\label{Extj(F,Ol(r))=0}
Ext^j(F,i_*\calo_l(r))=H^j(\Hom(F,i_*\calo_l(r)))\simeq
H^j(\calo_l(r-e)\oplus\calo_l(r))=0,\ \ \ j>0.
\end{equation}
Applying to the triple \eqref{triple for E new} the functor 
$\lhom(i_*\calo_l(r),-)$ and using \eqref{lext2(Ol(r),F)}, 
\eqref{lext1,3=0} and the isomorphisms 
$$ \lext^1(i_*\calo_l(r),
i_*\calo_l(r))=N_{l/\PP}\simeq2\calo_l(1) ~~{\rm and}~~ 
\lext^2(i_*\calo_l(r),i_*\calo_l(r))\simeq\det N_{l/\PP}
\simeq\calo_l(2), $$
we obtain an exact sequence
\begin{equation}\label{exact on l}
0\to2\calo_l(1)\to\lext^2(i_*\calo_l(r),E)\to\calo_l(2-r+e)
\oplus\calo_l(2-r)\to\calo_l(2)\xrightarrow{\gamma}
\lext^3(i_*\calo_l(r),E)\to0.
\end{equation}
Here one easily sees that $\mathrm{Supp}(\lext^3(i_*\calo_l
(r),E))=l$, hence $\gamma$ is an isomorphism
\begin{equation}\label{=Ol(2)}
\lext^3(i_*\calo_l(r),E)\simeq\calo_l(2),
\end{equation}
and \eqref{exact on l} yields an exact triple
$0\to2\calo_l(1)\to\lext^2(i_*\calo_l(r),E)\to\calo_l(2-r+e)
\oplus\calo_l(2-r)\to0$. Passing to cohomology of this triple
and using the condition $r\ge3$ we get
\begin{equation}\label{h1(...)=}
h^1(\lext^2(i_*\calo_l(r),E))=2r-e-6.
\end{equation}
Next, applying to the triple \eqref{triple for E new} the 
functor $\lhom(-,E)$ we obtain a long exact sequence
\begin{equation}\label{Hom(-,E)}
0\to\lhom(F,E)\to\lhom(F,E)\xrightarrow{\partial}\lext^1(Q_x,
E)\to...\to\lext^3(F,E)\to\lext^3(E,E)\to0. 
\end{equation}
Denoting $\mathcal{A}=\mathrm{im}(\partial)$ we obtain an 
exact triple
\begin{equation}\label{Hom(-,E)1}
0\to\lhom(F,E)\to\lhom(F,E)\to\mathcal{A}\to0, 
\end{equation}
Since $\mathcal{A}$ is a subsheaf of $\lext^1(Q_x,E)$ and
by \eqref{emptyset0} $\dim\lext^1(Q_x,E)\le1$, it follows 
that $\h^2(\mathcal{A})=0$, and the triple \eqref{Hom(-,E)1}
yields an epimorphism
\begin{equation}\label{h2 epi}
\h^2(\lhom(F,E))\twoheadrightarrow\h^2(\lhom(E,E)).
\end{equation}
Next, restricting the sequence \eqref{Hom(-,E)} onto
$\PP\smallsetminus\sing(F)$ and using \eqref{Exti(F,E)restr=0}
we obtain the isomorphism
\begin{equation}\label{lext(E,E)restr}
\lext^1(E,E)|_{\PP\smallsetminus\sing(F)}\simeq\lext^1
(i_*\calo_l(r),E).
\end{equation}
Since by \eqref{triple for E new} the sheaves $E$ and $F$
coincide outside $\mathrm{Supp}(Q_x)$, it follows from 
\eqref{lext(E,E)restr} and the reflexivity of $F$ that
\begin{equation}\label{lext(E,E)}
\lext^1(E,E)\simeq\lext^2(i_*\calo_l(r),E)\oplus
\lext^1(F,F),\ \ \ \ \ \ \dim\lext^1(F,F)=0.
\end{equation}
These equalities together with \eqref{h1(...)=} imply
\begin{equation}\label{h1(lext1)=}
h^1(\lext^1(E,E))=h^1(\lext^2(i_*\calo_l(r),E))=2r-e-6.
\end{equation} 
From \eqref{lext1(F,E)} and \eqref{lext(E,E)} we find
\begin{equation}\label{vanish H1 new}
\h^1(\lext^1(F,E))=\h^2(\lext^1(F,E))=0,
\end{equation}
\begin{equation}\label{vanish H2 new}
\h^2(\lext^1(E,E))=0.
\end{equation}
The spectral sequence of local-to-global Ext’s for the pair 
$(E,E)$ together with \eqref{vanish H2 new} yields the exact 
sequences
\begin{equation}\label{exact 1 new}
\h^0(\lext^1(E,E))\xrightarrow{d_2^{01}}\h^2(\lhom(E,E))\to
\coker d^{01}_2\to0,
\end{equation}
\begin{equation}\label{exact 11 new}
0\to\ker{\varepsilon}\to\ext^2(E,E)\xrightarrow{\varepsilon}
\h^0(\lext^2(E,E))\to0.
\end{equation}
\begin{equation}\label{exact 12 new}
0\to\coker d^{01}_2\to\ker{\varepsilon}\to\h^1(\lext^1(E,E))
\to0.
\end{equation}
Note that, since the sheaf $F$ is reflexive, the equalities
\eqref{Ext2,3=0} are still true, so that the rightmost part
of the long exact sequence \eqref{Hom(-,E)} yields the 
isomorphisms
\begin{equation}\label{lext2(E,E)}
\lext^2(E,E)\simeq\lext^3(Q_x,E)=\lext^3(\calo_S,E)\oplus
\lext^3(i_*\calo_l(r),E).
\end{equation}
Here, as in \eqref{h0=}, one has $h^0(\lext^3(\mathcal{O}_S,E)
=4s$, and by \eqref{=Ol(2)} we have $h^0(\lext^3(\calo_S,E))
=3$, so that \eqref{lext2(E,E)} implies
\begin{equation}\label{h0(ext2) new}
h^0(\lext^2(E,E))=4s+3.
\end{equation}
Besides, similar to \eqref{exact 1 new}-\eqref{exact 12 new}, 
the spectral sequence of local-to-global Ext’s for the pair 
$(F,E)$ together with \eqref{vanish H1 new} and the first 
equality \eqref{Ext2,3=0} yields the exact sequence
\begin{equation}\label{exact 2 new}
\h^0(\lext^1(F,E))\xrightarrow{d_2^{01}}\h^2(\lhom(F,E))\to
\ext^2(F,E)\to0,
\end{equation}
The exact sequences \eqref{exact 1 new} and \eqref{exact 2 
new} in view of \eqref{h2 epi} fit in a commutative diagram 
\begin{equation} \label{d012-diagram extended new}
\xymatrix{ \h^0(\lext^1(F,E)) \ar[d]\ar[r]^{d^{01}_2} & 
\h^2(\lhom(F,E)) \ar@{>>}[d]\ar[r] & \ext^2(F,E) \ar[r] 
\ar@{>>}[d] &  0 \ar[d] & \\
\h^0(\lext^1(E,E)) \ar[r]^{d^{01}_2} & \h^2(\lhom(E,E))\ar[r]
& \coker d^{01}_2 \ar[r] &\ 0. &}
\end{equation}
Now applying to \eqref{triple for E new} the functor 
$\Hom(F,-)$ we obtain the exact sequence
\begin{equation}\label{Ext2(F,E) new}
\ext^1(F,Q_x)\to\ext^2(F,E)\to\ext^2(F,F)\to\ext^2(F,Q_x).
\end{equation}
Recall that $Q_x=\calo_S\oplus i_*\calo_l(r)$, and the 
equalities \eqref{extj(F,OS)=0} are still true. This together
with \eqref{Extj(F,Ol(r))=0} yields $\ext^i(F,Q_x)=0,\ i=1,2,$
and \eqref{Ext2(F,E) new} implies the isomorphism. 
\begin{equation}\label{Ext2=...}
\ext^2(F,E)\simeq\ext^2(F,F).
\end{equation}
Now \eqref{exact 11 new}, \eqref{exact 12 new}, diagram
\eqref{d012-diagram extended new} and \eqref{Ext2=...} imply
the inequality
$$
\dim\ext^2(E,E)\le\dim\ext^2(F,F)+h^0(\lext^2(E,E))+
h^1(\lext^1(E,E))
$$
which in view of \eqref{h1(lext1)=} and \eqref{h0(ext2) new}
can be rewritten as
\begin{equation}\label{dim+...new}
\dim\ext^2(E,E)\le\dim\ext^2(F,F)+4s+2r-e-3.
\end{equation}
Next, since $c_1(Q_x)=0,\ c_2(Q_x)=-1$, it follows from 
\eqref{triple for E new} that $c_1(E)=c_1(F),\ i=1,2$. Thus, 
since both $E$ and $F$ are stable, Theorem \ref{dimext1} 
implies that $\dim\ext^1(E,E)-\dim\ext^2(E,E)=\dim\ext^1(F,F)-
\dim\ext^2(F,F)+8$. Hence, by \eqref{dim+...new}
$\dim\ext^1(E,E)\le\dim\ext^1(F,F)+4s+2r+5-e$. This 
inequality in view of \eqref{rational scheme2} and 
\eqref{dim=dim Ext} can be rewritten as
$$
\dim\ext^1(E,E)\le4s+2r+5-e+\dim R=\dim\overline{X}_R.
$$
On the other hand, since in \eqref{modular morphism for XR}
the modular morphism $\Psi:X_R\to\calm(e,n+1,m+2+e-2r-2s)$ is
an embedding, it follows that $\dim\ext^1(E,E)\ge\dim 
\overline{X}_R$. Thus, the last inequality is a strict 
equality, and we obtain \eqref{dim ext=dim X}.
\end{proof}

We are finally in position to give the proof of Main Theorem \ref{main2}. \label{pf mthm 2}

\vspace{2mm}
\noindent{\it Proof of Main Theorem \ref{main2}.} {\bf Item (i).}
It follows from Theorem \ref{Tmain2} (i) and the property (IV). Namely, take 
$R=\calr(-1,2n,(2n)^2)$. If $n\ge1$, then take  $s=2n^2$, so 
that, for this $s$, $Y_R$ is a rational generically reduced 
component of $\calm(-1,2n,0)$, with generic sheaf having 
0-dimensional singularities. If $n\ge3$, then for 
$(s,r)=(0,2(n^2-n-1))$, the scheme $X_R$ is a rational 
generically reduced component of $\calm(-1,2n,0)$, with 
generic sheaf having purely 1-dimensional singularities; 
b) for each pair $(s,r)$ such that $1\le s\le 2(n^2-n-1)$ and 
$r=2(n^2-n)+1-s$, the scheme $X_R$ is a rational 
generically reduced component of $\calm(-1,2n,0)$, with 
generic sheaf having singularities of mixed dimension.

\vspace{2mm}
\noindent\noindent{\bf Item (ii).}
It follows from Theorem \ref{Tmain2} (i) and the property (IV). Take $R=\calr
(0,n,n^2-n+2)$. If $n\ge1$, then for $s=\frac{n^2-n+2}{2}$, 
the scheme $Y_R$ is a rational generically reduced component 
of $\calm(-0,n,0)$, with generic sheaf having 0-dimensional 
singularities. If $n\ge3$, then for $(s,r)=(0,\frac{n(n-3)}{2}
+3)$, the scheme $X_R$ is a rational generically reduced 
component of $\calm(0,n,0)$, with generic sheaf having purely 
1-dimensional singularities. If $n\ge4$, then for each pair
$(s,r)$ such that $1\le s\le\frac{n(n-3)}{2}$ and 
$r=\frac{n(n-3)}{2}+3-s$, the scheme  $X_R$ is a rational 
generically reduced component of $\calm(0,n,0)$, with 
generic sheaf having singularities of mixed dimension.
Main Theorem \ref{main2} is proved.   ~\hfill$\Box$

\begin{Remark}
As it was shown by Le Potier in \cite{LeP1993}, the moduli 
scheme $\calm(0,2,0)$ consists of three irreducible 
components: one is the closure of locally free sheaves, while 
the other two have, as a general point, sheaves with 
0-dimensional singular locus obtained via elementary 
transformations of reflexive sheaves in $\calr(0,2,2)$ and 
$\calr(0,2,4)$ at points. While the results of this section 
do not cover these irreducible components, Le Potier has 
shown that they are also rational, via a different method. 
\end{Remark}

\section{Irreduciblility of $\calm(-1,2,4)$} 
\label{irreducible of M(-1,2,4)}

In the previous sections our results ensured the existence of 
irreducible components of the moduli spaces of torsion free 
sheaves with prescribed singularities, without focusing on 
the description of all irreducible components of the moduli 
space for given Chern classes. The aim of this 
and subsequent sections is to consider this problem for 
smallest value $c_2=2$ of the second Chern class. Namely, in 
Sections   
\ref{irreducible of M(-1,2,4)}-\ref{irreducible of M(-1,2,0)} 
we will obtain the complete characterization of the moduli 
spaces $\calm(-1,2,c_3)$ for all possible values of $c_3$. 
These results will illustrate why this study  becomes too 
complicated for large values of $c_2$.  

More precisely, in this section we will describe the 
irreducible components of the moduli spaces $\calm(-1,2,c_3)$ 
for possible values $c_3 = 0,2,4$ of the third Chern class. 
For the convenience of the reader, in the following 
proposition we will fix some numerical invariants of torsion 
free sheaves that we will use in this section.

\begin{Prop}\label{ChernClasses}
Let $E$ be a torsion free sheaf, $E^{\vee\vee}$ its double dual
and $Q_E : = E^{\vee \vee}/E$. The following holds:
\begin{itemize}
\item[(i)] $\dim Q_E\le1$ and $c_1(E^{\vee \vee}) = c_1(E)$;
\item[(ii)] if $\dim Q_E=1$ then $c_2(E^{\vee \vee}) = 
c_2(E)-\mult(Q_E)$, $c_3(E^{\vee\vee})=c_3(E)+c_3(Q_E)-c_1(E)\cdot\mult Q_E$,
where $\mult(Q_E)$ is the multiplicity of the sheaf $Q_E$;
\item[(iii)] if $\dim Q_E=0$ then $c_2(E^{\vee \vee})=c_2(E)$, 
$c_3(E^{\vee \vee})=c_3(E)+2 \cdot \mathrm{length}(Q_E)$.
\end{itemize}
If, in addition $E$ is stable with $c_1(E)=-1$, then
\begin{itemize}
\item[(iv)] $E^{\vee\vee}$ is stable;
\item[(v)] If $\dim Q_E=1$ then $c_2(E) \geq \mult Q_E\ge1$.
\end{itemize}
\end{Prop}
\begin{proof}
Since $E$ is torsion free, it fits in the following exact 
sequence:
\begin{equation}\label{fundamentalcomponents}
0 \to E \to E^{\vee \vee} \xrightarrow{\varepsilon} Q_E \to 0.
\end{equation}
The statement (i) is clear, since $E$ is torsion free. 
Therefore, computing the Chern classes we have the items (ii) 
and (iii). To show (iv), it is enough to consider the triple
$0\to A\xrightarrow{i} E^{\vee\vee}\to B\to0$ where both $A$ 
and $B$ are rank-1 torsion free sheaves with $c_1(A)+c_1(B)= 
c_1(E^{\vee\vee})=-1$. Since $\dim Q\le1$, it follows that
$\dim\mathrm{im}(\varepsilon\circ i)\le1$, where $\varepsilon$
is the epimorphism in \eqref{fundamentalcomponents}. 
Therefore, the rank 1 sheaf $A'=\ker(\varepsilon\circ i)$ 
satisfies the equality $c_1(A')=c_1(A)$. On the other hand, 
since $A'$ is a subsheaf of the stable sheaf $E$, it follows 
that $c_1(A')\le c_1(E)\le-1$. Hence, $c_1(A)\le 
c_1(E^{\vee\vee})=-1$ and $c_1(B)\ge0$, which implies that 
the reduced Hilbert polynomial 
of $A$ is less than that of $E^{\vee\vee}$, that is $E^{\vee 
\vee}$ is stable. In particular, $c_2(E^{\vee \vee}) \geq 1$,
see \cite[Cor. 3.3]{Harshorne-Reflexive}. Thus, if $\dim Q_E=1$ 
then, by (iv), $c_2(E)\ge\mult Q_E\ge1$.
\end{proof}

The next Lemma is an easy technical result that we use later in this section.
\begin{Lema}\label{nonemptyfamily} For each $F \in 
\calr(-1,1,1)$, consider the set $Y_F:=\{ l \in G(2,4);~ 
\sing(F) \subset l\}$, and the set
\begin{equation}\label{Y(r)}
Y(r) := \{(F,l,\varphi)|~ (F,l) \in \calr(-1,1,1)\times 
Y_F,~ \varphi \in  \Hom(F,i_{*}\mathcal{O}_l(r))_e / 
\mathrm{Aut}(i_{*}\mathcal{O}_l(r))\}.
\end{equation}
Then, for each $r \in \{-1,0,1\}$, the set $Y(r)$ is an 
irreducible scheme of dimension $8+2r$. In addition, the 
closure in $ \calm(-1,2,2-2r)$ of the image of the morphism 
$Y(r)\to\calm(-1,2,2-2r),\ (F,l,\varphi) 
\mapsto[\ker\varphi]$ 
is never an irreducible component of $\calm(-1,2,2-2r)$. 

\begin{proof}
For each $[F] \in \calr(-1,1,1)$, $\sing(F)$ is a unique 
point, 
so that the set $Y_F$ is a surface in the Grassmannian 
$G(2,4)$ isomorphic to $\mathbb{P}^2$. Therefore it is 
irreducible of dimension $2$. To see that the dimension of  
$\Hom(F,i_{*}\mathcal{O}_l(r)_e /\mathrm{Aut}(F,i_{*}
\mathcal{O}_l(r)\}$ is $3+2r$, apply the functor 
$\lhom(-,i_{*}\mathcal{O}_l(r))$ to the sequence 
\eqref{resR(-1,1,1)}, and recall that 
$\dim\h^0(\lext^1(F,i_{*}\mathcal{O}_l(r))) = 1$. 
Putting all these data together, we define the set 
$Y(r)$ by \eqref{Y(r)}.
By construction it is an irreducible scheme of dimension 
$8+2r$. Indeed one  has the surjective projection 
$$Y(r) \onto \calr(-1,1,1)\times Y_F ~~, ~~ ([F],l,\varphi) \mapsto ([F],l) ~~, $$
onto an irreducible scheme $\calr(-1,1,1)\times Y_F$ of 
dimension 5 (see Remark \ref{R(-1,1,1)}), with fibers 
$$ \Hom(F,i_{*}\mathcal{O}_l(r))_e/ 
\mathrm{Aut}(i_{*}\mathcal{O}_l(r)) 
\overset{\textrm{open}}{\hookrightarrow} 
\Hom(F,i_{*}\mathcal{O}_l(r))/\mathrm{Aut}(i_{*}
\mathcal{O}_l(r))$$ 
which have dimension $3+2r$.
\end{proof}
\end{Lema}
With the previous Lemma, we are already in position to 
prove the first main result of this section.

\begin{Teo}\label{M(-1,2,4)}
The moduli space $\calm(-1,2,4)$ of rank 2 stable sheaves 
on $\PP$ with Chern classes $c_1=-1,\ c_2=2,\ c_3= 4$ is 
the closure $\overline{\calr(-1,2,4)}$ of the moduli space 
$\calr(-1,2,4)$ of the rank 2 reflexive sheaves with Chern 
classes $c_1=-1,\ c_2=2,\ c_3=4$. Hence $\calm(-1,2,4)$ is irreducible, rational,
generically smooth, and of dimension 11. Moreover,
\begin{equation}\label{dim Sing=0}
\{[F]\in\calm(-1,2,4)|\ \dim\sing(F)=0\}=\calr(-1,2,4).
\end{equation}
\end{Teo}
\begin{proof}
By \cite[Thm 9.2]{Harshorne-Reflexive}, $\calr(-1,2,4)$ is 
irreducible of dimension $11$, and $\overline{\calr(-1,2,4)}$ 
is an irreducible component of $\calm(-1,2,4)$.
Consider $[E] \in \calm(-1,2,4) \setminus 
\calr(-1,2,4)$. By Proposition \ref{ChernClasses}, since 
$E^{\vee\vee}$ is stable, and either $\dim Q_E=1$ and $1\leq 
\mult Q_E \leq c_2(E)=2$, or $\dim Q_E=0$. We will study the 
possibilities for $\dim Q_E$ and $\mult Q_E$. 

i) If $\dim Q_E=1,\ \mult Q_E = 2$, then by Proposition 
\ref{ChernClasses}.b) $c_2(E^{\vee \vee}) = 0$, 
and by \cite[Thm 8.2]{Harshorne-Reflexive}, 
\begin{equation}\label{c3 le}
c_3(E^{\vee\vee}) \leq c_2(E^{\vee \vee})^2, 
\end{equation}
since $E^{\vee\vee}$ is stable. Therefore $c_3(E^{\vee 
\vee})=0$, that is $E^{\vee \vee}$ is a stable locally free 
rank 2 sheaf. Since $c_1(E^{\vee \vee})=-1,\ c_2(E^{\vee 
\vee})=0$, this contradicts to \cite[Cor. 3.5]{H1978}.

ii) If $\dim Q_E=\mult Q_E = 1$, then $c_2(E^{\vee \vee})=1$ 
and, as above, $0\le c_3(E^{\vee \vee})\le c_2^2(E^{\vee 
\vee})=1$. 
Moreover, the equality $\mult Q_E = 1$ implies that $Q_E$ is 
supported on a line, say, $l$ and it fits in an exact 
sequence of the form:
\begin{equation}\label{extensionQ_E}
0 \to Z_E\to Q_E\to i_{*}\mathcal{O}_l(r) \to 0,
\end{equation}
where $Z_E$ is the maximal $0$-dimensional subsheaf of $Q_E$ 
of length $s\ge0$, and $\mathcal{O}_{l}$ is the structure 
sheaf of the line $l$. This sequence and Proposition 
\ref{ChernClasses}.b) yield
\begin{equation} \label{chi(Q_E(t))}
\chi(Q_E(t)) = t+r+s+1,\ \ \ c_3(Q_E)=2(r+s-1),
\end{equation}
and $c_2(E^{\vee \vee})=1,\ \ \ c_3(E^{\vee\vee})=2r+2s+3
\ge0$. (Here the inequality $c_3(E^{\vee \vee})\ge0$ follows 
from \cite[Prop. 2.6]{Harshorne-Reflexive}.) Thus from 
\eqref{c3 le} we obtain $r+s=-1$, i. e.
\begin{equation}\label{c2,c3}
c_2(E^{\vee \vee})=c_3(E^{\vee \vee})=1.
\end{equation}
As $E^{\vee\vee}$ is stable, this means that 
$[E^{\vee\vee}]\in M(-1,1,1)$.
Note that, by \eqref{extensionQ_E}, there is an epimorphism 
$E^{\vee \vee} \twoheadrightarrow Q_E$, so that from 
\eqref{extensionQ_E} and the formula (\ref{splitingF}) in 
which we set $F=E^{\vee \vee}$ it follows that $r\ge-1$. This 
together with the relation $r+s=-1$ and the inequality 
$s\ge0$ shows that the only possible values for $r$ and $s$ 
are $r=-1,\ s=0$. We thus have $Q_E/Z_E=i_{*}\mathcal
{O}_l(-1)$. This together with \eqref{c2,c3} yields that, if 
$l\cap\sing(E^{\vee\vee})=\emptyset$, then, since 
$[E^{\vee\vee}]\in\calm(-1,1,1)$, it follows that $[E]$ 
belongs to the scheme $\mathcal{X}(-1,1,1,-1,0)$ defined in display 
\eqref{def calX}. Note that $\dim\mathcal{X}(-1,1,1,-1,0)=7$ 
by Remark \ref{X(-1,1,1,-1,0)}. Since by the deformation 
theory (see Theorem \ref{dimext1}) any irreducible component 
of $\calm(-1,2,4)$ has dimension at least 11, the last 
equality shows that the dimension of $\mathcal{X}(-1,1,1,-1,0)$
is too small to fill an irreducible component of 
$\calm(-1,2,4)$.

iii) If $\dim Q_E=0$, then $s=\mathrm{length}(Q_E)>0$
and, by Proposition \ref{ChernClasses}.c), $c_2(E)= 
c_2(E^{\vee\vee})=2$, $c_3(E^{\vee\vee})=c_3(E)+2s=
4+2s\ge6$. Therefore, $4=c_2^2(E^{\vee\vee})<c_3(E^{\vee\vee
})$. But this inequality contradicts the stability of 
$E^{\vee \vee}$ by \cite[Thm. 8.2(d)]{Harshorne-Reflexive}.

In conclusion, we have proved that $\calm(-1,2,4)= 
\overline{\calr(-1,2,4)}$, and the equality in display
\eqref{dim Sing=0} follows from iii) above.
The rationality of $\calr(-1,2,4)$ is known from 
\cite{Chang2}. Hence, $\calm(-1,2,4)$ is rational.
\end{proof}

As a by product of the previous proof, we obtain the 
following interesting result.

\begin{Cor}
The complement of $\calr(-1,2,4)$ in $\calm(-1,2,4)$ is 
precisely $\mathrm{X}(-1,1,1,-1,0)$.
\end{Cor}

\section{Description of families with 0-dimensional 
singularities} 
\label{irreducible of M(2)}

In this section we describe explicitly the sheaves in the 
families $\mathrm{T}(-1,2,2,1)$, $\mathrm{T}(-1,2,4,1)$ and 
$\mathrm{T}(-1,2,4,2)$. This description will be used later 
in the study of irreducible components of the moduli spaces 
$\mathcal{M} (-1,2,c_3)$ for $c_3=2$ and $c_3=4$. Everywhere 
below for a coherent sheaf $F$ on a given scheme $X$ 
we denote by $\mathbf{P}(F)$ the projective spectrum of the 
symmetric algebra $\mathrm{Sym}_{\calo_X}(F)$. Besides, as 
before, for any point $p\in\PP$ we denote $A_p=\mathrm{Aut}
(\calo_p)\simeq\mathbf{k}$.

We start with the following theorem describing, for $i=1$ 
and $i=2$, the irreducible families $\mathrm{T}(-1,2,2i,1)$ 
defined in Section \ref{New Irreducible Components} as 
the closures in 
$\calm(-1,2,2i-2)$ of their open subsets $\mathcal{T}
(-1,2,2i,1)$. For these $i$, consider the moduli spaces 
$R_i:=\calr(-1,2,2i)$ and the universal $\calo_{\PP\times 
R_i}$-sheaves $\mathbf{F}_i$, respectively.

\begin{Teo}\label{T(-1,2,2i,1)}
The scheme $\mathrm{T}(-1,2,2i,1)$, for $i\in\{1,2\}$, is an irreducible 
15-dimensional component of $\calm(-1,2,2i-2)$. This component 
contains an open subset of $\mathrm{T}(-1,2,2i,1)$, isomorphic 
to $\mathbf{P}(\mathbf{F}_i)$, which consists of all the points 
$[E]\in\calm(-1,2,2i-2)$ such that $E^{\vee\vee}/E$ is a 
0-dimensional scheme of length 1. This subset $\mathbf{P}
(\mathbf{F}_i)$ contains the open subset $\mathcal{T}
(-1,2,2i,1)$.

\end{Teo}
\begin{proof}
Let $i\in\{1,2\}$. For any point $y\in R_i$ we denote 
$F_{i,y}=\mathbf{F}_i|_{\PP\times\{y\}}$. By \cite[Lemma 
4.5]{St} $\mathbf{P}(F_{i,y})$ is an irreducible 
4-dimensional scheme for any $y\in R_i$. Hence, since 
$\dim R_i=11$, $i=1,2,$ it follows that $\mathbf{P}
(\mathbf{F}_i)$ is an irreducible 15-dimensional 
scheme. Consider the structure morphisms $\pi_i:
\mathbf{P} (\mathbf{F}_i)\to\PP\times R_i$ and the 
compositions $\theta_i=pr_1\circ\pi_i:
\mathbf{P}(\mathbf{F}_i)\to\PP$. By the functorial 
property of projective spectra \cite[Ch. II, Prop. 
7.12]{H} we have for $i=1,2$:
\begin{equation}\label{descrn of P(Fi)}
\mathbf{P}(\mathbf{F}_i)=\{z=(p,[F_i],[\psi]=\psi\ 
\mathrm{mod}A_p)|\ (p,[F_i])=\pi_i(z),\ \psi:\ F_i
\to\calo_p\ \mathrm{is\ an\ epimorphism}\}.
\end{equation}
Hence, each point $z=(p,[F_i],[\psi])\in
\mathbf{P}(\mathbf{F}_i)$ defines an exact triple
\begin{equation}\label{triple in T(-1,2,2i,1)}
0\to E_i\to F_i\xrightarrow{\psi}\calo_p\to0,\ \ \ \ \ \ 
[E_i=E_{i,z}:=\ker\psi]\in\calm(-1,2,2i-2),\ \ \ 
F_i=E_{i,z}^{\vee\vee},\ \ \ \ \ \ i=1,2.
\end{equation}
This triple is globalized to an $\calo_{\PP\times R_i
}$-triple in the following way. Namely, let $\mathbf{
\tilde{F}}_i=\mathbf{F}_i\otimes_{\calo_{R_i}}\calo_{
\mathbf{P}(\mathbf{F}_i)}$ and consider the "diagonal" 
embedding $j:\mathbf{P}(\mathbf{F}_i)\hookrightarrow\PP
\times\mathbf{P}(\mathbf{F}_i),\ z\mapsto(\theta_i(z),z)$. 
By construction, $j^*\mathbf{\tilde{F}}_i=\pi_i^*
\mathbf{F}_i$ and we obtain the composition of surjections 
$\boldsymbol{\psi}:\ \mathbf{\tilde{F}}_i
\twoheadrightarrow j_*j^*\mathbf{\tilde{F}}_i
=j_*\pi_i^*\mathbf{F}_i\twoheadrightarrow j_*
\calo_{\mathbf{P}(\mathbf{F}_i)}(1)$ which yields an exact 
$\calo_{\PP\times\mathbf{P}(\mathbf{F}_i)}$-
triple, where $\mathbf{E}_i:=\ker\boldsymbol{\psi}$:
\begin{equation}\label{bold Ei}
0\to\mathbf{E}_i\to\mathbf{\tilde{F}}_i\xrightarrow
{\boldsymbol{\psi}}j_*\calo_{\mathbf{P}(\mathbf{F}_i)}(1)\to0,
\ \ \ \ \ \ i=1,2.
\end{equation}
By construction, the sheaves in this triple are flat over 
$R_i$, 
hence its restriction onto $\PP\times\{z\}$ for any $z=(p,
[F_i],[\psi])\in\mathbf{P}(\mathbf{F}_i)$ yields the triple 
\eqref{triple in T(-1,2,2i,1)} with $E_{i,z}=
\mathbf{E}_i|_{\PP\times\{z\}}$.
Thus we obtain the modular morphism 
\begin{equation}\label{modular fi}
f_i:\ \mathbf{P}(\mathbf{F}_i)\to\calm(-1,2,2i-2),\ \ \ 
z\mapsto E_{i,z},\ \ \ \ \ \ \ i=1,2.
\end{equation}
This morphism is clearly an embedding, since the data
$([F_i],p,[\psi])$ in the triple 
\eqref{triple in T(-1,2,2i,1)} are uniquely recovered from 
the point $[E=E_{i,z}]\in\calm(-1,2,2i-2)$; namely, $F_i:=
E^{\vee\vee},\ p:=\supp(Q_E)$, where$Q_E:=
E^{\vee\vee}/E\simeq\calo_p$ since $\mathrm{length}~Q_E=1$ 
and $\psi:F_i\twoheadrightarrow \calo_p$ is the quotient 
epimorphism. We therefore identify $\mathbf{P}(\mathbf{F}_i)
$ with its image under the morphism $f_i$.

Last, under the description \eqref{descrn of P(Fi)} of 
$\mathbf{P}(\mathbf{F}_i)$ we have, by the definition of
$\mathcal{T}(-1,2,2i,1)$, that $\mathcal{T}(-1,2,2i,1)
=\{z=(p,[F_i],[\psi])\in\mathbf{P}(\mathbf{F}_i)|\ 
p\not\in\sing(F_i)\}$ is an open subset of 
$\mathbf{P}(\mathbf{F}_i)$ which is dense since $\mathbf{P}
(\mathbf{F}_i)$ is irreducible. Hence, by definition, its 
closure
in $\calm(-1,2,2i-2)$ coincides with $\mathrm{T}(-1,2,2i,1)$.
In addition, it is an irreducible component of $\calm(-1,2,
2i-2)$ by Theorem \ref{0dcomp}.
\end{proof}

Let us introduce one more piece of notation. For any 
$y\in R_2$, let $F_y:=\mathbf{F}_2|_{\PP\times\{y\}}$ and
let $pr_2:\PP\times R_2\to R_2$ be the projection.
Besides, for an arbitrary $\calo_{\PP\times R_2}$-sheaf $A$
and an integer $m\in\mathbb{Z}$ let $A(m):=A\otimes(\op3(m)
\boxtimes\calo_{R_2})$. The following remark will be 
important below in the study of the scheme $\mathbf{P}
(\mathbf{E}_2)$ for the $\calo_{\PP\times\mathbf{P}
(\mathbf{F}_2)}$-sheaf $\mathbf{E}_2$ defined in \eqref{bold 
Ei} for $i=2$.
\begin{Remark}\label{resolution for F2}
From \cite[Lemma 9.6 and Proof of Lemma 9.3]
{Harshorne-Reflexive} it follows that, for any $y\in R_2$, 
the sheaf $F_y$ fits in an exact triple
\begin{equation}\label{triple for Fy}
0\to\op3(-3)\to2\cdot\op3(-1)\oplus\op3(-2)
\xrightarrow{\phi} F_y\to0.
\end{equation}
This triple clearly globalizes to a locally free 
$\calo_{\PP\times R_2}$-resolution of the universal sheaf 
$\mathbf{F}_2$:
\begin{equation}\label{triple for F}
0\to\mathbf{L}_2\to\mathbf{L}_1\xrightarrow{\Phi}\mathbf{F}
\to0,\ \ \ \ \ \  \rk\mathbf{L}_1=3,\ \ \ \rk\mathbf{L}_2=1.
\end{equation}
explicitly, $\mathbf{L}_1$ fits in the exact triple
$0\to\op3(-1)\boxtimes M_0\to\mathbf{L}_1\to\op3(-2)
\boxtimes M_1\to0$ and $\mathbf{L}_2=\op3(-3)\boxtimes M_2$,
where $M_0,\ M_1,\ M_2$ are locally free 
$\calo_{R_2}$-sheaves 
of ranks 2, 1, 1, respectively, which are determined by 
$\mathbf{F}_2$ as: $M_0=pr_{2*}(\mathbf{F}_2(1))$,
$M_1=pr_{2*}(\mathbf{F}_2(2))/pr_{2*}(\mathrm{im}(ev))$, where
$ev:\op3(1)\boxtimes M_0\to\mathbf{F}_2(2)$ is the evaluation
morphism, and $M_2=\ker(pr_{2*}\mathbf{L}_1(3)
\xrightarrow{pr_{2*}\Phi}pr_{2*}\mathbf{F}_2(3))$.
\end{Remark}

Consider the structure morphism $\pi_2:\mathbf{P}
(\mathbf{F}_2)\to\PP\times R_2$. Note that the triple 
\eqref{triple for Fy} immediately yields that $\pi_2^{-1}
(p,y)$ equals to $\mathbb{P}^1$ if $p\not\in\sing(F_y
)$, respectively, equals $\mathbb{P}^2$ if $p\in\mathrm
{Sing}(F_y)$. As $\mathrm{codim}(\sing(F_y),\PP)=3$, 
it follows by the definition of $\mathcal{T}(-1,2,2i,1)$ that
\begin{equation}\label{codim-tau}
\mathrm{codim}_{\mathbf{P}(\mathbf{F}_2)}
(\mathbf{P}(\mathbf{F}_2)\smallsetminus\mathcal{T}
(-1,2,4,1))=2.
\end{equation}

Now proceed to the study of the scheme $\mathbf{P} 
(\mathbf{E}_2)$ endowed with the structure morphism $\pi:
\mathbf{P}(\mathbf{E}_2)\to\PP\times\mathbf{P}(\mathbf{F}_2)$ 
and consider the composition $\tau=pr_1\circ\pi:\mathbf{P}
(\mathbf{E}_2)\to\PP$. Similarly to \eqref{descrn of P(Fi)}, 
in view of the functorial property of projective spectra 
\cite[Ch. II, Prop. 7.12]{H} we obtain the following 
description of the scheme $\mathbf{P}(\mathbf{E}_2)$:
\begin{equation}\label{descrn of P(E2)}
\mathbf{P}(\mathbf{E}_2)=\{w=(q,[E_2],[\varphi]=\varphi\ 
\mathrm{mod}A_q)|\ (q,[E_2])=\pi(w),\ \varphi:\ 
E_2\to\calo_q\ \mathrm{is\ an\ epimorphism}\}.
\end{equation}
It follows now that each point $w=(q,[E_2],[\varphi])\in
\mathbf{P}(\mathbf{E}_2)$ defines an exact triple
\begin{equation}\label{triple in T(-1,2,4,2)}
0\to E_w\to E_2\xrightarrow{\varphi}\calo_q\to0,\ \ \ \ \ \ 
[E_w:=\ker\varphi]\in\calm(-1,2,0).
\end{equation}
This triple is globalized to an $\calo_{\PP\times\mathbf{P}
(\mathbf{F}_2)}$-triple which is constructed completely 
similar to the triples \eqref{bold Ei}. Namely, let 
$\mathbf{\tilde{E}}_2
=\mathbf{E}_2\otimes_{\calo_{\mathbf{P}(\mathbf{F}_2)}}\calo_{
\mathbf{P}(\mathbf{E}_2)}$ and consider the "diagonal" 
embedding $j:\mathbf{P}(\mathbf{E_2})\hookrightarrow\PP\times
\mathbf{P}(\mathbf{E}_2),\ w\mapsto(\tau(w),w)$. Then 
$j^*\mathbf{\tilde{E}}_2=\pi^*\mathbf{E}_2$ and we obtain 
the composition of surjections $\boldsymbol{\varphi}:\ 
\mathbf{\tilde{E}}_2
\twoheadrightarrow j_*j^*\mathbf{\tilde{E}}_2
=j_*\pi^*\mathbf{E}_2\twoheadrightarrow j_*\calo_{\mathbf{P}
(\mathbf{E}_2)}(1)$ which yields an exact $\calo_{\PP\times
\mathbf{P}(\mathbf{E}_2)}$-triple, where $\mathbf{E}:=
\ker\boldsymbol{\varphi}$:
\begin{equation}\label{bf E}
0\to\mathbf{E}\to\mathbf{\tilde{E}}_2\xrightarrow
{\boldsymbol{\varphi}}j_*\calo_{\mathbf{P}(\mathbf{E}_2)}(1)
\to0.
\end{equation}
By construction, the restriction of this triple onto 
$\PP\times\{w\}$ for any $w=(p,[E_2],[\varphi])\in\mathbf{P}
(\mathbf{E}_2)$ yields the triple \eqref{triple in 
T(-1,2,4,2)}, where $E_w=\mathbf{E}|_{\PP\times
\{w\}}$ and where $[E_2]\in\mathbf{P}(\mathbf{F}_2)$ by 
Theorem \ref{T(-1,2,2i,1)},(ii) fits in the triple 
\eqref{triple in T(-1,2,2i,1)} for $i=2$: $0\to E_2\to 
F_2\xrightarrow{\psi}\calo_p\to0$, $F_2=E_2^{\vee\vee}$. 
Combining this triple with \eqref{triple in T(-1,2,4,2)}, we 
obtain the equality $F_2=E_w^{\vee\vee}$ and two exact 
triples, where $E=E_w$:
\begin{equation}\label{QE of length 2}
0\to E\to E^{\vee\vee}\to Q_E\to0,\ \ \ \ \ \ \ \ \ \ 
0\to\calo_q\to Q_E\to\calo_p\to0.
\end{equation} 
Besides, we have a modular morphism 
$f:\ \mathbf{P}(\mathbf{E}_2)\to\calm(-1,2,0),\ \ \ 
w\mapsto [E_w]$. From \eqref{QE of length 2} and the 
definition of the family $\mathcal{T}(-1,2,4,2)$ given in 
Theorem \ref{0dcomp} it follows that
\begin{equation}\label{open tau(-1,2,4,2)}
\mathcal{T}(-1,2,4,2)=\{[E]\in f(\mathbf{P}(\mathbf{E}_2))\ 
|\ \supp(Q_E)=p\sqcup q,\ \supp(Q_E)\cap
\sing(E^{\vee\vee})=\emptyset \}.
\end{equation}
\begin{Teo}\label{T(-1,2,4,2)}
The scheme $\mathrm{T}(-1,2,4,2)$ is an irreducible 
19-dimensional component of $\calm(-1,2,0)$. This component 
contains a dense subset, isomorphic to $f(\mathbf{P}
(\mathbf{E}_2))$, which consists of all the points 
$[E]\in\calm(-1,2,2)$ such that $E^{\vee\vee}/E$ is a 
0-dimensional scheme of length 2. This subset $f(\mathbf{P}
(\mathbf{E}_2))$ contains $\mathcal{T} (-1,2,4,2)$ as the 
dense open subset described in \eqref{open tau(-1,2,4,2)}.
\end{Teo}
\begin{proof}
We have to prove the irreducibility of $\mathbf{P}
(\mathbf{E}_2)$. Since $\mathbf{P}(\mathbf{F}_2)$ is 
irreducible, it is enough to prove that, for an arbitrary 
point $z=(p,[F_2],[\psi])\in\mathbf{P}(\mathbf{F}_2)$, 
the fiber $p_E^{-1}(z)$ of the composition $p_E:\mathbf{P}
(\mathbf{E}_2)\xrightarrow{\pi}\PP\times\mathbf{P}(\mathbf{F}_
2)\xrightarrow{pr_2}\mathbf{P}(\mathbf{F}_2)$ is irreducible 
of dimension 4. Note that the sheaves $F_2$ and 
$E_2=\mathbf{E}_2|_{\PP\times\{z\}}$ fit in the exact triple 
\eqref{triple in T(-1,2,2i,1)} for $i=2$. Besides, $F_2$ fits in the 
exact triple \eqref{triple for Fy} in which we set $F_y=
F_2$. These two triples are included in a commutative diagram
\begin{equation}\label{comm diagram0}
\xymatrix{&  & 0 \ar[d]  & 0 \ar[d]  & \\
0\ar[r]& \op3(-3)\ar@{=}[d] \ar[r] & \mathcal{G} \ar[d]\ar[r] 
& E_2\ar[d]\ar[r] & 0\\
0\ar[r]& \op3(-3) \ar[r] & 2\cdot\op3(-1)\oplus\op3(-2) 
\ar[r]^-{\phi} \ar[d]^{\lambda} & F_2\ar[r]\ar[d]^{\psi} & 0\\
& &\mathcal{O}_p\ar@{=}[r]\ar[d] & \mathcal{O}_p \ar[d] & & \\
& & 0 & 0, &  & }
\end{equation}
where $\lambda:=\psi\circ\phi$ and $\mathcal{G}=\ker
(\lambda)$. 
Here, the surjection $\lambda$ induces an embedding of a 
point 
$$ {}^{\sharp}\lambda:\ w=\mathbf{P}(\calo_p)
\hookrightarrow W:=\mathbf{P}(2\cdot\op3(-1)\oplus\op3(-2)), $$ 
and from standard properties of projective spectra it
follows that $\mathbf{P}(\mathcal{G})$ is a small 
birational modification of $W$. More precisely, this 
modification as the composition 
of the blowing up $\sigma_w$ of $W$ at the point $w$ and the 
contraction of  the proper preimage of the fiber
$\pi_W^{-1}(p,z)$ under $\sigma_w$, where $\pi_W:W\to\PP$ is
the structure morphism. In particular,$\mathbf{P}(\mathcal{G}
)$ is an irreducible projective scheme of dimension 
$\dim\mathbf{P}(\mathcal{G})=5$. By the same reason, from the 
rightmost vertical triple of \eqref{comm diagram0} it follows 
that, if $p\not\in\sing(F_2)$, the scheme 
$\mathbf{P}(E_2)$ is a small birational modification of 
$\mathbf{P}(F_2)$. Namely, this modification  
is the composition of the blowing up $\sigma_p$ of 
$\mathbf{P}(F_2)$ at its smooth point $\mathbf{P}(\calo_p)$, 
and the contaraction of the proper preimage of the fiber
$\pi^{-1}(p,z)$ under $\sigma_p$. Therefore, since by
\cite[Lemma 4.5]{St} $\mathbf{P}(F_2)$ is irreducible, 
$\mathbf{P}(E_2)$ is an irreducible scheme of dimension 
$\dim\mathbf{P}(E_2)=4$, if $p\not\in\sing(F_2)$,
i. e. when $z\in\mathcal{T}(-1,2,4,1)$. 
This implies that the scheme $\mathbf{P}(\mathbf{E}_2)_0
:=p_E^{-1}(\mathcal{T}(-1,2,4,1))$ is irreducible of 
dimension 19, since by Theorem \ref{0dcomp} $\mathcal{T}
(-1,2,4,1)$ is irreducible of dimension 15. 

Next, an easy computation with the diagram \eqref{comm 
diagram0} 
yields: $(E_2\otimes\calo_p)^{\vee}\subset
(\mathcal{G}\otimes\calo_p)^{\vee}=\mathbf{k}^5$, hence
\begin{equation}\label{fibre  in P4}
\pi_E^{-1}(z,p)=\mathbb{P}((E_2\otimes\calo_p)^{\vee})\subset
\mathbb{P}((\mathcal{G}\otimes\calo_p)^{\vee})=\mathbb{P}^4.
\end{equation}
Now, acccording to Remark \ref{resolution for F2}, the 
middle horizontal triple in \eqref{comm diagram0} globalizes 
to the exact $\calo_{\PP\times\mathbf{P}
(\mathbf{F}_2)}$-triple $0\to\mathbf{\tilde{L}}_2
\to\mathbf{\tilde{L}}_1\to\mathbf{\tilde{F}}_2\to0$ 
obtained by lifting the exact triple \eqref{triple for F} 
from $\PP\times R_2$ onto $\PP\times\mathbf{P}
(\mathbf{F}_2)$. Similarly, the rightmost vertical, the 
middle vertical and the upper horizontal triples in 
\eqref{comm diagram0} globalize, respectively, to the triple 
\eqref{bold Ei} for $i=2$, the triple $0\to\mathbf{G}\to
\mathbf{\tilde{L}}_1\to j_*\calo_{\mathbf{P}(\mathbf{F})}(1)
\to0$ and the triple
\begin{equation}\label{triple for G,E}
0\to\mathbf{\tilde{L}}_2\to\mathbf{G}\to\mathbf{E}\to0,
\end{equation}
where $\mathbf{G}$ is an $\calo_{\PP\times\mathbf{P}
(\mathbf{F}_2)}$-sheaf such that $\mathbf{G}|_{\{z\}
\times\PP}=\mathcal{G}$.
Consider the composition
$$ p_G:\mathbf{P}(\mathbf{G})
\xrightarrow{\pi_G}\mathbf{P}(\mathbf{F})\times\PP
\xrightarrow{pr_1}\mathbf{P}(\mathbf{F}) $$
where $\pi_G$ is 
the structure morphism of $\mathbf{P}(\mathbf{G})$. Note 
that the sheaf $\mathbf{\tilde{L}}_2$ in the triple 
\eqref{triple for G,E} is invertible, hence this triple 
shows that $\mathbf{P}(\mathbf{E})$ is a Cartier divisor 
in $\mathbf{P}(\mathbf{G})$ defined as the zero-set of 
some section $0\ne s\in\h^0(\calo_{\mathbf{P}(\mathbf{G})}(1)
\otimes p_G^*\mathbf{\tilde{L}}_2^{\vee})$. On the other 
hand, the fibers $p_G^{-1}(z)=\mathbf{P}(\mathcal{G})$ of 
$p_G$ are irreducible projective schemes, that is $p_G$ is a 
projective morphism with irreducible 5-dimensional fibers 
over the irreducible 15-dimensional scheme $\mathbf{P}
(\mathbf{F}_2)$. It follows that, if $\mathbf{P}
(\mathbf{E}_2)$ is reducible, then any its irreducible 
component $U$ has dimension
\begin{equation}\label{dim U}
\dim U=\dim\mathbf{P}(\mathbf{G})-1=19.
\end{equation}
Note that, for any $z=(p,[F_2],\psi\ \mathrm{mod}A_p)\in
\mathbf{P}(\mathbf{F}_2)$, we have $F_2|_{\PP\smallsetminus
\{p\}}=E_2|_{\PP\smallsetminus\{p\}}$, hence, since $\dim
\mathbf{P}(F_2)=4$, $\mathbf{P}(E_2|_{\PP\smallsetminus\{p\}}
)$ is a 4-dimensional scheme. On the other hand, by 
definition,
$$ \mathbf{P}(E_2)=p_E^{-1}(z)=\pi_E^{-1}(\{z\}
\times\PP)=\mathbf{P}(E_2|_{\PP\smallsetminus\{p\}})\cup
\pi_E^{-1}(z,p).$$
Hence, by \eqref{fibre  in P4}, for any
$z\in\mathbf{P}(\mathbf{F}_2)$, we have
\begin{equation}\label{dim fibre}
\dim\mathbf{P}(E_2)=4.
\end{equation}
This together with \eqref{codim-tau} implies that 
$\mathbf{P}(\mathbf{E}_2)\smallsetminus
\mathbf{P}(\mathbf{E}_2)_0$ has codimension 2 in $\mathbf{P}
(\mathbf{E}_2)$. Therefore, by \eqref{dim U}, $\mathbf{P}
(\mathbf{E}_2)$ is irreducible and contains $\mathbf{P}
(\mathbf{E}_2)_0$ as a dense open subset. 

Finally, remark that, by the description given in display \eqref{open tau(-1,2,4,2)}, the set $\mathcal{T}(-1,2,4,2)$ is a nonempty open subset of 
$\mathbf{P}(\mathbf{E}_2)_0$, hence it is dense in 
$\mathbf{P}(\mathbf{E}_2)$.  
\end{proof}

\section{Description of families with mixed 
singularities}\label{descr of X}

We now proceed to the description of the sets $X(-1,1,1,-1,1
)$ and $X(-1,1,1,0,1)$. Our aim is to construct explicitly 
certain open dense subsets of them, together with a universal 
family of sheaves over these subsets, which will be used in 
our further results. We start with the following lemma.
\begin{Lema}\label{two irred}
Let $G:=G(2,4)$ be the Grassmannian of lines in $\PP$ and $M=\calm(-1,1,1)\simeq\PP$ (see Remark 
\ref{R(-1,1,1)}). Consider $[F]\in M$, and let $\mathcal{E}$ 
and $E$ be the sheaves on $\PP$ fitting in the exact triples
\begin{equation}\label{triple F, calE}
0\to\mathcal{E}\to F\xrightarrow{\varepsilon}\calo_l(-1)\to0,
\end{equation}
\begin{equation}\label{triple calE,E}
0\to E\to\mathcal{E}\xrightarrow{\gamma}\calo_p\to0,
\end{equation}
for some line $l\in G$ and some point $p\in\PP$. Then 
$\mathbf{P}(\mathcal{E})$ and $\mathbf{P}(E)$ are irreducible 
generically smooth schemes of dimension 4.
\end{Lema}
\begin{proof}
We first show that $\mathcal{E}$ fits in the exact triple:
\begin{equation}\label{resoln calE}
0\to\op3(-3)\to\op3(-2)\oplus2\cdot\op3(-1)\xrightarrow{e}
\mathcal{E}\to0.
\end{equation}
Let $x_0:=\sing(F)$ (see Remark \ref{R(-1,1,1)}). 
Consider the two possible cases (a) $x_0\in l$ and (b) 
$x_0\not\in l$.\\
Case (a): $x_0\in l$. Note that from the definition of the 
sheaf $F$ it follows easily that $F$ fits in the exact triple 
$0\to\op3
(-1)\to F\xrightarrow{\delta}\cali_{l,\PP}\to0$. Since $\cali_
{l,\PP}\otimes\calo_l=N^{\vee}_{l/\PP}\simeq2\cdot\calo_l(-1)
$, it follows that there exists an epimorphism 
$\beta:\cali_{l,\PP}
\twoheadrightarrow\calo_l(-1)$ such that $\beta\circ\delta=
\varepsilon$, where $\varepsilon$ is the epimorphism in 
\eqref{triple F, calE}. Besides,$\ker\beta\simeq
\cali_{C,\PP}$, where $C$ is a nonreduced conic supported on 
$l$, and we obtain exact triples
$$
0\to\op3(-1)\to\mathcal{E}\to\cali_{C,\PP}\to0,\ \ \ \ \ \ 
0\to\op3(-3)\to\op3(-1)\oplus\op3(-2)\to\cali_{C,\PP}\to0.
$$
These two triples yield the resolution \eqref{resoln calE} by 
push-out, since $\ext^1(\op3(-1)\oplus\op3(-2),\op3(-1))=0$.\\
Case(b): $x_0\not\in l$. Note that, by Remark 
\eqref{R(-1,1,1)},
$F$ fits in the exact triple $0\to\op3(-2)\to3\cdot\op3(-1)
\xrightarrow{\alpha}F\to0$. Since clearly $\ker(\varepsilon
\circ\alpha:3\cdot\op3(-1)\twoheadrightarrow\calo_l(-1))\cong
2\cdot\op3(-1)\oplus\cali_{l,\PP}(-1)$, the last triple 
together with the triple \eqref{triple F, calE} yields the 
exact triple
$$
0\to\op3(-3)\xrightarrow{i}2\cdot\op3(-1)\oplus\cali_{l,\PP}
(-1)\to\mathcal{E}\to0.
$$
Let $c:2\cdot\op3(-1)\oplus\cali_{l,\PP}(-1)\twoheadrightarrow
\cali_{l,\PP}(-1)$ be the canonical epimorphism and consider
the composition $c\circ i:\op3(-2)\to\cali_{l,\PP}(-1)$. If 
this composition is the zero map, then $\mathrm{im}(i)\subset2
\cdot\op3(-1)$ and $\coker(i)\subset\mathcal{E}$. Since 
$\mathcal{E}$ is torsion free, it follows that $\coker(i)=
\cali_{m,\PP}$ for some line $m$ distinct from $l$, and 
$\mathcal{E}$ fits in the exact triple $0\to\cali_{m,\PP}
\to\mathcal{E}\to\cali_{l,\PP}(-1)\to0$. This triple implies 
that $m\subset\sing(\mathcal{E})$, contrary to the 
evident equality $\sing(\mathcal{E})=x_0\sqcup l$. 
Hence the composition $c\circ i$ is a nonzero morphism, so 
that $\coker(c\circ i)\cong\calo_{\mathbb{P}^2}(-2)$ for some 
projective plane $\mathbb{P}^2$ in $\PP$. We thus obtain an 
exact triple $0\to2\cdot\op3(-1)\to\mathcal{E}\to
\calo_{\mathbb{P}^2}(-2)\to0$. This triple and and the exact 
triple $0\to\op3(-3)\to\op3(-2)\to\calo_{\mathbb{P}^2}(-2)
\to0$ by push-out yield \eqref{resoln calE}, since 
$\ext^1(\op3(-2),2\cdot\op3(-1))=0$. 

Now from \eqref{resoln calE} it follows that $\mathbf{P}
(\mathcal{E})$ is a Cartier divisor in $W:=\mathbf{P}(\op3(-2)
\oplus2\cdot\op3(-1)$, and the same argument as in the proof 
of Theorem \ref{T(-1,2,4,2)} shows that $\mathbf{P}(\mathcal{E}
)$ is irreducible. Next, the triples \eqref{triple calE,E} 
and \eqref{resoln calE} yield exact triples
$$
0\to\cdot\op3(-3)\to\mathcal{G}\to E\to0,\ \ \ \ \ \ 
 0\to\mathcal{G}\to\op3(-2)\oplus2\cdot\op3(-1)
\xrightarrow{\gamma\circ e}\calo_p\to0.
$$
The second triple here shows that $\mathbf{\mathcal{G}}$ is 
irreducible as a small birational modification of the scheme 
$W$ defined above, hence it is irreducible. On the other 
hand, the first triple shows that $\mathbf{P}(E)$ is a 
Cartier divisor in $\mathbf{\mathcal{G}}$, and again the same 
argument as in the proof of Theorem \ref{T(-1,2,4,2)} yields 
the irreducibility of $\mathbf{P}(E)$. \end{proof}

Now, let  and  $\Gamma=\{(x,l)\in\PP\times G\ |\ x\in l\}$ 
the graph of incidence, and
$\calo_{\PP\times M}$-sheaf (see Remark \ref{R(-1,1,1)}). 
for $l\in G$ denote $A_l:=\mathrm{Aut}(\calo_l(-1)),\ 
A'_l:=\mathrm{Aut}(\calo_l),\ A_l\simeq\mathbf{k}^*\simeq 
A'_l$. Define the sets
\begin{equation}\label{B}
B:=\{(l,[F],\epsilon\ \mathrm{mod}A_l))\ |\ (l,[F])\in 
G\times M,\ \epsilon:F\to\calo_l(-1)\ \mathrm{is\ an\ 
epimorphism}\}.
\end{equation}
\begin{equation}\label{B'}
B':=\{(l,[F],\epsilon'\ \mathrm{mod}A'_l))\ |\ (l,[F])\in G
\times M,\ \epsilon':F\to\calo_l\ \mathrm{is\ an\ 
epimorphism}\}.
\end{equation}
We have the following proposition.

\begin{Prop}\label{descriptn of B,B'} The following claims 
are true.
\begin{itemize}
\item[(i)] $B$, respectively, $B'$ is the set of closed points of an 
irreducible scheme of dimension 7, respectively, of dimension 9.
\item[(ii)] There is an $\calo_{\PP\times B}$-sheaf $\boldsymbol
{\mathcal{E}}$ and an invertible 
$\calo_{\mathbf{\Gamma}}$-sheaf 
$\mathbf{L}$ fitting in the exact triple
$0\to\boldsymbol{\mathcal{E}}\to\mathbf{F}_B
\xrightarrow{\varepsilon}\mathbf{L}\to0$, where
$\mathbf{F}_B=\mathbf{F}\underset{\calo_M}{\otimes}\calo_B$ 
and $\mathbf{\Gamma}=\Gamma\times_MB$.
Respectively, there is an $\calo_{\PP\times B'}$-sheaf 
$\boldsymbol{\mathcal{E'}}$ and an invertible $\calo_{\mathbf
{\Gamma'}}$-sheaf $\mathbf{L'}$ fitting in the exact triple
$0\to\boldsymbol{\mathcal{E'}}\to\mathbf{F}_{B'}
\xrightarrow{\varepsilon}\mathbf{L'}\to0$, where
$\mathbf{F}_{B'}=\mathbf{F}\underset{\calo_M}{\otimes}\calo_{B
'}$ and $\mathbf{\Gamma'}=\Gamma\times_M{B'}$. These triples, 
being restricted onto $\PP\times\{b\}$, respectively, onto 
$\PP\times\{b'\}$ for any points $b=(l,[F],\epsilon\ 
\mathrm{mod}A_l)\in B$, $b'=(l,[F],\epsilon'\ 
\mathrm{mod}A'_l)\in B$, yield: 
\begin{equation}\label{triple with cal E}
0\to\mathcal{E}_b\xrightarrow{\iota}F\xrightarrow{\epsilon}
\calo_l(-1)\to0,\ \ \ \ \ \ 
\mathcal{E}_b\cong\boldsymbol{\mathcal{E}}|_{\PP\times\{b\}}.
\end{equation}
\begin{equation}\label{triple with cal E'}
0\to\mathcal{E'}_{b'}\xrightarrow{\iota'}F\xrightarrow
{\epsilon'}\calo_l\to0,\ \ \ \ \ \ 
\mathcal{E'}_{b'}\cong\boldsymbol
{\mathcal{E'}}|_{\PP\times\{b'\}}.
\end{equation}
\item[(iii)] $\mathbf{P}(\boldsymbol{\mathcal{E}})$, respectively, 
$\mathbf{P}(\boldsymbol{\mathcal{E'}})$ is an irreducible 
generically smooth scheme of dimension 11, respectively, of 
dimension 13.
\end{itemize}
\end{Prop}
\begin{proof}
It is enough to argue fiberwise over $M$, i.e. for a fixed sheaf
$[F]\in M$. Let $y=\sing(F)$ and consider the sets
$B_y=B\times_M\{y\}$ and $B'_y=B'\times_M\{y\}$.
Any points $b=(l,[F],\epsilon)\in B_y$ and $b'=(l,[F],\epsilon')
\in B'_y$ define the exact triples \eqref{triple with cal E} 
and 
\eqref{triple with cal E'} with $\mathcal{E}_b=\ker(\epsilon)$
and $\mathcal{E'}_{b'}=\ker(\epsilon')$, respectively. These 
triples, together with the exact triple $0\to\op3(-2)\to3\cdot
\op3(-1)\xrightarrow{\delta} F\to0$ from Remark 
\ref{R(-1,1,1)}, 
yield commutative diagrams with $\mathcal{G}=\ker(\epsilon\circ
\delta)$ and $\mathcal{G'}=\ker(\epsilon'\circ\delta)$, 
respectively:
\begin{equation}\label{comm diagram1}
\xymatrix{\op3(-2)\ar@{=}[d]\ar@{>->}[r] & \mathcal{G}
\ar@{>->}[d]\ar@{->>}[r] & \mathcal{E}_b\ar@{>->}[d] \\
\op3(-2) \ar@{>->}[r][r] & 3\cdot\op3(-1)\ar@{->>}[r]^-{\delta} 
\ar@{->>}[d]^{\epsilon\circ\delta} & F \ar@{->>}[d]^{\epsilon}\\
&  \calo_l(-1) \ar@{=}[r] & \mathcal{O}_l(-1),  &   }
\xymatrix{\op3(-2)\ar@{=}[d]\ar@{>->}[r] & \mathcal{G'}
\ar@{>->}[d]\ar@{->>}[r] & \mathcal{E'}_{b'}\ar@{>->}[d] \\
\op3(-2) \ar@{>->}[r][r] & 3\cdot\op3(-1)\ar@{->>}[r]^-{\delta} 
\ar@{->>}[d]^{\epsilon'\circ\delta} & 
F\ar@{->>}[d]^{\epsilon'}\\
&  \calo_l\ar@{=}[r] & \mathcal{O}_l.  &   }
\end{equation}
Consider the scheme $\Pi:=\mathbf{P}(3\cdot\op3(-1))\cong\PP
\times \mathbb{P}^2$. To the epimorphism $\epsilon\circ\delta$ 
in the left diagram \eqref{comm diagram1} there corresponds an 
injective morphism $i:\mathbf{P}(\calo_l(-1))\hookrightarrow 
\Pi$ which defines a point $x\in\mathbb{P}^2$ such that 
$\mathrm{im}(i)=l_x:=\{x\}\times l$.
Respectively, to the epimorphism $\epsilon'\circ\delta$ in the 
right diagram \eqref{comm diagram1} there corresponds an 
injective morphism $i':\mathbf{P}(\calo_l)\hookrightarrow \Pi$ 
which defines a point $x'\in(\mathbb{P}_l)_e$, where 
$(\mathbb{P}_l)_e$ is the set of epimorphisms $3\cdot\op3(-1)
\twoheadrightarrow\calo_l\ \mathrm{mod}A'_l$
considered as a dense open subset of the projective  5-space 
$\mathbb{P}(\Hom(3\cdot\op3(-1),\calo_l))$. For this point 
$x'$, we denote $l_{x'}:=\mathrm{im}(i')$.
Besides, to the epimorphism $\delta$ in both diagrams there 
corresponds an injective morphism $i_{\delta}:\mathbf{P}(F)
\hookrightarrow\Pi$. From now on we will identify 
$\mathbf{P}(F)$ with its image under $i_{\delta}$. Now by
\eqref{comm diagram1} the condition $b\in B_y$ and the 
condition $b'\in B'_y$, yield the inclusions
\begin{equation}\label{condn on x}
l_x\subset\mathbf{P}(F),\ \ \ {\mathrm{respectively,}}\ \ \ 
l_{x'}\subset\mathbf{P}(F).
\end{equation} 
Next, by the middle horizontal triple in diagrams \eqref{comm 
diagram1}, $\mathbf{P}(F)$ is a Cartier divisor on $\Pi$ such 
that $\calo_{\Pi}(\mathbf{P}(F))\cong\op3(2)\boxtimes
\calo_{\mathbb{P}^2}\otimes\calo_{\Pi}(1)\cong\op3(1)\boxtimes
\calo_{\mathbb{P}^2}(1))$. Hence
\begin{equation}\label{section s}
\mathbf{P}(F)=(s)_0,\ \ \ \ 0\ne s\in\h^0(\op3(1)\boxtimes
\calo_{\mathbb{P}^2}(1)),
\end{equation}
and the conditons \eqref{condn on x} mean that $s|_{l_x}=0$, 
respectively, $s|_{l_{x'}}=0$. 

Consider the first of these conditions  $s|_{l_x}=0$.
Let $\Pi=\PP\times\mathbb{P}^2\xleftarrow{p}\Gamma\times\mathbb
{P}^2\xrightarrow{q}G\times\mathbb{P}^2$ be the projections. 
Then by construction the sheaf $q_*p^*(\op3(1)\boxtimes
\calo_{\mathbb{P}^2}(1))$ is isomorphic to the sheaf 
$\mathcal{A}=\calo_{\mathbb{P}^2}(1)\boxtimes\mathcal{Q}$, 
where 
$\mathcal{Q}$ is the universal quotient rank 2 bundle on $G$.
In addition, under the natural isomorphism of spaces of 
sections 
$\h^0(\calo_{\op3(1)\boxtimes\mathbb{P}^2}(1))\cong\h^0
(\mathcal{A})$, the section $s$ from \eqref{section s} 
corresponds to the section $\tilde{s}\in\h^0(\mathcal{A})$.
The above condition $s|_{l_x}=0$ then means that the section 
$\tilde{s}$ vanishes at the point $(l,x)\in 
G\times\mathbb{P}^2$.
On the other hand, by the universal property of $\mathbf{P}(F)$ 
(see \cite[Ch. II, Prop. 7.12]{H}) it follows that to give an 
epimorphism $\epsilon:F\twoheadrightarrow\calo_l(-1))\ 
\mathrm{mod}A_l$ is equivalent to give an embedding 
$l_x\hookrightarrow\mathbf{P}(F)$ in \eqref{condn on x}. 
This together with 
the condition $(l,x)\in(\tilde{s})_0$ yields a natural 
isomorphism of schemes 
\begin{equation}\label{By=}
B_y\simeq(\tilde{s})_0.
\end{equation}
Under this isomorphism the fiber of the projection 
$B_y\simeq(\tilde{s})_0\to G,\ (l,x)\mapsto l$ is naturally 
identified with $\mathbb{P}(\Hom(F,\calo_l(-1)))$. By
\eqref{splitingF} this projective space is a point if
$l\not\in Z_y:=\{l\in G\ |\ y\in l\}$, respectively, is
$\mathbb{P}^1$ if $l\in Z_y$. This together with the universal
property of blowing ups \cite[Ch. II, Prop. 7.14]{H} implies
that $B_y$ is isomorphic to the blow-up of $G$ along the smooth
center $Z_y\simeq\mathbb{P}^2$. In particular, $B_y$ is 
irreducible, of dimension 4. Hence $B$ is irreducible of 
dimension 7. 

Now proceed to the second  condition $s|_{l_{x'}}=0$. For this,
consider the scheme $G'=\{(l,x')|\ l\in G,\ x'\in(\mathbb{P}_l)
_e\}$ with the projection $\psi:G'\to G,\ (l,x')\mapsto l$, 
and the graph of incidence $\Gamma'=\{(z,l,x')\in\Pi\times G'|\ 
z\in l_{x'}\}$ with the projections $\Pi\xleftarrow{p'}\Gamma'
\xrightarrow{q'}G'$. One checks that $\calo_{\Pi}(\mathbf{P}(F))
|_{l_x}\cong\calo_{\mathbb{P}^1}(2)$. This implies that, 
applying
the functor $q'_*{p'}^*$ to the section $s$ from \eqref{section 
s} we obtain the section $\tilde{s}'\in\h^0(\psi^*S^2\mathcal{Q}
\otimes\mathcal{D})$ for some invertible $\calo_{G'}$-sheaf 
$\mathcal{D}$ such that the condition $s|_{l_{x'}}=0$ is 
equivalent to the condition  $(l,x')\in(\tilde{s}')_0$. This 
similarly to \eqref{By=} yields $B'_y\simeq(\tilde{s}')_0.$ 

Under this isomorphism the fiber of the projection 
$\psi|_{B'_y}:B'_y\simeq(\tilde{s}')_0\to G,\ (l,x')\mapsto l$ 
is naturally identified with $\mathbb{P}(\Hom(F,\calo_l))$. By
\eqref{splitingF} this projective space is $\mathbb{P}^2$ if
$l\not\in Z_y$, respectively, is $\PP$ if $l\in Z_y$. This 
implies that $\tilde{s}'$ as a section of a rank 3 vector 
bundle 
is regular, and its zero locus $B'_y$ is irreducible. Hence 
$B'$ is irreducible of dimension 9. We thus have proved the 
statement (i) of Lemma. 

The statement (ii) is clear. To prove the statement (iii), it is
also enough to argue fiberwise over $M$. For the above point
$b\in B_y$, we have to prove the irreducibility and generic 
smoothness of the scheme $\mathbf{P}(\mathcal{E}_b)$. This is 
just the statement of Lemma \eqref{two irred} in which we set 
$\mathcal{E}=\mathcal{E}_b$. The irreducibility and generic 
smoothness of $\mathbf{P}(\mathcal{E}'_{b'})$ for $b'\in B'_y$
is completetly similar.
\end{proof}

Let $\rho:\mathbf{P}(\boldsymbol{\mathcal{E}})\to\PP\times B$ 
be 
the structure morphism, and consider the compositions $\theta=
pr_1\circ\rho:\mathbf{P}(\boldsymbol{\mathcal{E}})\to\PP$ and
$\tau=pr_2\circ\rho:\mathbf{P}(\boldsymbol{\mathcal{E}})\to B$. 
Set $\boldsymbol{\mathcal{\tilde{E}}}:=(\mathrm{id}_{\PP}\times
\tau)^*\boldsymbol{\mathcal{E}}=\boldsymbol{\mathcal{E}}
\otimes_{\calo_B}\calo_{\mathbf{P}(\boldsymbol{\mathcal{E}})}$ 
and consider the "diagonal" embedding 
$j:\mathbf{P}(\boldsymbol{\mathcal{E}})\hookrightarrow\PP\times
\mathbf{P}(\boldsymbol{\mathcal{E}}),\ z\mapsto(\theta(z),z)$. 
By construction, $j^*\boldsymbol{\mathcal{\tilde{E}}}=\rho^*
\boldsymbol{\mathcal{E}}$ and we obtain the composition of 
surjections
$\mathbf{e}:\ \boldsymbol{\mathcal{\tilde{E}}}\twoheadrightarrow
j_*j^*\boldsymbol{\mathcal{\tilde{E}}}=j_*\rho^*
\boldsymbol{\mathcal{E}}\twoheadrightarrow j_*
\calo_{\mathbf{P}(\boldsymbol{\mathcal{E}})}(1)$ which yields an
exact $\calo_{\PP\times\mathbf{P}(\boldsymbol{\mathcal{E}})}$-
triple, where $\mathbf{E}:=\ker\mathbf{e}$:
\begin{equation}\label{bold E}
0\to\mathbf{E}\to\boldsymbol{\mathcal{\tilde{E}}}
\xrightarrow{\mathbf{e}}j_*\calo_{\mathbf{P}(\boldsymbol
{\mathcal{E}})}(1)\to0.
\end{equation}
In a similar way we define the morphisms $\rho':\mathbf{P}
(\boldsymbol{\mathcal{E'}})\to\PP\times B'$, $\theta'=
pr_1\circ\rho':\mathbf{P}(\boldsymbol{\mathcal{E}'})\to\PP$,
$j':\mathbf{P}(\boldsymbol{\mathcal{E'}})\hookrightarrow\PP
\times
\mathbf{P}(\boldsymbol{\mathcal{E'}}),\ z\mapsto(\theta'(z),z)$,
the sheaf $\boldsymbol{\mathcal{\tilde{E}'}}:=\boldsymbol
{\mathcal{E}'}\otimes_{\calo_{B'}}\calo_{\mathbf{P}(\boldsymbol
{\mathcal{E}'})}$, and the surjection $\mathbf{e'}:\ \boldsymbol
{\mathcal{\tilde{E}'}}\twoheadrightarrow j'_*\calo_{\mathbf{P}
(\boldsymbol{\mathcal{E}'})}(1)$ which yields an
exact $\calo_{\PP\times\mathbf{P}(\boldsymbol{\mathcal{E}'})}$-
triple, where $\mathbf{E'}:=\ker\mathbf{e'}$:
\begin{equation}\label{bold E'}
0\to\mathbf{E'}\to\boldsymbol{\mathcal{\tilde{E}'}}\xrightarrow
{\mathbf{e'}}j'_*\calo_{\mathbf{P}(\boldsymbol{\mathcal{E}'})}
(1)\to0.
\end{equation}

Below we will also consider extensions of $\op3$-sheaves of the 
form
\begin{equation}\label{extn Q}
0\to\calo_q\to Q\xrightarrow{\gamma}i_*\calo_l(-1)\to0,\ \ \ \ 
\ \end{equation}
\begin{equation}\label{2extn Q}
0\to\calo_q\to Q\xrightarrow{\gamma}i_*\calo_l\to0,\ \ \ \ \ 
\end{equation}
where $(q,l)\in\PP\times G$ and $i:l\hookrightarrow\PP$ is the 
embedding. Below we also set $A_Q:=\mathrm{Aut}(Q)$ for $Q$ in 
\eqref{extn Q} and \eqref{2extn Q}.
\begin{Prop}\label{P(E)=X} The following are true.
\begin{itemize}
\item[(i)] There are isomorphisms of schemes $\Phi:\mathbf{P}
(\boldsymbol{\mathcal{E}})\xrightarrow{\simeq}X$ and $\Phi':
\mathbf{P}(\boldsymbol{\mathcal{E}'})\xrightarrow{\simeq}X'$, 
where
\begin{equation}\label{X}
X=\{([F],Q,\delta\ \mathrm{mod}A_Q)|\ [F]\in M,\ Q\ 
\mathrm{fits\ in}\ \eqref{extn Q},\ \delta:F\to Q\ \mathrm{is\ 
surjective}\}.
\end{equation}
\begin{equation}\label{'}
X'=\{([F],Q,\delta\ \mathrm{mod}A_Q)|\ [F]\in M,\ Q\ 
\mathrm{fits\ in}\ \eqref{2extn Q},\ \delta:F\to Q\ \mathrm{is\ 
surjective}\}.
\end{equation}
\item[(ii)] There are inclusions of dense open subschemes
$$ \mathcal{X}(-1,1,1,-1,1)\hookrightarrow\mathbf{P}(\boldsymbol
{\mathcal{E}}) ~~ {\it and} ~~ \mathcal{X}(-1,1,1,,1)\hookrightarrow
\mathbf{P}(\boldsymbol{\mathcal{E'}}). $$
The modular morphisms
$$ f:\mathbf{P}(\boldsymbol{\mathcal{E}})\to
\calm(-1,2,2),\ z\mapsto[\mathbf{E}|_{\PP\times\{z\}}]  ~~ {\it and} ~~
f':\mathbf{P}(\boldsymbol{\mathcal{E'}})\to
\calm(-1,2,0),\ z\mapsto[\mathbf{E'}|_{\PP\times\{z\}}] $$
are injective, and the closures of their images are $\mathrm{X}
(-1,1,1,-1,1)$ and $\mathrm{X}(-1,1,1,0,1)$,
respectively.
\end{itemize}
\end{Prop} 
\begin{proof}
(i) It is enough to consider $\mathbf{P}(\boldsymbol{\mathcal
{E}})$, since the argument with $\mathbf{P}(\boldsymbol
{\mathcal{E'}})$ is similar.
For any point $z\in\mathbf{P}(\boldsymbol{\mathcal{E}})$ let
$(q,b)=\rho(z)$. By definition the triple \eqref{bold E}
resricted onto $\PP\times\{z\}$ is the triple 
\begin{equation}
0\to E_z\to\mathcal{E}_b\xrightarrow{e_z}\calo_q\to0.
\end{equation}
On the other hand, by Proposition \ref{descriptn of B,B'}.(ii), 
$b=(l,[F],\epsilon\ \mathrm{mod}A_l)$ and
$\mathcal{E}_b$ fits in the triple \eqref{triple with cal E} in
which $F=\mathcal{E}_b^{\vee\vee}$, $\iota:\mathcal{E}_b\to
\mathcal{E}_b^{\vee\vee}$ is the canonical morphism and 
$\epsilon:\mathcal{E}_b^{\vee\vee}\twoheadrightarrow\calo_l(-1)$
is the quotient morphism. Since $\mathcal{E}_b^{\vee\vee}=
E_z^{\vee\vee}$, the composition $\tau:E_z\to\mathcal{E}_b
\xrightarrow{\iota}E_z^{\vee\vee}$ is the canonical morphism
of the sheaf $E_z$ into its reflexive hull, and $Q:=\coker(i)$ 
fits in the triple \eqref{extn Q}. We thus have an exact triple
\begin{equation}
0\to E_z\xrightarrow{\tau}E_z^{\vee\vee}\xrightarrow{\delta}Q
\to0,
\end{equation} 
where $\delta$ is the quotient morphism. This defines a morphism
$$ \Phi:\mathbf{P}(\boldsymbol{\mathcal{E}})\xrightarrow{\simeq}X,
\ z\mapsto([E_z^{\vee\vee}],Q=E_z^{\vee\vee}/E_z,\delta\ 
\mathrm{mod}A_Q).$$ 
To construct the inverse morphism 
$\Phi^{-1}$, take a point $x=([F],Q,\delta\ \mathrm{mod}A_Q)$ and set
$\mathcal{E}=\ker(\gamma\circ\delta:F\twoheadrightarrow
\calo_l(-1))$, where $\gamma$ in \eqref{extn Q} is the morphism 
of factorization of $Q$ by its maximal artinian subsheaf 
$\calo_q$. 
We thus obtain the induced epimorphism $e:\mathcal{E}
\twoheadrightarrow\ker(\gamma)=\calo_q$, hence a
point $[e]=e\ \mathrm{mod}A_q\in\mathbf{P}(\boldsymbol
{\mathcal{E}})$. This yields the desired morphism $\Phi^{-1}:
\ X\xrightarrow{\simeq}\mathbf{P}(\boldsymbol{\mathcal{E}}),
\ x\mapsto[e]$.\\
(ii) The injectivity of the modular morphism $f$ is clear from 
the above. In addition, under the description \eqref{X}, the 
scheme $\mathcal{X}(-1,1,1,-1,1)$ is the set of those points
$z=([F],Q,\delta\ \mathrm{mod}A_Q)\in\mathbf{P}(\boldsymbol
{\mathcal{E}})$, with $(q,l)=\rho(z)$, for which $q\not\in l$,
$q\ne\sing(F)\not\in l$. This is clearly a nonempty 
open subset of the scheme 
$\mathbf{P}(\boldsymbol{\mathcal{E}})$ 
which is dense since $\mathbf{P}(\boldsymbol{\mathcal{E}})$ 
is irreducible by Proposition \ref{descriptn of B,B'}.(ii). 
\end{proof}

Consider the sheaf $\mathbf{E}$ defined in \eqref{bold E}. Let
$\mathbf{r}:\mathbf{P}(\mathbf{E})\to\PP\times\mathbf{P}
(\boldsymbol{\mathcal{E}})$ be the structure morphism, and 
consider the composition $\mathbf{t}=pr_1\circ\mathbf{r}:
\mathbf{P}(\mathbf{E})\to\PP$ and the "diagonal" embedding 
$\mathbf{j}:\mathbf{P}(\mathbf{E})\hookrightarrow\PP\times
\mathbf{P}(\mathbf{E}),\ w\mapsto(\mathbf{t}(w),w)$. 
Set $\mathbf{P}(\mathbf{\tilde{E}}):=\mathbf{E}
\otimes_{\calo_{\mathbf{P}(\boldsymbol{\mathcal{E}})}}
\calo_{\mathbf{P}(\mathbf{E})}$. By construction, $\mathbf{j}^*
\mathbf{\tilde{E}}=\mathbf{r}^*\mathbf{E}$ and we obtain the 
composition of surjections
$\mathbf{\tilde{e}}:\ \mathbf{\tilde{E}}\twoheadrightarrow
\mathbf{j}_*\mathbf{j}^*\mathbf{\tilde{E}}=\mathbf{j}_*
\mathbf{r}^*\mathbf{E}\twoheadrightarrow \mathbf{j}_*
\calo_{\mathbf{P}(\mathbf{E})}(1)$ which yields an exact
$\calo_{\PP\times\mathbf{P}(\mathbf{E})}$-triple, where 
$\mathbf{\hat{E}}:=\ker\mathbf{\tilde{e}}$:
\begin{equation}\label{hat bold E}
0\to\mathbf{\hat{E}}\to\mathbf{\tilde{E}}\xrightarrow{\mathbf
{\tilde{e}}}\mathbf{j}_*\calo_{\mathbf{P}(\mathbf{E})}(1)\to0.
\end{equation}
We will also consider the exact triples of the form
\begin{equation}\label{new extn Q}
0\to Z\to Q\to i_*\calo_l(-1)\to0,\ \ \ \ \ 
\dim Z=0,\ \ \ \mathrm{length}(Z)=2,
\end{equation}
where $i:l\hookrightarrow\PP$ is the embedding of a line $l\in 
G$.
\begin{Prop}\label{X(-1,1,1,-1,2)}
The sheaf $\mathbf{\hat{E}}$ defined in \eqref{hat bold E} 
determines the modular morphism
$$ \hat{\Phi}:\mathbf{P}
(\mathbf{E})\to\mathcal{M}(-1,2,0),\ w\mapsto\left[\mathbf
{\hat{E}}|_{\PP\times\{w\}}\right] , $$
and the closure $\overline
{\hat{\Phi}(\mathbf{P}(\mathbf{E}))}$ of its image in 
$\mathcal{M}(-1,2,0)$ coincides with the scheme $\mathrm{X}
(-1,1,1,-1,2)$. In particular, $\mathrm{X}(-1,1,1,-1,2)$ 
contains all the points $[E]$ such that $Q=E^{\vee\vee}/E$ 
fits in the triple of the form \eqref{new extn Q}.
\end{Prop}
\begin{proof}
First note that, by the definition of the sheaf $\mathbf{E}$, 
the scheme $\mathbf{P}(\mathbf{E})$ is fibered over the scheme
$\mathbf{P}(\boldsymbol{\mathcal{E}})$ with fiberes of the form
$\mathbf{P}(E)$ described in Lemma \ref{two irred}. Hence by 
that Lemma, these fibers are irreducible of dimension 4. 
Besides, the 
scheme $\mathbf{P}(\boldsymbol{\mathcal{E}})$ is also 
irreducible of dimension 11 by Proposition \ref{descriptn of 
B,B'}.(iii). Hence the scheme $\mathbf{P}(\mathbf{E})$ is 
irreducible of dimension 15. The fact that it contains the 
scheme $\mathcal{X}(-1,1,1,-1,2)$ as a dense open subset is
proved in the same way as the statement of item (ii) of Proposition 
\ref{P(E)=X}, based on the universal property of 
$\mathbf{P}(\mathbf{E})$ from \cite[Ch. II, Prop. 7.12]{H}.
\end{proof}

\begin{Prop}\label{X in T} The following claims hold.
\begin{itemize}
\item[(i)] The scheme $\mathrm{X}(-1,1,1,-1,1)$ is contained in 
$\mathrm{T}(-1,2,4,1)$.
\item[(ii)] The scheme $\mathrm{X}(-1,1,1,-1,2)$ is contained 
in $\mathrm{T}(-1,2,4,2)$.
\item[(iii)] The scheme $\mathrm{X}(-1,1,1,0,1)$ is contained 
in $\mathrm{T}(-1,2,2,1)$.
\end{itemize}
\end{Prop}
\begin{proof}
(i) We will construct a flat family $\mathbb{E}=\{E_t\}_{t\in
\mathbb{P}^1}$ of sheaves from $\mathrm{T}(-1,2,4,1)$ such 
that, 
for a certain point $t_0\in\mathbb{P}^1$, $E_{t_0}$ is a smooth 
point of $\calm(-1,2,2)$ lying in $\mathrm{X}(-1,1,1,-1,1)\cap 
\mathrm{T}(-1,2,4,1)$. From this the statement (i) will follow. 
Fix a plane $\mathbb{P}^2$ in $\PP$ and choose a pencil of 
conics in $\mathbb{P}^2$ considered as a divisor $D$ in 
$\mathbb{P}^2\times\mathbb{P}^1$, with the projection 
$D\xrightarrow{p}\mathbb{P}^1$, and, for $t\in\mathbb{P}^1$, 
denote $C_t=p^{-1}(t)$. We choose the pencil $D$ in such a way 
that, for two distinct marked points $t_0,t_1\in\mathbb{P}^1$, 
$C_{t_0}=l_1\cup l_2$ is a union of two distinct lines and
$C_{t_1}$ is a smooth conic intersecting $C_{t_0}$ at 4 distinct
points. Fix a point $q\in\PP\smallsetminus\mathbb{P}^2$, and 
on $\Sigma=\PP\times\mathbb{P}^1$ consider the line $L=\{q\}
\times\mathbb{P}^1$ and the extension of sheaves, where
$\mathcal{A}=\op3(-1)\boxtimes\calo_{\mathbb{P}^1}(-1)$:
     \begin{equation}\label{bbE}
0\to\mathcal{A}\otimes\cali_{L,\Sigma}\to\mathbb{E}\to
\cali_{D,\Sigma}\to0.
\end{equation} 
The extension group corresponding to \eqref{bbE} is
$V:=\ext^1(\cali_{D,\Sigma},\mathcal{A}\otimes\cali_{L,\Sigma})
\simeq\ext^1(\cali_{D,\Sigma},\mathcal{A})\simeq\\
\h^0(\lext^1(\cali_{D,\Sigma},\mathcal{A}))\simeq
\h^0(\lext^2(\calo_{D},\mathcal{A}))\simeq\h^0(\op3(2)\boxtimes
\calo_{\mathbb{P}^1}|_D)\simeq\h^0(\calo_{\mathbb{P}^2}(2))$. 
Thus the element of $V$ defining the extension \eqref{bbE} is 
understood as a section $0\ne s\in\h^0(\calo_{\mathbb{P}^2}
(2))$. Now pick the section $s$ such that it vanishes on the 
line $l_1$ and  doesn't vanish on the line $l_2$; hence it also 
doesn't vanish on the conic $C_{t_1}$. Then by the Serre 
construction we obtain the following properties of sheaves 
$E_{t}=\mathbb{E}|_{\PP\times\{t\}}$.\\ 
(i.a) Under the generic choice of the conic $C_{t_1}$, for 
generic 
$t\in\mathbb{P}^1$, the sheaf $[E_{t}]$ is a generic sheaf from
$\mathrm{T}(-1,2,4,1)$. In other words, $\mathbb{P}^1\subset 
\mathrm{T}(-1,2,4,1)$.
In particular, $[E_{t_0}]\in \mathrm{T}(-1,2,4,1)$.\\
(i.b) The sheaf $E_{t_0}$ fits in the exact triple
$0\to E_{t_0}\xrightarrow{\mathrm{can}}E_{t_0}^{\vee\vee}\to
\calo_l(-1)\oplus\calo_q\to0$, where $[E_{t_0}^{\vee\vee}]\in 
M$ and, by the construction of $E_{t_0}$, $q\not\in
\sing(E_{t_0}^{\vee\vee})\cup l$. This means that 
$[E_{t_0}]\in\mathcal{X}(-1,1,1,-1,1)$, and
Theorem \ref{NewComponentsmixed}.(iii) implies that 
$\dim\ext^1(E_{t_0},E_{t_0})=15$. Hence, since
$\dim \mathrm{T}(-1,2,4,1)=15$ (see Theorem \ref{0dcomp}), it 
follows that $E_{t_0}$ is a smooth point of $\mathrm{T}
(-1,2,4,1)$ and of $\calm(-1,2,2)$ as well. \\
(ii) This is completely similar to the statement (i) above.
The only difference is that, instead of fixing a point 
$q\in\PP\smallsetminus\mathbb{P}^2$, we fix two distinct points 
$q_1,q_2\in\PP\smallsetminus\mathbb{P}^2$, and on $\Sigma=\PP
\times\mathbb{P}^1$ consider the two corresponding lines 
$L_i=\{q_i\}\times\mathbb{P}^1$ and the extension of sheaves 
similar to \eqref{bbE}: $0\to\mathcal{A}\otimes\cali_{L_1\sqcup 
L_2,\Sigma}\to\mathbb{E}\to\cali_{D,\Sigma}\to0$. Respectively, 
for $V$ we take the group $\ext^1(\cali_{D,\Sigma},\mathcal{A}
\otimes\cali_{L_1\sqcup L_2,\Sigma})\simeq\h^0
(\calo_{\mathbb{P}^2}(2))$. The rest of the argument is 
literally the same as in (i).

(iii) Similar to the above we construct a flat family $\mathbb{E}
=\{E_t\}_{t\in\mathbb{A}^1}$ of sheaves from $\mathrm{T}(-1,2,2,1)$ 
such that, for a 
certain point $t_0\in\mathbb{A}^1$, $E_{t_0}$ is a smooth point 
of $\calm(-1,2,0)$ lying in $\mathrm{X}(-1,1,1,0,1)\cap 
\mathrm{T}(-1,2,2,1)$. 
From this the statement (ii) will follow. We will use the 
description of sheaves from $\calm(-1,2,2)$ given in 
\cite[Lemma 2.4]{Chang}. Thus, instead of the above family of 
conics $D=\{C_t
\}_{t\in\mathbb{P}^1}$ we take for $C_t,\ t\in\mathbb{A}^1$, a 
fixed union $Y=l_1\sqcup l_2$ of two disjoint lines in $\PP$, 
fix a point $q\in\PP\smallsetminus Y$, set $D=Y\times\mathbb
{A}^1$, $L=\{q\}\times\mathbb{A}^1$. For these data consider 
the extension \eqref{bbE}, where we set 
$\mathcal{A}=\op3(-1)\boxtimes\calo_{\mathbb{A}^1}$, and then d 
the extension group 
$V=\ext^1(\cali_{D,\Sigma},\mathcal{A}\otimes
\cali_{L,\Sigma})$ as above. One easily see that $V\cong 
\h^0(\calo_{l_1})\oplus\h^0(\calo_{l_2})$. For $i=1,2$ pick a 
nonzero vector $v_i\in\h^0(\calo_{l_i})$ and identify the base
$\mathbb{A}^1$ of the family $\{C_t\},\ t\in\mathbb{A}^1$, with
the subset $\{(v_1,tv_2)|t\in\mathbf{k}\}$. By the Serre 
construction we obtain the following properties of sheaves 
$E_{t}=\mathbb{E}|_{\PP\times\{t\}}$.\\ 
(a) For $t\in\mathbb{A}^1\smallsetminus\{0\}$, the sheaf 
$[E_{t}]$ by definition belongs to $\mathcal{T}(-1,2,2,1)$. It
it follows that $\mathbb{A}^1\subset\mathrm{T}(-1,2,2,1)$.
In particular, $[E_0]\in\mathrm{T}(-1,2,2,1)$.\\
(b) The sheaf $E_0$ fits in the exact triple
$0\to E_0\xrightarrow{\mathrm{can}}E_0^{\vee\vee}\to
\calo_{_2}\oplus\calo_q\to0$, where $[E_0^{\vee\vee}]\in M$ 
and, by the construction of $E_0$, $q\not\in\sing
(E_0^{\vee\vee})\cup l$. This means that $[E_0]\in\mathcal{X}
(-1,1,1,0,1)$, and Theorem \ref{NewComponentsmixed}.(iii) 
implies that $\dim\ext^1(E_0,E_0)=15$. Hence, since $\dim\mathrm
{T}(-1,2,2,1)=15$ (see Theorem \ref{0dcomp}), it follows that 
$E_0$ is a smooth point of $\mathrm{T}(-1,2,2,1)$ and of 
$\calm(-1,2,0)$ as well. \end{proof}

\section{Irreducible components of $\calm(-1,2,2)$}
\label{irreducible of M(-1,2,2)}

Now, we are in position to prove the next main result of 
this paper.

\begin{Teo}\label{M(-1,2,2)}
The moduli space $\calm(-1,2,2)$ of rank 2 stable sheaves on 
$\PP$ with Chern classes $c_1 = -1, c_2 = 2, c_3 = 2$, has 
exactly $2$ irreducible rational components, namely:
\begin{itemize}
\item[(i)] the closure $\overline{\calr(-1,2,2)}$ of the family of 
reflexive sheaves $\calr(-1,2,2)$, of dimension $11$;
\item[(ii)] the irreducible component $\mathrm{T}(-1,2,4,1)$ given  by 
Theorem \ref{0dcomp}, of dimension $15$, whose generic element 
is a torsion free sheaf $E$ such that $E^{\vee \vee} \in 
\calr(-1,2,4)$ and $Q_E$ is a sheaf of length 1;
\item[(iii)] in addition, $\calm(-1,2,2)\smallsetminus\calr(-1,2,2)=
\mathrm{T}(-1,2,4,1)$.
\end{itemize}
\end{Teo}
\begin{proof} By \cite[Thm 2.5]{Chang}, $\calr(-1,2,2)$ is 
irreducible, nonsingular of dimension $11$, and its closure 
$\overline{\calr(-1,2,2)}$ in $\calm(-1,2,2)$ is an 
irreducible component of $\calm(-1,2,2)$ of dimension $11$. 
Consider $E \in \calm(-1,2,2) \setminus \calr(-1,2,2)$. By 
Proposition \ref{ChernClasses}, either $\dim Q_E=1$ and 
$1\leq\mult Q_E\leq 2$, or $\dim Q_E=0$. Consider all the 
possibilities for $\dim Q_E$ and $\mult Q_E$.\\
i) If $\dim Q_E=1$ and $\mult Q_E = 2$, then $c_2(E^{\vee 
\vee}) = 0$, and, as in the case i) of the proof of Theorem
\ref{M(-1,2,4)}, we are led to a contradiction.\\
ii) If $\dim Q_E=1$ and $\mult Q_E = 1$, then $c_2(E^{\vee 
\vee})=1$ and $c_3(E^{\vee\vee})=1$ and $Q_E$ is supported 
on a line. Then $Q_E$ fit in an exact sequence of the form 
(\ref{extensionQ_E}) where $Z_E$ is the maximal artinian 
subsheaf of $Q_E$. Then the Euler characteristic of $Q_E(t)$ 
is given by formula (\ref{chi(Q_E(t))}) which together with 
(\ref{fundamentalcomponents}) yields:
$-1=\chi(E)=\chi(E^{\vee\vee})-\chi(Q_E)=-1-r-s$. Hence 
$-r-s=0$ and, since we have an epimorphism 
$\delta:E^{\vee\vee}
\to Q_E$, from equation (\ref{splitingF}) it 
follows that $r\geq -1$. This implies that the possible 
values for $r$ and $s$ are $r=s=0$ or $r=-1$, $s=1$.\\
Case ii.1) Assume that $r=s=0$. In this case, 
$Q_E \simeq i_{*}\mathcal{O}_{l}$, for some line $l$, where 
$i:l\into\PP$ is the embedding. If $l\cap\sing 
E^{\vee\vee}=\emptyset$, then $E$ is a sheaf in 
$\mathcal{X}(-1,1,1,0,0)$, i. e. a generic sheaf in 
$\mathrm{X}(-1,1,1,0,0)$, where 
$\dim\mathrm{X}(-1,1,1,0,0)=9$ 
by \eqref{dimfamilymixed}. However, this dimension is too 
small for $\mathrm{X}(-1,1,1,0,0)$ to be an irreducible 
component of $\calm(-1,2,2)$. Next, if $l\cap\sing
~E^{\vee\vee}\neq\emptyset$, then $E \in Y(0)$ and  by 
Lemma \ref{nonemptyfamily} also $Y(0)$ does not fill an 
irreducible component of $\calm(-1,2,2)$. \\
Case ii.2) Assume that $r=-1$ and $s=1$. In this case, $Q_E$ 
fits into the exact triple \eqref{extn Q} for some pair
$(q,l)\in\PP\times G$, where $i:l \into\PP$ is the embedding 
and $Q=Q_E$. Since $[E^{\vee\vee}]\in M=\calm(-1,1,1)$, 
Proposition \ref{P(E)=X}.(i),(ii) yields that 
$[\delta:E^{\vee\vee}\twoheadrightarrow Q]=\delta\ \mathrm
{mod}A_Q$ is the point in
$\mathbf{P}(\boldsymbol{\mathcal{E}})$, i. e., $[E]\in
\mathrm{X}(-1,1,1,-1,1)$. This together with Proposition \ref{X 
in T} implies that $[E]\in\mathrm{T}(-1,2,4,1)$. \\
iii) If $\dim Q_E=0$, then $s=\mathrm{length}(Q_E)>0$. By 
Proposition \ref{ChernClasses}, $c_2(E)=c_2(E^{\vee\vee})=2$ 
and $c_3(E^{\vee\vee})=c_3(E)+2s=2+2s\ge4$. On the other hand,
$c_3(E^{\vee\vee})\le c_2(E^{\vee\vee})^2=4$ by \eqref{c3 le} 
since $E^{\vee\vee}$ is stable. Hence $s = 1$, 
$c_3(E^{\vee\vee})=4$, i. e. $Q_E\simeq\mathcal{O}_p$ for 
some point $p\in\PP$. Then by Theorem \ref{T(-1,2,2i,1)}.(ii) 
$[E]$ belongs to the irreducible component $\mathrm{T}
(-1,2,4,1)$ of $\calm(-1,2,2)$.

In conclusion, we have proved that $\calm(-1,2,2) = 
\overline{\calr(-1,2,2)}\cup\mathrm{T}(-1,2,4,1)$.
Finally, remark that the rationality of $\overline{\calr(-1,2,2)}$ is known from \cite{Chang2}, and the rationality of $\mathrm{T}(-1,2,4,1)$ follows from Main Theorem \ref{main2}.
\end{proof}

\section{Irreducible components of $\calm(-1,2,0)$} 
\label{irreducible of M(-1,2,0)}

We are now ready to describe all the irreducible components of 
$\calm(-1,2,0)$.
\begin{Teo}\label{M(-1,2,0)} 
The moduli space $\calm(-1,2,0)$ of rank 2 stable sheaves on 
$\PP$ with Chern classes $c_1 = -1, c_2 = 2, c_3 = 0$, has 
exactly 4 irreducible rational components, namely:
\begin{itemize}
\item[(i)] The closure of the family of stable rank $2$ locally free 
sheaves $\calb(-1,2)$, of dimension $11$;
\item[(ii)] The irreducible component $\mathrm{X}(-1,1,1,1,0)$ of 
dimension $11$, described by Theorem \ref{NewComponentsmixed}, 
whose generic element is a torsion free sheaf $E$ such that 
$E^{\vee \vee} \in \calr(-1,1,1)$ and $Q_E=i_*\calo_l(1)$ for 
some line $i:l\hookrightarrow\PP$.
\item[(iii)] The irreducible component $\mathrm{T}(-1,2,2,1)$ of 
dimension $15$ described  in Theorem \ref{0dcomp}, whose 
generic sheaf is a torsion free sheaf $E$ such that $E^{\vee 
\vee}\in\calr(-1,2,2)$ and $Q_E$ is a sheaf of length 1.
\item[(iv)] The irreducible component $\mathrm{T}(-1,2,4,2)$ of 
dimension $19$ described by the Theorem \ref{0dcomp}, whose 
generic sheaf is a torsion free sheaf $E$ such that $E^{\vee 
\vee}\in\calr(-1,2,4)$ and $Q_E$ is a sheaf of length 2, 
supported at two distinct points.
\end{itemize}
\end{Teo}
\begin{proof}
By \cite[Proposition 4.1]{H1978}, $\overline{\calb(-1,2)}$ is 
an irreducible component of $\calm(-1,2,0)$ of dimension $11$. 
Consider $[E]\in \calm(-1,2,0) \setminus \calb(-1,2)$; again,  
By Proposition \ref{ChernClasses}, either $\dim Q_E=1$ and 
$\leq\mult Q_E \leq 2$, or $\dim Q_E=0$. We will study the 
possibilities for $\dim Q_E$ and $\mult Q_E$. \\
i) If $\dim Q_E=1$, $\mult Q_E = 2$, then 
$c_2(E^{\vee\vee})=0$, 
and by \eqref{c3 le} $c_3(E^{\vee\vee})=0$. Thus $E^{\vee 
\vee}$ 
is a stable rank 2 vector bundle with $c_1(E^{\vee \vee})=-1,\ 
c_2(E^{\vee \vee})=0$, contrary to \cite[Cor. 3.5]{H1978}.\\
ii) If $\dim Q_E=\mult Q_E = 1$, then $c_2(E^{\vee \vee})=1$. 
Hence $Q_E$ is supported on a line $i:l\into \PP$ and, 
possibly, isolated points and fits in the exact sequence 
\ref{extensionQ_E}, 
where $\dim Z_E\le0,\ \mathrm{length}Z_E=s$. Then the second 
formula \eqref{chi(Q_E(t))} and Proposition 
\ref{ChernClasses}.b) 
yield $c_3(E^{\vee \vee})=2(r+s)-1$. This
together with \eqref{c3 le} implies that $0\le c_3(E^{\vee\vee})
=2(r+s)-1\le1$. Hence, $c_3(E^{\vee\vee})=r+s=1$, i. e. 
$[E^{\vee\vee}]\in\calm(-1,1,1)$. Since we have an epimorphism 
$E^{\vee \vee} \overset{\delta}{\twoheadrightarrow}Q_E$, it 
follows from \eqref{splitingF} that $r\ge-1$. This together 
with the inequality $s\ge0$ implies that the  possible values 
for $r$ 
and $s$ are: $r=1$ and $s=0$, or $r=0$ and $s=1$, or $r=-1$ and 
$s=2$. Consider these three cases.\\
Case ii.1): $r = 1$ and $s = 0$. In this case, $Q_E \simeq 
i_{*}\mathcal{O}_{l}(1)$. If $l\cap\sing(E^{\vee\vee}) = 
\emptyset$, then by definition $[E]\in\mathcal{X}(-1,1,1,1,0)$, 
that is $E$ is a generic sheaf in $\mathrm{X}(-1,1,1,1,0)$. 
Here by Theorem \ref{NewComponentsmixed}.(i) 
$\mathrm{X}(-1,1,1,1,0)$ 
is the irreducible component of dimension $11$ in 
$\calm(-1,2,0)$. If $l \cap \sing(E) \neq  \emptyset$, 
then  by the Lemma \ref{nonemptyfamily} the family of all such 
$E$ cannot constitute an irreducible component of $\calm
(-1,2,0)$.\\  
Case ii.2): $r=0$ and $s=1$. In this case, the triple
\eqref{extensionQ_E} is:
$0\to\mathcal{O}_p\to Q_E\xrightarrow{\mathrm{can}}
i_{*}\mathcal{O}_{l}\to0$, where $p$ is some point in $\PP$.
This together with the Snake Lemma implies that the sheaf 
$\mathcal{E}$ defined as the kernel of the composition 
$E^{\vee\vee}\overset{\mathrm{can}}{\twoheadrightarrow}Q_E
\overset{\delta}{\twoheadrightarrow}i_{*}\mathcal{O}_{l}$ fits 
in the exact triple $0\to E\to\mathcal{E}\to\calo_p\to0$.
This triple and the stability of $E$ implies the stability of
$\mathcal{E}$, hence $[\mathcal{E}]\in\calm(-1,2,2)$. Since
$\dim\sing(E)=1$, we have by definition that $[E]\in
X(-1,1,1,0,1)$. From Theorem \ref{X in T}.(iii) it follows now 
that $[E]\in\mathrm{T}(-1,2,2,1)$.\\
Case ii.3): $r=-1$ and $s=2$. In this case, $Q_E$ 
fits into an exact sequence of the form:
$0 \to Z_E \to Q_E \to i_{*}\mathcal{O}_{l}(-1) \to 0$,
where $Z_E$ has length $2$ and where  $i:l\into\PP$ is the 
embedding of some line $l$. Therefore, by Proposition 
\ref{X(-1,1,1,-1,2)}, $[E]\in\mathrm{X}(-1,1,1,-1,2)$. Then 
from Proposition \ref{X in T}(ii) it follows that $[E]\in
\mathrm{T}(-1,2,4,2)$.

iii) If $\dim Q_E = 0$, let $s = \mathrm{length}(Q_E)$, since 
we are assuming that $E$ is properly torsion free, it follows 
that $s>0$. By Proposition \ref{ChernClasses}, $c_2(E) = 
c_2(E^{\vee \vee})$ and and  $c_3(E^{\vee \vee}) = 
c_3(E)+2s$. Therefore, 
either $s=1$, $c_3(E^{\vee \vee})=2$, or $s=2$, and 
$c_3(E^{\vee \vee})=4$. Consider both these cases.\\
Case 1: $s= 1$, then $c_3(E^{\vee \vee})=2$. In this case $Q_E 
\simeq\mathcal{O}_p$ for some point $p\in\PP$, so that  
$[E]\in\mathrm{T}(-1,2,2,1)$ by Theorem Theorem \ref{T(-1,2,2i,1)}.(i).\\ 
Case 2: $s = 2$, then $Q_E$ has length $2$, and $[E]\in 
\mathrm{T}(-1,2,4,2)$ by Theorem \ref{T(-1,2,4,2)}.

In conclusion, we have proved that $\calm(-1,2,0)= 
\overline{\calb(-1,2)}\cup\mathrm{T}(-1,2,2,1)\cup 
\mathrm{T}(-1,2,4,2)\cup\mathrm{X}(-1,1,1,1,0)$.
Remark also that the rationality of of $\overline{\calb(-1,2)}$ 
is proved in \cite{Hart2}, the rationalty of 
$\mathrm{T}(-1,2,2,1)$ and $\mathrm{T}(-1,2,4,2)$ follows from
Main Theorem \ref{main2}, and the rationality of 
$\mathrm{X}(-1,1,1,1,0)$ also follows from Main Theorem 
\ref{main2} with small additional argument due to the 
elementary transformations of sheaves along the line $l$.
\end{proof}

\begin{Remark}
Meseguer, Sols and Str\o mme proved in \cite{MSS} that 
$\calm(-1,2,0)$ contains, besides $\calb(-1,2)$, at least two 
families of non locally free sheaves  containing sheaves that 
are not limits of locally free sheaves. Zavodchikov then proved 
in \cite{Z} that these families of sheaves form irreducible 
components of dimension 15 and 19; they coincide with the 
components we denoted by $\mathrm{T}(-1,2,2,1)$ and 
$\mathrm{T}(-1,2,4,2)$, respectively. Later, Zavodchikov proved 
in \cite{Za} that $\calm(-1,2,0)$ consists of exactly 4 
irreducible components; this article, however, is only 
available in russian.

We emphasize that our proof of Theorem \ref{M(-1,2,0)} is 
completely independent from the results in \cite{MSS,Z,Za}, 
treating $\calm(-1,2,c_3)$ in a uniform manner for all 3 
possible values of $c_3$. Additionally, it also provides 
further information on the generic element of each component. 
\end{Remark}

\section{Connectedness of the spaces $\calm(-1,2,c_3)$}\label
{Connectedness of M(2)}

Since the space $\calm(-1,2,4)$ is irreducible, it is 
obviously connected. In this section we will prove that the 
spaces $\calm(-1,2,2)$ and $\calm(-1,2,0)$ are also connected.
\begin{Teo}\label{connected1}
The moduli space $\calm(-1,2,2)$ is connected. 
\end{Teo}
\begin{proof}
First, remark that one easily constructs a flat family of 
curves $\mathcal{Z}$ in $\PP$ with base $\mathbb{A}^1$, i. e.,
a subscheme $\mathcal{Z}$ of $\PP\times\mathbb{A}^1$ with
the projection $\pi:\ \mathcal{Z}\to\PP\times\mathbb{A}^1 
\xrightarrow{pr_2}\mathbb{A}^1$ satisfying the properties:\\  
a) for $t\in\mathbb{A}^1\smallsetminus\{0\}$, the fiber $Z_t 
:=\pi^{-1}(t)$ is a disjoint union $l_{1t}\sqcup l_{2t}$ of two lines;\\
b) the zero fiber $Z_0:=\pi^{-1}(0)$ as a set is the union of 
two distinct lines $l_{10}$ and $l_{20}$ meeting at a point, 
say, $p$ such that $(Z_0)_{red}=l_{10}\cup l_{20}$ and $Z_0$ as a 
scheme has $p$ as an embedded point; more precisely, there is 
an exact triple:
\begin{equation}\label{pointp} 
0\to\calo_p\to\calo_{Z_0}\to\calo_{(Z_0)_{red}} \to 0. 
\end{equation}
Indeed, to construct the family $\mathcal{Z}$, consider the
projective space $\mathbb{P}^4$ with coordinates $(u:x:y:z:w)$
and the affine line $\mathbb{A}^1$ with coordinate $t$. In
$\mathbb{P}^4$ consider a reduced subscheme $W$ given by the 
equations $\{xz=xw=yz=yw=0\}$ (This $W$ is just a union of 
two projective planes in $\mathbb{P}^4$ intersecting at the 
point $\tilde{p}=(1:0:0:0:0)$.) Next, in $\mathbb{P}^4\times
\mathbb{A}^1$ take a divizor $D=\{tu=x+y+z+w\}$, and let 
$\tilde{\mathcal{Z}}=D\cap W\times\mathbb{A}^1$. Furthermore, 
in $\mathbb{P}^4$ take a hyperplane $\PP_0=\{x+y+z+w=0\}$, 
fix some isomorphism $f:\ D\xrightarrow{\simeq}\PP_0\times
\mathbb{A}^1\simeq\PP\times\mathbb{A}^1$ and set $p=f(\tilde
{p})$. (For instance, one can take for $f$ a morphism 
$((u:x:y:z:w),t)\mapsto((u:x:y:z:-(x+y+z)),t)$. Then
the subscheme $\mathcal{Z}=f(\tilde{\mathcal{Z}})$ satisfies 
the above properties a) and b).

Let $p_2:\PP\times\mathbb{A}^1\to\mathbb{A}^1$ be the 
projection. One checks that, for $t\in\mathbb{A}^1$, 
$\ext^1(\cali_{Z_t},\op3(-1))$ has fixed dimension 4 while the
higher  Ext-groups of this pair vanish, hence the base change
for relative Ext-sheaves (see, e. g., \cite[Thm. 1.4]{L}) shows
that the sheaf $\mathcal{A}=\lext^1_{p_2}(\cali_{\mathcal{Z},\PP
\times{\mathbb{A}^1}},\mathcal{O}_{\PP}(-1)\boxtimes\calo_
{\mathbb{A}^1})$ is a locally free $\calo_
{\mathbb{A}^1}$-sheaf and there exists a nowhere vanishing 
section $s\in\h^0(\mathcal{A})$. Furthermore, by the spectral
sequence of global-to-relative Ext we may consider $s$ as an
element of the group $\ext^1_{p_2}(\cali_{\mathcal{Z},\PP
\times{\mathbb{A}^1}},\mathcal{O}_{\PP}(-1)\boxtimes\calo_
{\mathbb{A}^1})$. This element defines an extension of $\calo_
{\PP\times\mathbb{A}^1}$-sheaves 
\begin{equation}\label{flat over A1}
0\to\mathcal{O}_{\PP}(-1)\boxtimes\calo_{\mathbb{A}^1}\to
\mathbf{E}\to \cali_{\mathcal{Z},\PP\times{\mathbb{A}^1}}\to0.
\end{equation}
The sheaf $\mathbf{E}$ is flat over $\mathbb{A}^1$and, by 
construction, for $t\in \mathbb{A}^1$, the restriction of
\eqref{flat over A1} is a nonsplitting extension of 
$\op3$-sheaves $0\to\mathcal{O}_{\PP}(-1)\to E_t\to\cali_{Z_t,
\PP}\to0$, where $E_t:=\mathbf{E}|_{\mathbb{P}^3\times\{t\}}$.
Hence, $[E_t]\in\calm(-1,2,2)$, i. e., we obtain a modular
morphism $\Phi:\ \mathbb{A}^1\to\calm(-1,2,2),\ \ t\mapsto
[E_t]$.

Note that, for $t\ne0$, $[E_t]\in\calr(-1,2,2)$ by \cite
[Lemma 2.4]{Chang}, i. e. $\Phi(\mathbb{A}^1\setminus\{0\})
\subset\calr(-1,2,2)$. It follows that $[E_0]\in\overline
{\calr(-1,2,2)}$. Besides, $E_0$ fits the exact sequence
\begin{equation}\label{E0}
0\to\mathcal{O}_{\PP}(-1)\xrightarrow{r}E_0\to\cali_{Z_0,\PP}
\to0.
\end{equation}
From \eqref{pointp} and \eqref{E0} we deduce the following 
exact sequences:
\begin{equation}\label{E0p}
0\to E_0\to E_0^{\vee\vee}\to\mathcal{O}_p\to 0,
\end{equation}
\begin{equation}\label{stability}
0\to \mathcal{O}(-1) \stackrel{s}{\to} E_0^{\vee\vee}
\to \cali_{(Z_0)_{red},\PP}\to 0.
\end{equation}
where $s$ is the composition morphism $r$ in the sequece 
\eqref{E0} with the canonical monomorphism $E_0\to E_0^{\vee
\vee}$. From the sequence \eqref{stability} and \cite
[Proposition 4.2]{Harshorne-Reflexive} we conclude that 
$E_0^{\vee \vee}$ is stable. Moreover, by \eqref{E0p} and 
Theorem \ref{T(-1,2,2i,1)}.(ii), $[E_0]\in\mathrm{T}(-1,2,4,
1)$. This yields the proof since, by Theorem \ref{M(-1,2,2)}, 
$\calm(-1,2,2)=\overline{\calr(-1,2,2)}\cup\mathrm{T}
(-1,2,4,1)$.
\end{proof}

\begin{Teo}\label{connected2}
The moduli space $\calm(-1,2,0)$ is connected.
\end{Teo}
\begin{proof}
By Theorem \ref{M(-1,2,0)}, 
$$
\calm(-1,2,0)=\overline{\calb(-1,2)}\cup\mathrm{T}(-1,2,2,
1)\cup\mathrm{T}(-1,2,4,2)\cup\mathrm{X}(-1,1,1,1,0).
$$
We are going to prove that: (i) the component $\overline{
\calb(-1,2)}$ intersects the components $\mathrm{T}(-1,2,2,
1)$ and $\mathrm{T}(-1,2,4,2)$; (ii) the component $\mathrm{X}
(-1,1,1,1,0)$ intersects the component $\mathrm{T}(-1,2,4,2
)$. This will imply the connectedness of $\calm(-1,2,0)$. 

(i) By \cite[Example 3.1.2]
{Harshorne-Reflexive}, the generic sheaf $[E]\in\calb(-1,2)$ 
fits into an exact triple of the form
\begin{equation}\label{conicextension}
0 \to \mathcal{O}_{\PP}(-2) \to E \to \cali_{Z}(1) \to 0 ~~,
\end{equation}
where $I_Z$ is the ideal sheaf of a disjoint union of two 
conics in $\PP$. The proof is similar to the proof of Theorem 
\ref{connected1}, with minor changes, that we will include 
here for completeness. 

Consider the following two $1$-dimensional flat families of 
curves $\mathcal{Z}^1$ and $\mathcal{Z}^2$ in $\PP$, with 
base $U$ open subset in $\mathbb{A}^1$ containing the point 
$0$, and with projections $\pi_i: \mathcal{Z}^i\hookrightarrow
\PP\times U\xrightarrow{pr_2}U$, $i=1,2,$ such that $\mathcal
{Z}^i$ satisfies the conditions ($a$) and ($b_i$), $i=1,2$, 
where:\\
($a$) for $t\ne0$, the fiber $Z_t^i:=\pi_i^{-1}(t)$ 
is a disjoint union of  two conics;\\
($b_1$) the fiber $Z^1_0$ at $0$ as a set is the union of two 
smooth conics, $C_1$ and $C_2$ meeting in a unique point, say 
$p$, i. e., $(Z_0^1)_{red}=C_1\cup C_2$ and $p=C_1\cap C_2$, 
and as a scheme $Z_0^1$ has an embedded point $p$ such that
the following triple is exact:
\begin{equation}\label{pointpnovo} 
0\to \calo_p\to\calo_{Z_0^1}\to\calo_{(Z_0^1)_{red}}\to0. 
\end{equation}
($b_2$) the fiber $Z^2_0$ at $0$ is a union of two conics, 
$C_1$ and $C_2$ meeting in two distinct fat points of 
multiplicity $2$. That is, $(Z_0^2)_{red} =C_1\cup C_2 $, 
and $\{p_1,p_2\}=C_1\cap C_2$, and as scheme $Z_0^2$ 
has two embedded points $p_1$ and $p_2$:
\begin{equation}\label{2pointp} 
0\to\calo_{p_1}\oplus\calo_{p_2}\to\calo_{Z_0^2}\to
\calo_{(Z_0^2)_{red}}\to0. 
\end{equation}
Now similarly to \eqref{flat over A1} we obtain the exact 
triples:
\begin{equation}\label{flat over U}
0\to\mathcal{O}_{\PP}(-2)\boxtimes\calo_U\to
\mathbf{E}^i\to \cali_{\mathcal{Z}^i,\PP\times U}\otimes
\mathcal{O}_{\PP}(1)\boxtimes\calo_U\to0,\ \ \ \ 
i=1,2,
\end{equation}
such that, for $t\in U$, the restriction of
\eqref{flat over U} is a nonsplitting extension of 
$\op3$-sheaves $0\to\mathcal{O}_{\PP}(-2)\to E^i_t\to 
\cali_{Z^i_t,\PP}(1)\to0$, where 
$E^i_t:=\mathbf{E}^i|_{\mathbb{P}^3\times\{t\}}$.
Hence, $[E^i_t]\in\calm(-1,2,0)$, i. e., we obtain modular
morphisms $\Phi_i:\ U\to\calm(-1,2,0),\ \ t\mapsto[E^i_t]$.
Note that, for $t\ne0$,  each $[E_t^i]\in\calb(-1,2)$ by 
\cite[Example 3.1.2]{H1978}. Hence, also  $[E_0^i]\in\overline
{\calb(-1,2)}$, $i=1,2$. Besides, $E_0^i$ fit in the following 
exact triples:
\begin{equation}\label{E0i}
0 \to \mathcal{O}_{\PP}(-2) \stackrel{r_i}{\to} E_0^i\to 
\cali_{Z_0^i,\PP}(1)\to0,\ \ \ i=1,2.
\end{equation}
The triples \eqref{pointpnovo},\eqref{2pointp} and 
\eqref{E0i}, yield the following exact sequences:
\begin{equation}\label{E0p1}
0\to E_0^1\to (E_0^1)^{\vee\vee}\to\mathcal{O}_p\to 0,
\end{equation}
\begin{equation}\label{E0p2}
0\to E_0^2\to (E_0^2)^{\vee\vee}\to \mathcal{O}_{p_1} 
\oplus \mathcal{O}_{p_2} \to 0,
\end{equation}
\begin{equation}\label{stabilityi}
0\to \mathcal{O}(-2) \stackrel{s_i}{\to} (E_0^i)^{\vee\vee}
\to \cali_{(Z_0^i)_{red},\PP}(1)\to 0,\ \ \ \ \ i=1,2,
\end{equation}
where $s_i$ is the composition morphism $r_i$ from 
\eqref{E0i} with the canonical monomorphism $E_0^i\to(E_0^i)^
{\vee\vee}$. From sequence 
\eqref{stabilityi} and \cite[Proposition 4.2]
{Harshorne-Reflexive} we conclude that $(E_0^i)^{\vee\vee}$ 
is stable, i. e. $[(E_0^i)^{\vee\vee}]\in\calm(-1,2,2i)$, 
$i=1,2$. Thus, \eqref{E0p1} and Theorem \ref{M(-1,2,0)}.(c) 
yield $[E_0^1]\in\mathrm{T}(-1,2,2,1)$; respectively, 
\eqref{E0p2} and Theorem \ref{M(-1,2,0)}.(d) yield $[E_0^2] 
\in\mathrm{T}(-1,2,4,2)$. Since, by the above, $[E_0^1],
[E_0^2]\in\overline{\calb(-1,2)}$, $i=1,2$, it follows, that 
$\overline{\calb(-1,2)}\cap\mathrm{T}(-1,2,2,1)\ne
\emptyset$ and $\overline{\calb(-1,2)}\cap\mathrm{T}
(-1,2,4,2)\ne\emptyset$, as stated. 

(ii) Fix a sheaf $[F]\in\calr^*(-1,1,1)$, a line $l$ in
$\PP$ such that $l\cap\sing(F)=\emptyset$, and two distinct 
points $p_1,p_2\in l$. Consider the surface $S=l\times\mathbb
{A}^1$ with the projection $pr _2:S\to\mathbb{A}^1$. The points 
$p_1,p_2$ define two points $\tilde{p}_i=(p_i,0)\in 
pr_2^{-1}(0),\ i=1,2$. Since $F|_l\cong\calo_l\oplus
\calo_l(-1)$ (see \eqref{splitingF}), it follows that there
exists an epimorphism $\mathbf{e}:F\boxtimes\calo_{\mathbb
{A}^1}\twoheadrightarrow\calb:=\cali_{\tilde{p}_1\sqcup
\tilde{p}_2,S}\otimes\calo_l(1)\boxtimes\calo_{\mathbb{A}^1}$.
Consider an $\calo_{\PP\times\mathbb{A}^1}$-sheaf $\mathbf{E}=
\ker\mathbf{e}$, flat over $\mathbb{A}^1$ and, for $t\in
\mathbb{A}^1$, set $E_t:=\mathbf{E}|_{\PP\times\{t\}}$. By 
construction, the restriction of the exact triple $0\to
\mathbf{E}\to F\boxtimes\calo_{\mathbb{A}^1}\to\calb\to0$ 
onto $\PP\times\{t\}$ yields the exact sequences
\begin{equation}\label{seq 1}
0\to E_t\to F\to\calo_l(1)\to0,\ \ \ \ \ \ t\in\mathbb{A}^1
\smallsetminus\{0\},
\end{equation}
\begin{equation}\label{seq 2}
0\to E_0\to F\to\calo_l(-1)\oplus\calo_{p_1}\oplus\calo_{p_2}
\to0,
\end{equation}
and there is a modular morphism $\Psi:\mathbb{A}^1\to
\calm(-1,2,0),\ t\mapsto[E_t]$. By the definition of 
$\mathcal{X}(-1,1,1,1,0)$ and $\mathrm{X}(-1,1,1,1,0)$ 
(see \eqref{familyX} and \eqref{def calX}), it follows from 
\eqref{seq 1} that $[E_t]\in\mathcal{X}(-1,1,1,1,0)$ for 
$t\in\mathbb{A}^1\smallsetminus\{0\}$, i. e. $\Psi(\mathbb
{A}^1\smallsetminus\{0\})\subset \mathcal{X}(-1,1,1,1,0)
\subset\mathrm{X}(-1,1,1,1,0)$. Hence, $[E_0]\in\mathrm{X}
(-1,1,1,1,0)$. On the other hand, by the definition of 
$\mathcal{X}(-1,1,1,-1,2)$, $[E_0]\in\mathcal{X}(-1,1,1,-1,2)
\subset\mathrm{X}(-1,1,1,-1,2)$. Since by Theorem \ref{X in 
T}.(ii) $\mathrm{X}(-1,1,1,-1,2)$ lies in $\mathrm{T}
(-1,2,4,2)$, it follows that $\mathrm{X}(-1,1,1,1,0)\cap
\mathrm{T}(-1,2,4,2)\ne\emptyset$. 
\end{proof}

The first part of the previous proof can also be regarded as a 
proof of the following claim. Let $E$ be a stable torsion free 
sheaf with $(c_1(E),c_2(E),c_3(E))=(-1,2,0)$ such that
\begin{itemize}
\item[(i)] there exists a nontrivial section in 
$H^0(E^{\vee\vee}(2))$ that vanishes along the union of two 
conics intersecting in a point $p$, and
\item[(ii)] $E^{\vee\vee}/E=\mathcal{O}_p$.
\end{itemize}
Then $E$ is smoothable. Indeed, these two hypotheses imply that 
$E$ fits into the following exact sequence
$$ 0 \to \op3(-2) \to E \to I_Z(1) \to 0, $$
where $Z$ coinciding the scheme $Z^1_0$ described in item 
($b_1$) above. Deforming $Z$ into a union of disjoint conics, 
we obtain a deformation of $E$ into a locally free sheaf $F$ 
with $[F]\in\calb(-1,2)$.

Similarly, if $E$ satisfies the following two hypotheses
\begin{itemize}
\item[(i')] there exists a nontrivial section in 
$H^0(E^{\vee\vee}(2))$ that vanishes along the union of two 
conics intersecting in two points $p$ and $q$;
\item[(ii')] $E^{\vee\vee}/E=\mathcal{O}_p\oplus\mathcal{O}_q$,
\end{itemize}
then $E$ is smoothable. 

\end{document}